\documentclass[11pt]{article}
\usepackage{amssymb}
\usepackage{amsthm}
\usepackage{amsmath}
\usepackage{amsfonts}
\usepackage{graphicx}
\usepackage{hyperref}

\newtheorem{theorem}{Theorem}
\theoremstyle{plain}

\newtheorem{lemma}{Lemma}

\newtheorem{proposition}{Proposition}
\newtheorem{remark}{Remark}
\newtheorem* {AsmK}{Assumption K}
\newtheorem* {AsmZ}{Assumption Z}
\newtheorem* {AsmQ}{Assumption Q}
\newtheorem* {AsmJ}{Assumption J}
\newtheorem* {AsmC}{Assumption C}
\input{tcilatex}
\begin{document}

\title{\textbf{Robust Inference for Threshold Regression Models}\thanks{%
We thank anonymous referees and an Associate Editor for their constructive comments.		
M. Seo 
gratefully acknowledges the support from Promising-Pioneering
Researcher Program through Seoul National University (SNU) and from the
Ministry of Education of the Republic of Korea and the National Research
Foundation of Korea (NRF-0405-20180026).}}
\author{Javier Hidalgo \\
London School of Economics \and Jungyoon Lee \\
Royal Holloway, University London \and Myung Hwan Seo \\
Seoul National University }
\date{}
\maketitle

\begin{abstract}
This paper is concerned with inference in threshold regression models when
the practitioners do not know whether at the threshold point the true
specification has a kink or a jump. We nest previous works that assume
either continuity or discontinuity at the threshold point and develop robust
inference methods on the parameters of the model, which are valid under both
specifications. In particular, we found that the parameter values under the
kink restriction are irregular points of the Hessian matrix of the expected
Gaussian quasi-likelihood. This irregularity destroys the asymptotic
normality and induces the nonstandard cube root convergence rate for the
threshold estimate. However, it also enables us to
obtain the same asymptotic distribution as in Hansen (2000) for the
quasi-likelihood ratio statistic for the unknown threshold
up to an unknown scale parameter. We show that this scale parameter can be
consistently estimated by a kernel method
as long as no higher order kernel is used.
Furthermore, we propose to construct confidence intervals for the unknown
threshold by bootstrap test inversion, also known as grid bootstrap. Finite
sample performances of the grid bootstrap confidence intervals are examined
through Monte Carlo simulations. We also implement our procedure to an
economic empirical application.
\end{abstract}



\thispagestyle{empty}

\newpage

\section{\textbf{INTRODUCTION}\label{sec:Intro}}

This paper examines robust inference in threshold models without a priori
knowledge on whether the model is or not continuous at the threshold point.
Since its introduction, threshold models have gained a lot of attention in
econometrics, statistics and other fields, see Tong (1990) and Hansen $%
\left( 2000\right) $ among others. In the time series context, their
popularity is due to the fact that they are capable to explain nonlinear
features present in many data such as chaos, cycles, irreversibility among
others. In addition they have proved to have superior forecast performance
in times of recession, see Tiao and Tsay $\left( 1994\right) $.

We nest previous works that assume either continuity or discontinuity at the threshold point and develop robust inference methods on the parameters of the model, which are valid under both specifications
When looking at inferences regarding these type of models, the literature
has explicitly assumed that either the threshold regression model is
continuous and kinked or it is discontinuous at the threshold point. For
instance, Chan (1993) and Hansen $\left( 2000\right) $ have focused on
inference when the model is discontinuous at the threshold point, whereas
Chan and Tsay $\left( 1998\right) ,$ Hansen $\left( 2017\right) $ and Feder
(1975a) have focused on inference in kink models. However, there is no a
priori reason to believe that the model is or it is not continuous. The main
motivation to have a \textquotedblleft unified\textquotedblright\ or robust
inference theory for these models is that their statistical properties are
very different whether one estimates the model under the restriction of
continuity or not. In particular, the estimates of the parameters of the
model are all square root $n$-consistent and asymptotically normal when the
model is estimated under the (true) assumption of continuity, but under
discontinuity the least squares estimator of $\gamma $ is super consistent,
asymptotically independent of the slope parameter estimates, and
non-Gaussian. So, it is worthwhile to obtain some statistical properties of
estimates of the parameters in a model that nests continuous and
discontinuous frameworks.

We show an interesting property that the estimator of the threshold
parameter fails to be root-$n$ consistent, contrary to what one might
expect, if the model is continuous but the true restriction is not imposed
in the estimation procedure. More specifically, we show that the rate of
convergence of the estimate of the threshold point becomes $n^{1/3}$ in
contrast to $n^{1/2}$, which was first obtained by Feder $\left(
1975a\right) $ and in the time series context by Chan and Tsay $\left(
1998\right) $ by imposing the (true) constraint of a kink in its estimation.
The asymptotic distribution of the threshold estimator is no longer normal
but the \textquotedblleft $\limfunc{argmax}$\textquotedblright\ of some
Gaussian process. On the other hand, we find that the unconstrained
estimator of the slope parameters is asymptotically independent of the
estimator of the threshold point, contrary to the findings in previous works. 
The asymptotic independence is also the case under the jump
models of Chan $\left( 1993\right) $ or Hansen $(2000)$ but not under the
constrained estimation of Feder's (1975) or Chan and Tsay's (1998) kink
models. This finding is interesting and new, when compared to standard
results in regression models, where it is known that the consequence of not
using the (true) restrictions is inefficiency but otherwise the asymptotic
distribution is still Gaussian and the rate of convergence is the same. So,
we conclude that the statistical inference for threshold regression models
hinges too much on the unverified assumption of kink versus jump.

Our preceding discussion motivates us to develop a robust inference in the
threshold regression model. To that end, we first show that a
quasi-likelihood ratio statistic for the location of the threshold has the
same asymptotic distribution up to a scale constant that depends on whether
the true regression model has a kink or a jump. Second, we present an
estimator for the scale factor based on the ratio of two kernel
Nadaraya-Watson estimators. The consistency of this estimator is standard
under the jump model but non-standard under the kink model because both its
numerator and denominator converge to zero in probability. However, we prove
that, similar to L'Hopital rule, the ratio of the two degenerating terms
still converges in probability to the correct scale factor under the
interesting requirement that higher-order kernels should not be used. Third,
we show that the asymptotic distribution of the unconstrained estimator of
the slope parameters when the model has a kink is identical to the one under
the jump specification, which results from the asymptotic independence
between the estimators of the slope and threshold parameters. This is not
the case if the (correct) kink assumption were employed in the estimation of
the parameters.

The last goal of this paper is to present valid bootstrap schemes for the\
construction of confidence sets for the threshold location. The motivation
comes from the fact that sometimes the asymptotic critical values appear to
be a poor approximation to the finite-sample ones, as documented by Hansen $%
(2000)$ and also in our Section \ref{sec:MC} among others. In addition, the
first-order validity of the bootstrap is of theoretical interest and it has
not been established even under the Hansen's $(2000)$ shrinking jump design.
The interest stems from two sets of findings in the literature regarding the
failure of bootstrap for non-standard estimators: firstly with cube-root
estimators such as the maximum score estimator, and secondly with
super-consistent estimators such as the estimator of autoregressive
coefficients of unit root processes and the threshold estimator under Chan's
(1993) model, see Abrevaya and Huang $(2005)$, Seijo and Sen $(2011)$, and
Yu $(2014)$, just to name a few. Note that the unconstrained estimator of
the threshold belongs to the cube-root class under the kink model and to the
super-consistent class under the jump models. Unlike failures of bootstrap
in the cases listed above, we show that the proposed bootstrap statistics,
which build on the wild bootstrap, correctly approximate the sampling
distribution of the scaled quasi-likelihood ratio statistic in our settings.
This contrast is perhaps due to the fact that the nuisance parameter in the
asymptotic distribution under the non-shrinking model is
infinite-dimensional while the ones in our continuous and shrinking
specifications are finite-dimensional scaling terms. Furthermore, we propose
bootstrap test inversion confidence interval for the threshold, also known
as the grid bootstrap in Hansen $(1999)$, to enhance the finite-sample
coverage probability.

We then present results of a small Monte Carlo experiment, which report good
finite-sample performance of our bootstrap procedure for inference on the
threshold location. In our empirical application, we apply our robust
inferential method to the time series data on real GDP growth and
debt-to-GDP ratio of a number of countries. Numerous works had fitted jump
threshold models to a variety of of datasets, see e.g. Caner, Grennes, and
Koehler-Geib $(2010)$, Cecchetti, Mohanty, and Zampolli $(2011)$, and Lee et
al. $(2017)$, while Hansen $(2017)$ had fitted kink threshold model to the
US time series data. As there is little guidance from economic theory on
suitability of jump or kink models, we advocate the use of our robust
inference, and find substantial heterogeneity across countries in not just
the estimated model parameters but also in the presence and location of
threshold effect.

In Section \ref{sec:model} we introduce the model and present a set of
regularity assumptions and describe how to estimate the parameters of the
model. In particular, we examine the properties of the least squares
estimator of the parameters when the model is continuous but we estimate
them without this knowledge. In Section \ref{sec:Unified} we then develop
robust inferential methods for model parameters that are valid under both
continuous and discontinuous settings, despite the slower rate of
convergence for the estimate of the threshold under the kink specification.
We then present in Section \ref{sec:btrp} a bootstrap algorithm for
inference on the model parameters, establishing their validity. Section \ref%
{sec:MC} presents results of a small Monte Carlo study, followed by Section %
\ref{sec:Application}, which contains the empirical application. Section \ref%
{sec:Conclusion} concludes. This paper has an appendix that contains some of
the proofs and an online supplement that presents the remaining proofs,
technical lemmas, and more numerical results for Sections \ref{sec:MC} and %
\ref{sec:Application}.

\section{\textbf{MODEL AND ESTIMATORS}\label{sec:model}}

We shall consider the following threshold regression model 
\begin{equation}
y_{t}=\beta ^{\prime }x_{t}+\delta ^{\prime }x_{t}\mathbf{1}\left\{
q_{t}>\gamma \right\} +\varepsilon _{t}\text{,}  \label{eq:model}
\end{equation}%
where $\mathbf{1}\left\{ \cdot \right\} $ denotes the indicator function and 
$x_{t}$ is a $k$-dimensional vector of regressors. The parameter $\gamma $
is referred to as a threshold point, taking values in a compact parameter
space $\Gamma $, which is a subset of the interior on the domain of the
threshold variable $q_{t}$. It is worth mentioning that all our results hold
true also when $q_{t}=t$, which is the case with structural break models.
However, we have opted not to include this scenario for the sake of clarity
and notational simplicity.

We assume that $q_{t}$ is an element of the regressor vector $x_{t}$ and
denote 
\begin{equation}
x_{t}=\left( 1,x_{t2}^{\prime },q_{t}\right) ^{\prime };\ \ \ \delta =\left(
\delta _{1},\delta _{2}^{\prime },\delta _{3}\right) ^{\prime }\text{,}
\label{x_not}
\end{equation}%
where $\delta $ is partitioned to match the dimensionality of $x_{t}$. Also
we shall abbreviate $\mathbf{1}_{t}\left( \gamma \right) =\mathbf{1}\left\{
q_{t}>\gamma \right\} $ and $x_{t}\left( \gamma \right) =\left(
x_{t}^{\prime },x_{t}^{\prime }\mathbf{1}_{t}\left( \gamma \right) \right)
^{\prime }$, so that we can write $\left( \ref{eq:model}\right) $ as 
\begin{eqnarray}
y_{t} &=&\beta ^{\prime }x_{t}+\delta _{1}\mathbf{1}_{t}\left( \gamma
\right) +\delta _{2}^{\prime }x_{t2}\mathbf{1}_{t}\left( \gamma \right)
+\delta _{3}q_{t}\mathbf{1}_{t}\left( \gamma \right) +\varepsilon _{t}
\label{eq:model'} \\
&=&\alpha ^{\prime }x_{t}\left( \gamma \right) +\varepsilon _{t}\text{,}%
\quad \text{where}\quad \alpha =(\beta ^{\prime },\delta ^{\prime })^{\prime
}\text{.}  \notag
\end{eqnarray}

Before stating some regularity assumptions on the model, we need to
introduce some extra notation. Let $f\left( \cdot \right) $ denote the
density function of $q_{t}$, which we assume to exist, and $\sigma
^{2}\left( \gamma \right) =E\left( \varepsilon _{t}^{2}\mid q_{t}=\gamma
\right) $, the conditional variance function of error term, while $\sigma
^{2}=E(\varepsilon _{t}^{2})$ denotes the unconditional variance. Denote $%
k\times k$ matrices $D\left( \gamma \right) =E\left( x_{t}x_{t}^{\prime
}|q_{t}=\gamma \right) $, $V\left( \gamma \right) =E\left(
x_{t}x_{t}^{\prime }\varepsilon _{t}^{2}|q_{t}=\gamma \right) $ and let $%
D=D\left( \gamma _{0}\right) $ and $V=V\left( \gamma _{0}\right) $. As usual
the \textquotedblleft $0$\textquotedblright\ subscript on a parameter
indicates its true unknown value. Finally, let $M=E(\mathbf{%
x_{t}x_{t}^{\prime }})$ and $\Omega =E(\mathbf{x_{t}x_{t}^{\prime }}%
\varepsilon _{t}^{2})$ with $\mathbf{x_{t}}=x_{t}\left( \gamma _{0}\right) $.

\begin{AsmZ}
Let $\left\{ x_{t},\varepsilon _{t}\right\} _{t\in \mathbb{Z}}$ be a
strictly stationary, ergodic sequence of random variables such that their $%
\rho $-mixing coefficients satisfy $\sum_{m=1}^{\infty }\rho
_{m}^{1/2}<\infty $ and $E\left( \varepsilon _{t}|{{\mathcal{F}}}%
_{t-1}\right) =0$, where ${{\mathcal{F}}}_{t}$ is the filtration up to time $%
t$. Furthermore, $M,\Omega >0$, $E\left\Vert x_{t}\right\Vert ^{4}<\infty $, 
$E\left\Vert x_{t}\varepsilon _{t}\right\Vert ^{4}<\infty $ and $E\left\vert
\varepsilon _{t}\right\vert ^{4+\eta }<\infty $ for some $\eta >0$.
\end{AsmZ}

\begin{AsmQ}
The functions $f\left( \gamma \right) $, $V\left( \gamma \right) \ $and $%
D\left( \gamma \right) $ are continuous at $\gamma =\gamma _{0}$. For all $%
\gamma \in \Gamma $, the functions $f\left( \gamma \right) $, $E\big(%
x_{t}x_{t}^{\prime }\mathbf{1}\left\{ q_{t}\leq \gamma \right\} \big)$ and $%
E\left( x_{t2}x_{t2}^{\prime }|q_{t}=\gamma \right) $ are positive and
continuous, and the functions $f\left( \gamma \right) $,$\ E\big(%
|x_{t}|^{4}|q_{t}=\gamma \big)$ and $E\big(|x_{t}\varepsilon
_{t}|^{4}|q_{t}=\gamma \big)$ are bounded by some $C<\infty $.
\end{AsmQ}

Assumptions \textbf{Z }and \textbf{Q} are commonly imposed on the
distribution of $\left\{ x_{t},\varepsilon _{t}\right\} $, see e.g. Hansen $%
\left( 2000\right) $, so his comments apply here. As discussed therein, the
self-exciting threshold autoregressive model of Tong $\left( 1990\right) $
satisfies Assumption \textbf{Z.} The condition for $E\left(
x_{t2}x_{t2}^{\prime }|q_{t}=\gamma \right) $ is written in terms of $x_{t2}$
as the other elements in $x_{t}$ are fixed given $q_{t}=\gamma $. While we
allow conditional heteroscedasticity of a general form, Assumption \textbf{Q}
requires continuity of the conditional variance function $\sigma ^{2}(\cdot
) $ at $\gamma _{0}$.

\subsection{Estimators\label{sec:LSE}}

We estimate $\theta _{0}=\left( \alpha _{0}^{\prime },\gamma _{0}\right)
^{\prime }$ by the (non-linear) least squares estimator (LSE), that is, 
\begin{equation}
\widehat{\theta }=\left( \widehat{\alpha }^{\prime },\widehat{\gamma }%
\right) ^{\prime }:=\underset{\theta \in \Theta }{\func{argmin}}\,\mathbb{S}%
_{n}\left( \theta \right) \text{,}  \label{theta_hat}
\end{equation}%
where $\Theta =\left( \Lambda ,\Gamma \right) $ is a compact set in ${{\ 
\mathbb{R}}}^{2k+1}$ and 
\begin{equation}
{{\mathbb{S}}}_{n}\left( \theta \right) :=\frac{1}{n}\sum_{t=1}^{n}\left(
y_{t}-\alpha ^{\prime }x_{t}\left( \gamma \right) \right) ^{2}\text{,}
\label{s_theta}
\end{equation}%
which is a step function in $\gamma $ at $q_{t}$'s. For its computation, we
shall employ a step-wise algorithm. To that end, one could employ the grid
search algorithm on $\Gamma _{n}=\Gamma \cap \left\{ q_{1},...,q_{n}\right\} 
$ to find $\widehat{\gamma }$. Define the concentrated sum of squared
residuals 
\begin{equation}
\widehat{\mathbb{S}}_{n}\left( \gamma \right) :=\frac{1}{n}%
\sum_{t=1}^{n}\left( y_{t}-\widehat{\alpha }^{\prime }\left( \gamma \right)
x_{t}\left( \gamma \right) \right) ^{2}\text{,}  \label{ssngm}
\end{equation}%
where 
\begin{equation}
\widehat{\alpha }\left( \gamma \right) :=\underset{\alpha \in \Lambda }{%
\func{argmin}}\text{~}\frac{1}{n}\sum_{t=1}^{n}\left( y_{t}-\alpha ^{\prime
}x_{t}\left( \gamma \right) \right) ^{2}  \label{s_alpha}
\end{equation}%
is the LSE of $\alpha $ for a given $\gamma $. Then, our estimator of $%
\alpha $ is $\widehat{\alpha }:=\widehat{\alpha }\left( \widehat{\gamma }%
\right) $, with 
\begin{equation}
\widehat{\gamma }:=\underset{\gamma \in \Gamma _{n}}{\func{argmin}}\,%
\widehat{\mathbb{S}}_{n}\left( \gamma \right) \text{.}  \label{s_gamma}
\end{equation}%
Since the minimizer is given by an interval, it is common to let the
estimator be the maximum. This is the unconstrained LSE and for comparison
we also describe the continuity constrained least squares estimator (CLSE), which
minimizes $\left( \ref{s_theta}\right) $ under Assumption \textbf{C} in the next section,
\begin{equation}
\widetilde{\theta }=\left( \widetilde{\alpha }^{\prime },\widetilde{\gamma }%
\right) ^{\prime }:=\underset{\theta \in \Theta :\delta _{1}+\delta
_{3}\gamma =0;\delta _{2}=0}{\func{argmin}}\mathbb{S}_{n}\left( \theta
\right) \text{.}  \label{theta_tilde}
\end{equation}%
This estimator was considered by Feder $\left( 1975a\right) $ and later by Chan
and Tsay $(1998)$ or Hansen $(2017),$ who have established the asymptotic
normality of $\widetilde{\theta }$ with the standard squared root
consistency.

\section{\textbf{Robust Confidence Regions}\label{sec:Unified}}

This section presents our main results, namely how to perform robust
inference in threshold models and in particular on the location of the
threshold point. We begin with developing inference methods for the
regression coefficients $\alpha _{0}$ and the unknown threshold $\gamma _{0}$
based on the LSE $\widehat{\theta }$ when the true regression model has a
kink. Then, they are compared with other inference methods that are
developed under different sampling schemes such as Hansen $\left(
2000\right) $. In particular, we show that a judicious choice of statistics
enables us to perform a robust inference in the sense that the same critical
values can be employed for inference whether the model has a kink or a jump.
That is, we do not need to know whether the model has a kink or a jump to
make inference for the parameters $\alpha _{0}$ and $\gamma _{0}$. As
mentioned in the introduction the motivation comes from the rather
surprising results given in Proposition \ref{Consistency} and Theorem \ref%
{Th:AD_cube} below.

First we state the kink model in terms of assumption.

\begin{AsmC}
Assume that $\delta _{30}\neq 0$ and 
\begin{equation}
\delta _{10}+\delta _{30}\gamma _{0}=0;\ \ \ \delta _{20}=0\text{.}
\label{eq:conti}
\end{equation}
\end{AsmC}

Under Assumption \textit{\textbf{C}} the model (\ref{eq:model'}) is written
as%
\begin{equation}
y_{t}=x_{t}^{\prime }\beta _{0}+\delta _{30}(q_{t}-\gamma _{0})\mathbf{1}%
_{t}\left( \gamma _{0}\right) +\varepsilon _{t}\text{.}  \label{5}
\end{equation}%
Feder (1975), Chan and Tsay (1998), and Hansen (2017) considered the
estimation of the model (\ref{5}) along with an auxiliary condition of $%
\delta _{30}\neq 0$ to ensure the identification of the change-point $\gamma
_{0}$. This is a model with a kink.

Then, the next proposition establishes the consistency and rates of
convergence of the LSE $\widehat{\theta }$ defined in $\left( \ref{theta_hat}%
\right) $ under Assumption \textit{\textbf{C.}}

\begin{proposition}
\label{Consistency}\textit{Under Assumptions \textbf{C,} \textit{\textbf{Z}}
and \textbf{Q}, we have that} 
\begin{equation*}
\widehat{\alpha }-\alpha _{0}=O_{p}\big(n^{-1/2}\big)\ \ \ \text{and}\ \ \ \ 
\widehat{\gamma }-\gamma _{0}=O_{p}\big(n^{-1/3}\big)\text{.}
\end{equation*}
\end{proposition}

The results of Proposition \ref{Consistency} are surprising because the
convergence rate of $\widehat{\gamma }$ is slower than that of the CLSE $%
\widetilde{\gamma }$, which is known to be $n^{-1/2}$ as shown in the
aforementioned works. That is, using the true restriction on the parameters
leads to a faster rate of convergence of the estimator of $\gamma _{0}$, not
just reducing its asymptotic variance as is often the case.

Next we present the asymptotic distribution of $\widehat{\theta }$.

\begin{theorem}
\label{Th:AD_cube} Let Assumptions \textit{\textbf{C,} \textit{\textbf{Z}}
and \textbf{Q}} hold and $B_{1}\left( \cdot \right) $ and $B_{2}\left( \cdot
\right) $ be two independent standard Brownian motions. Define $W\left(
g\right) :=B_{1}\left( -g\right) \mathbf{1}\left\{ g<0\right\} +B_{2}\left(
g\right) \mathbf{1}\left\{ g>0\right\} $. Then, 
\begin{eqnarray*}
&&n^{1/2}(\widehat{\alpha }-\alpha _{0})\overset{d}{\longrightarrow }%
\mathcal{N}\left( 0,M^{-1}\Omega M^{-1}\right) \\
&&n^{1/3}(\widehat{\gamma }-\gamma _{0})\overset{d}{\longrightarrow }%
\underset{g\in \mathbb{R}}{\func{argmax}}\big(2\delta _{30}\sqrt{\frac{%
\sigma ^{2}\left( \gamma _{0}\right) f\left( \gamma _{0}\right) }{3}}W\left(
g^{3}\right) +\frac{\delta _{30}^{2}}{3}f\left( \gamma _{0}\right)
\left\vert g\right\vert ^{3}\big)\text{,}
\end{eqnarray*}%
where the two limit distributions are independent of each other.
\end{theorem}

The asymptotic independence is a consequence of the different convergence
rates between the two sets of estimators $\widehat{\alpha }$ and $\widehat{%
\gamma }$ by similar arguments as in Chan $\left( 1993\right) $, albeit the
rate for $\widehat{\gamma }$ being slower than that for $\widehat{\alpha }$
in our case. The asymptotic independence does not hold for the CLSE $%
\widetilde{\gamma }$ and $\widetilde{\alpha }$, which converge at the same
rate as mentioned above and they are jointly asymptotically normal with a
non-diagonal variance covariance matrix.

Theorem \ref{Th:AD_cube} suggests that Gonzalo and Wolf's $(2005)$
subsampling procedure would be correct if they had used the normalization $%
n^{1/3}$ instead of the incorrect one $n^{1/2}$. On the other hand, it is
worth mentioning that Seo and Linton $\left( 2007\right) $ considered the
smoothed least squares estimator for the same setup. The convergence rate
for their smoothed least squares estimator for $\gamma $ was slower than our
cube-root rate under their assumptions for the smoothing parameter.

\begin{remark}
\label{Rem:Heuristic}We now present a heuristic discussion to illustrate why
the constrained and unconstrained estimators of $\gamma _{0}$ have different
rates of convergence and the unconstrained estimator belongs to the
cube-root class explored by Kim and Pollard $\left( 1990\right) $ for the $%
i.i.d.$ data and Seo and Otsu $\left( 2018\right) $ for more general setups.
For simplicity of illustration, we begin with a simplified model, where $%
x_{t}=\left( 1,q_{t}\right) ^{\prime }$, $\delta =\left( \delta _{1},\delta
_{3}\right) ^{\prime }$, $\beta $ is fixed at $\beta _{0}=0$, and thus $%
\theta =\left( \delta ^{\prime },\gamma \right) ^{\prime }$. In addition we
shall assume $\gamma _{0}=0$ and thus $\delta _{10}=0$ by (\ref{eq:conti})
without loss of generality since we can always rename the variable $%
q_{t}-\gamma _{0}$ as $q_{t}$. It is well known that the rates of
convergence of an M-estimator is governed by the local behavior of its
criterion function around the true value provided that the estimator is
consistent. Then the convergence rate of LSE $\widehat{\theta }=\left( 
\widehat{\delta }^{\prime },\widehat{\gamma }\right) ^{\prime }$ is
determined by the stochastic expansion of 
\begin{eqnarray}
&&{{{\mathbb{S}}}}_{n}(\theta )-{{{\mathbb{S}}}}_{n}(\theta _{0})
\label{s_1} \\
&=&\frac{1}{n}\sum_{t=1}^{n}\left( \delta _{30}q_{t}\mathbf{1}_{t}\left(
0\right) -\left( \delta _{1}+\delta _{3}q_{t}\right) \mathbf{1}_{t}\left(
\gamma \right) \right) ^{2}+\frac{2}{n}\sum_{t=1}^{n}\varepsilon _{t}\left(
\delta _{30}q_{t}\mathbf{1}_{t}\left( 0\right) -\left( \delta _{1}+\delta
_{3}q_{t}\right) \mathbf{\ 1}_{t}\left( \gamma \right) \right) \text{,} 
\notag
\end{eqnarray}%
in small neighborhoods of $\delta =\delta _{0}$ and $\gamma =\gamma _{0}=0$%
. Consider $\gamma >0$. The case of $\gamma <0$ is handled similarly. Then,
as $\mathbf{1}_{t}\left( 0\right) =\mathbf{1}_{t}\left( \gamma \right) +%
\mathbf{1}\left\{ 0<q_{t}\leq \gamma \right\} $ and $\mathbf{1}_{t}\left(
\gamma \right) \mathbf{1}\left\{ 0<q_{t}\leq \gamma \right\} =0$,%
\begin{eqnarray*}
&&E\left( \delta _{30}q_{t}\mathbf{1}_{t}\left( 0\right) -\left( \delta
_{1}+\delta _{3}q_{t}\right) \mathbf{1}_{t}\left( \gamma \right) \right) ^{2}
\\
&=&E\left( \delta _{1}+\left( \delta _{3}-\delta _{30}\right) q_{t}\right)
^{2}\mathbf{1}_{t}\left( \gamma \right) +E\left( \delta _{30}q_{t}\right)
^{2}\mathbf{1}\left\{ 0<q_{t}\leq \gamma \right\} \\
&\sim &\left\Vert \delta -\delta _{0}\right\Vert ^{2}+\gamma ^{3},
\end{eqnarray*}%
because for some positive constant $c$, 
\begin{equation*}
E\left[ q_{t}^{2}\mathbf{1}\left\{ 0<q_{t}\leq \gamma \right\} \right]
=\int_{0}^{\gamma }q^{2}f\left( q\right) dq\sim \frac{c}{3}\left\vert \gamma
\right\vert ^{3}
\end{equation*}%
due to Assumption \textbf{Q}. This cubic approximation at $\gamma =\gamma
_{0}$ is non-standard and invalidates the asymptotic normality of $\widehat{%
\gamma }$, which builds on the quadratic approximation.\footnote{%
This also shows that the asymptotic variance formula $U^{-1}VU^{-1}$ in
Gonzalo and Wolf's (2005) Theorem A.1 and Remark A.1 is not properly defined
due to the degeneracy of $U$, where $U$ is the second derivative matrix of
the expected criterion function that is evaluated under the continuity
restriction.} Similarly, 
\begin{equation*}
\func{var}\left( \frac{1}{n}\sum_{t=1}^{n}\varepsilon _{t}\left( \delta
_{30}q_{t}\mathbf{1}_{t}\left( 0\right) -\left( \delta _{1}+\delta
_{3}q_{t}\right) \mathbf{1}_{t}\left( \gamma \right) \right) \right) \sim 
\frac{\left\Vert \delta -\delta _{0}\right\Vert ^{2}+\left\vert \gamma
\right\vert ^{3}}{n}\text{.}
\end{equation*}%
Thus, the last two displayed expressions suggest that 
\begin{equation*}
\widehat{\delta }-\delta _{0}=O_{p}\left( n^{-1/2}\right) \ \ \ \text{and}\
\ \widehat{\gamma }=O_{p}\left( n^{-1/3}\right) \text{,}
\end{equation*}%
as these rates of convergence balance the speeds at which the bias and
standard deviation of ${{{\mathbb{S}}}}_{n}\left( \theta \right) -{{{\mathbb{%
S}}}}_{n}\left( \theta _{0}\right) $ converge to zero. In comparison, the
CLSE $\left( \widetilde{\delta }_{3},\widetilde{\gamma }\right) ^{\prime }$
is ruled by 
\begin{eqnarray*}
&&\mathbb{S}_{n}(\theta )-\mathbb{S}_{n}(\theta _{0}) \\
&=&\frac{1}{n}\sum_{t=1}^{n}\left( \delta _{30}q_{t}\mathbf{1}_{t}\left(
0\right) -\delta _{3}\left( q_{t}-\gamma \right) \mathbf{1}_{t}\left( \gamma
\right) \right) ^{2}+\frac{2}{n}\sum_{t=1}^{n}\varepsilon _{t}\left( \delta
_{30}q_{t}\mathbf{1}_{t}\left( 0\right) -\delta _{3}\left( q_{t}-\gamma
\right) \mathbf{1}_{t}\left( \gamma \right) \right) \text{,}
\end{eqnarray*}%
due to the continuity constraint (\ref{eq:conti}), for which we observe the
quadratic expansion 
\begin{eqnarray*}
E\left( \delta _{30}q_{t}\mathbf{1}_{t}\left( 0\right) -\delta _{3}\left(
q_{t}-\gamma \right) \mathbf{1}_{t}\left( \gamma \right) \right) ^{2} &\sim
&\left\vert \delta _{3}-\delta _{30}\right\vert ^{2}+\gamma ^{2} \\
\func{var}\left( \frac{2}{n}\sum_{t=1}^{n}\varepsilon _{t}\left( \delta
_{30}q_{t}\mathbf{1}_{t}\left( 0\right) -\delta _{3}\left( q_{t}-\gamma
\right) \mathbf{1}_{t}\left( \gamma \right) \right) \right) &\sim &\frac{%
\left\vert \delta _{3}-\delta _{30}\right\vert ^{2}+\left\vert \gamma
\right\vert ^{2}}{n}\text{.}
\end{eqnarray*}%
This yields that 
\begin{equation*}
\widetilde{\delta }_{3}-\delta _{30}=O_{p}\left( n^{-1/2}\right) \ \ \ \text{
and\ }\ \ \widetilde{\gamma }=O_{p}\left( n^{-1/2}\right) \text{,}
\end{equation*}%
which coincides with the rates of convergence that both Feder $(1975$a, b)
and Chan and Tsay $(1998)$ obtained.
\end{remark}

An intuitive explanation for the preceding Proposition, Theorem, and Remark
is to appeal to \textquotedblleft misspecification\textquotedblright .
Although the unconstrained model (\ref{eq:model}) encompasses both
continuous and discontinuous models, the estimated regression function is
almost surely discontinuous, since the probability that the LSE $\widehat{%
\theta}$ fulfills the continuity restriction is zero.

%

\subsection{Inference on Regression Coefficient $\protect\alpha $}

Theorem \ref{Th:AD_cube} in Section 3.1, Lemma A.12 of Hansen $\left(
2000\right) $ and Theorem 2 of Chan $\left( 1993\right) $ report the same
asymptotic distribution for $\widehat{\alpha },$ namely $\mathcal{N}\left(
0,M^{-1}\Omega M^{-1}\right) $, which is asymptotically independent of $%
\widehat{\gamma }$. Thus, the inference for $\alpha _{0}$ is uniform under
any widely used sampling scheme with strongly identified $\gamma _{0}$,
provided that the respective sample moments 
\begin{equation*}
\widehat{M}=\frac{1}{n}\sum_{t=1}^{n}x_{t}\left( \widehat{\gamma }\right)
x_{t}\left( \widehat{\gamma }\right) ^{\prime }\text{; \ }\ \ \widehat{%
\Omega }=\frac{1}{n}\sum_{t=1}^{n}x_{t}\left( \widehat{\gamma }\right)
x_{t}\left( \widehat{\gamma }\right) ^{\prime }\widehat{\varepsilon }%
_{t}^{2},
\end{equation*}%
where $\widehat{\varepsilon }_{t}=y_{t}-x_{t}\left( \widehat{\gamma }\right)
^{\prime }\widehat{\alpha }$, are consistent under each data generating
process. This is the case due to the uniform law of large numbers, which
only requires consistency of $\widehat{\gamma }$.

It is worthwhile to mention that this \textquotedblleft
oracle\textquotedblright\ property of $\widehat{\alpha }$ does not hold true
for the CLSE $\widetilde{\alpha }$, whose asymptotic distribution is
affected by that of $\widetilde{\gamma }$, as was first noticed and shown by
Feder $(1975a)$ and later extended to time series data by Chan and Tsay $%
(1998)$.

\subsection{\textbf{Inference on Threshold }$\protect\gamma $}

\label{sec:UI for thr}

The main purpose of this section is to develop a method to construct
confidence regions for $\gamma_{0} $ that is valid regardless of whether the
regression model has a kink or a jump at the true value of $\gamma_{0} $.
Conventionally, inference on $\gamma $ has been done after assuming either
that the model has a kink or that it has a jump, i.e. the practitioner
chooses between jump or kink models before estimating the threshold point.
More specifically, if we decide that the model has a jump, then one follows
e.g. Hansen $\left( 2000\right) $, whereas if one has chosen the kink model
then one needs to employ the asymptotic normal inference as in Feder $\left(
1975a\right) $ and others. One of our findings is that Hansen $\left(
2000\right) $ results are not valid if the model had a kink and likewise
Feder's results are not valid if the model had a jump.

Thus, this section develops robust confidence regions that are valid
regardless which of the two models is the true specification. To ease
reference, we recall Hansen's (2000) diminishing jump specification:

\begin{AsmJ}
For some $0<\varphi <1/2$ and $d\neq 0$, $\delta _{0}=d\cdot n^{-\varphi }$
and $d^{\prime }Vd\ $and $d^{\prime }Dd$ are positive for all $n$.
\end{AsmJ}

When $\varphi $ is greater than or equal to $1/2$, $\delta _{0}$ is too
small to consistently estimate $\gamma _{0}$, and such case is excluded. And
we suppress the dependence of $\delta _{0}$ on the sample size $n$ to
simplify the notation.

To develop robust confidence sets, we need to find a statistic whose
asymptotic distribution is invariant to the true parameter value, that is, a
statistic whose asymptotic distribution does not change suddenly under
Assumption \textbf{C}. We begin by introducing a Gaussian quasi-likelihood
ratio statistic based on the unconstrained model $\left( \ref{eq:model}%
\right) $. Specifically, let%
\begin{equation*}
QLR_{n}=n\frac{\widehat{\mathbb{S}}_{n}\left( \gamma _{0}\right) -\widehat{%
\mathbb{S}}_{n}\left( \widehat{\gamma }\right) }{\widehat{\mathbb{S}}%
_{n}\left( \widehat{\gamma }\right) }\text{,}
\end{equation*}%
where $\widehat{\mathbb{S}}_{n}\left( \gamma \right) $ is defined in $\left( %
\ref{ssngm}\right) $.

We now derive the following asymptotic distribution for $ QLR_n $,
which contrasts with the asymptotic distribution obtained by Hansen $(2000)$
under Assumption \textbf{J}. 

\begin{proposition}
\label{Th:QLR} Suppose that Assumptions \textbf{C,} \textbf{Z} and \textbf{Q}
hold. Then, as $n\rightarrow \infty $, 
\begin{equation*}
QLR_{n}\overset{d}{\longrightarrow }\zeta \max_{g\in \mathbb{R}}\left(
2W\left( g\right) -\left\vert g\right\vert \right) \text{,}
\end{equation*}%
where 
\begin{equation*}
\zeta =\frac{\sigma ^{2}\left( \gamma _{0}\right) }{\sigma ^{2}}\text{.}
\end{equation*}
\end{proposition}

In comparison, we recall Hansen's (2000) results that
\begin{equation}
QLR_{n}\overset{d}{\longrightarrow }\xi \max_{g\in \mathbb{R}}\left(
2W\left( g\right) -\left\vert g\right\vert \right) \text{,}  \label{hansen_1}
\end{equation}%
where 
\begin{equation*}
\xi =\frac{E\big(\left( x_{t}^{\prime }d\varepsilon _{t}\right)
	^{2}|q_{t}=\gamma _{0}\big)}{\sigma ^{2}E\big(\left( x_{t}^{\prime }d\right)
	^{2}|q_{t}=\gamma _{0}\big)},
\end{equation*}%
and that the distribution function of $\max_{g\in \mathbb{R}}\left(
2W\left( g\right) -\left\vert g\right\vert \right) $ is given by $F\left(
z\right) =\left( 1-e^{-z/2}\right) ^{2}$.

The results of our Proposition \ref{Th:QLR} and that in $\left( \ref{hansen_1}%
\right) $ indicate that the only difference between the limit distributions
of $QLR_n$ under the kink and jump specifications is the scaling factor.
This is the case despite the fact the estimator $\widehat{\gamma }$ exhibits
different rates of convergence across the two settings.

Next, we propose an estimator of the unknown scaling of $QLR_{n}$ that
converges in probability to $\xi $ under Assumption \textbf{J}, while it
converges to $\zeta $ under Assumption \textbf{C}, thus adapting to the
unknown true scaling in each situation. We begin with a natural estimator of 
$\xi $, which is a ratio of two Nadaraya-Watson estimators of the
conditional expectations. That is, 
\begin{equation}
\widehat{\xi }=\frac{\frac{1}{n}\sum_{t=1}^{n}\big(\widehat{\delta }^{\prime
}x_{t}\big)^{2}\widehat{\varepsilon }_{t}^{2}K\left( \frac{q_{t}-\widehat{%
\gamma }}{a}\right) }{\mathbb{S}_{n}\big(\widehat{\theta }\big)\frac{1}{n}%
\sum_{t=1}^{n}\big(\widehat{\delta }^{\prime }x_{t}\big)^{2}K\left( \frac{%
q_{t}-\widehat{\gamma }}{a}\right) }\text{,}  \label{xhihat}
\end{equation}%
where $K\left( \cdot \right) $ and $a$ are, respectively, the kernel
function and bandwidth parameter and $\widehat{\varepsilon }_{t}$'s are the
least squares residuals. The consistency of $\widehat{\xi}$ to $\xi $ is
standard, as argued in Hansen $(2000)$.

However, it is not trivial to establish that $\widehat{\xi }\overset{p}{%
\longrightarrow }\zeta $ when the true model has a kink at $\gamma_{0} $
because both numerator and denominator degenerates asymptotically in
Assumption \textbf{C}. It turns out that we need to impose some
unconventional restrictions on the kernel function $K$ and the bandwidth $a$%
. Specifically, we assume

\begin{AsmK}
Assume the following for $K\left( \cdot \right) $ and $a.$
\end{AsmK}

\begin{description}
\item[$\mathbf{K1}$] $K\left( \cdot \right) $ is symmetric and $\kappa
_{\ell }=\int_{-\infty }^{\infty }u^{\ell }K\left( u\right) du<C$ for $\ell
\leq 4$ and $\kappa _{2}\neq 0$.

\item[$\mathbf{K2}$] $K\left( \cdot \right) $ is twice continuously
differentiable with the first derivative $K^{\prime }\left( \cdot \right) $
and for all $u\ $such that $\left\vert w/u\right\vert \leq C$ as $%
w\rightarrow 0$ $K^{\prime }\left( u+w\right) /K^{\prime }\left( u\right)
\rightarrow 1$.

\item[$\mathbf{K3}$] $K\left( u\right) =\int \phi \left( v\right) e^{ivu}dv$
, where the characteristic function $\phi \left( v\right) $ satisfies that $%
v\phi \left( v\right) $ is integrable.

\item[$\mathbf{K4}$] $a^{-3}n^{-1}+a\rightarrow 0$ as $n\rightarrow \infty $.
\end{description}

It is clear that the Epanechnikov and the Gaussian kernel functions satisfy $%
\mathbf{K1}$, $\mathbf{K2}$ and $\mathbf{K3}$. One important observation is
that $\mathbf{K1}$ rules out higher-order kernels by assuming $\kappa
_{2}\neq 0$. The consequence of dropping the assumption that $\kappa
_{2}\not=0$ is discussed in detail in Remark \ref{Rem_k2} that follows the
next proposition.

\begin{proposition}
\label{Prop:xihat}Suppose Assumptions \textbf{Z}, \textbf{Q} and \textbf{K}
hold true. Then, under Assumption \textbf{C}%
\begin{equation*}
\widehat{\xi }\overset{P}{\rightarrow }\zeta \text{,}
\end{equation*}%
while $\widehat{\xi }\overset{P}{\rightarrow }\xi $ under Assumption \textbf{%
J}.
\end{proposition}

\begin{remark}
\label{Rem_k2} We now comment on the consequence of dropping the assumption
that $\kappa _{2}\not=0$. If we allowed for higher-order kernels, that is $%
\kappa _{2}=0$ and $\kappa _{3}=0$ but $\kappa _{4}\neq 0$, $\widehat{\xi }$
would not be consistent. Indeed, Proposition \ref{Prop:xihat} and Lemma \ref%
{lem:4proposition}\ in the Appendix indicate that, without loss of
generality for $\gamma _{0}=0$ and $\sigma ^{2}=1$, $\widehat{\xi }$
converges in probability to 
\begin{equation*}
\frac{\frac{\partial ^{2}}{\partial q^{2}}f\left( q\right) g_{0}\left(
q\right) \mid _{q=0}}{\frac{\partial ^{2}}{\partial q^{2}}f\left( q\right)
g_{0}^{\ast }\left( q\right) \mid _{q=0}}\text{,}
\end{equation*}%
where $g_{r}\left( q\right) =E\left( x_{t2}^{r}\varepsilon _{t}^{2}\mid
q_{t}=q\right) \ $and $g_{r}^{\ast }\left( q\right) =E\left( x_{t2}^{r}\mid
q_{t}=q\right) $. This is the case because dropping in $\mathbf{K1}$ the
assumption of $\kappa _{2}\not=0$ and letting $\kappa _{2}=\kappa _{3}=0$,
the numerator in $\left( \ref{xhihat}\right) $ will be 
\begin{equation*}
\kappa _{4}\delta _{3}^{2}a^{4}\frac{\partial ^{2}}{\partial q^{2}}\left(
f\left( 0\right) g_{0}\left( 0\right) \right) \left( 1+o_{p}\left( 1\right)
\right) \text{,}
\end{equation*}%
whereas the denominator in $\left( \ref{xhihat}\right) $ becomes 
\begin{equation*}
\kappa _{4}\delta _{3}^{2}a^{4}\frac{\partial ^{2}}{\partial q^{2}}\left(
f\left( 0\right) g_{0}^{\ast }\left( 0\right) \right) \left( 1+o_{p}\left(
1\right) \right) \text{.}
\end{equation*}%
So that, unless $E(\varepsilon _{t}^{2}\mid q_{t}=\gamma _{0})=E(\varepsilon
_{t}^{2})$, we obtain that (similar to the L'Hopital rule): 
\begin{equation*}
\widehat{\xi }\overset{P}{\rightarrow }\frac{\frac{\partial ^{2}}{\partial
q^{2}}f\left( q\right) g_{0}\left( q\right) \mid _{q=0}}{\frac{\partial ^{2}%
}{\partial q^{2}}f\left( q\right) g_{0}^{\ast }\left( q\right) \mid _{q=0}}=%
\frac{\frac{\partial ^{2}}{\partial q^{2}}\left( f\left( q\right) E\left[
\varepsilon _{t}^{2}\mid q_{t}=q\right] \right) \mid _{q=0}}{\frac{\partial
^{2}}{\partial q^{2}}f\left( q\right) \mid _{q=0}}\neq \zeta \text{,}
\end{equation*}%
and hence $\widehat{\xi }$ would not be a consistent estimator of the scale
factor $\zeta $.
\end{remark}

We can construct the $100s$ percent confidence set of $\gamma _{0}$ by 
\begin{equation*}
\widehat{\Gamma }_{s}=\left\{ \gamma \in \Gamma :\widehat{\xi }%
^{-1}QLR_{n}\left( \gamma \right) \leq F^{-1}\left( s\right) \right\} \text{.%
} 
\end{equation*}%
As we have already argued, this confidence set is valid under both
scenarios, as the next theorem shows.

\begin{theorem}
\label{Th:Gamma_s} Let Assumption \textbf{K}, \textbf{Z} and \textbf{Q} hold
true and suppose that either Assumption \textbf{C} or \textbf{J} hold. Then,
for any $s\in \left( 0,1\right) $, 
\begin{equation*}
P\{\gamma _{0}\in \widehat{\Gamma }_{s}\}\rightarrow s\text{.}
\end{equation*}
\end{theorem}

\section{\textbf{BOOTSTRAP}\label{sec:btrp}}

This section develops a bootstrap-based test inversion confidence interval
for the unknown threshold parameter $\gamma _{0}$, which is valid under 
Assumption \textbf{C} as well as under Assumption \textbf{J}. We do not discuss the bootstrap for $%
\alpha _{0}$ in detail but note that the bootstrap for the linear regression
can be employed,\footnote{%
This excludes the case where $\gamma _{0}$ is not strongly identified in the
sense that $\delta _{0}=d\cdot n^{-\varphi }$ with $\varphi \geq 1/2$. This
case has not been explored except when $d=0$, see e.g. Hansen (1996) and it
is an interesting future research area.} see e.g. Shao and Tu $\left(
1995\right) $, since we can treat $\widehat{\gamma }$ as $\gamma _{0}$ for
the inference on $\alpha _{0}$ due to the arguments leading to the
asymptotic independence between $\widehat{\alpha }$ and $\widehat{\gamma }$.

We propose using the bootstrap test inversion method, also known as the grid
bootstrap, of D\"{u}mbgen $\left( 1991\right) $ to build confidence
intervals for the parameter $\gamma $, see also Carpenter $\left(
1999\right) $ and Hansen $\left( 1999\right) $. Such a test inversion
bootstrap confidence interval (BCI) is known to have certain optimality
properties as in e.g. Brown, Casella and Hwang $\left( 1995\right) $ from
the Bayesian perspective. Mikusheva $\left( 2007\right) $ showed that test
inversion BCI attains correct coverage probability uniformly over the
parameter space for the sum of coefficients in autoregressive models,
despite the behavior of the estimator not being uniform over the parameter
space.

For a given confidence level $s$, one can exploit the duality between
hypothesis testing and confidence interval by inverting tests to obtain a
confidence region 
\begin{equation*}
\widehat{\Gamma}_{s}^{\ast }=\left\{ \gamma \in \Gamma :\widehat{\xi }\left(
\gamma \right) ^{-1}QLR_{n}\left( \gamma \right) \leq F_{n}^{\ast }\left(
s|\gamma \right) \right\} \text{ ,}
\end{equation*}%
where $F_{n}^{\ast }\left( s|\gamma \right) $ is the bootstrap estimate of
the $s$th quantile of the statistic $\widehat{\xi }\left( \gamma \right)
^{-1}QLR_{n}\left( \gamma \right) $ when $\gamma _{0}=\gamma $. In other
words, it denotes the bootstrap critical value of level ($1-s$) testing for $%
\mathcal{H}_{0}:\gamma _{0}=\gamma $. In practice, one would estimate $%
F_{n}^{\ast }\left( s|\gamma \right) $ over a grid of $\gamma ^{\prime }s$
and use some smoothing method such as linear interpolation or kernel
averaging to obtain a smoothed bootstrap quantile function over a range of $%
\gamma $. The region $\widehat{\Gamma}_{s}^{\ast }$ is known as $s$-level
grid bootstrap confidence interval (BCI) of $\gamma $ in the terminology of
Hansen $\left( 1999\right) $.


\begin{figure}[tbp]
\begin{center}
\includegraphics[width=6in]{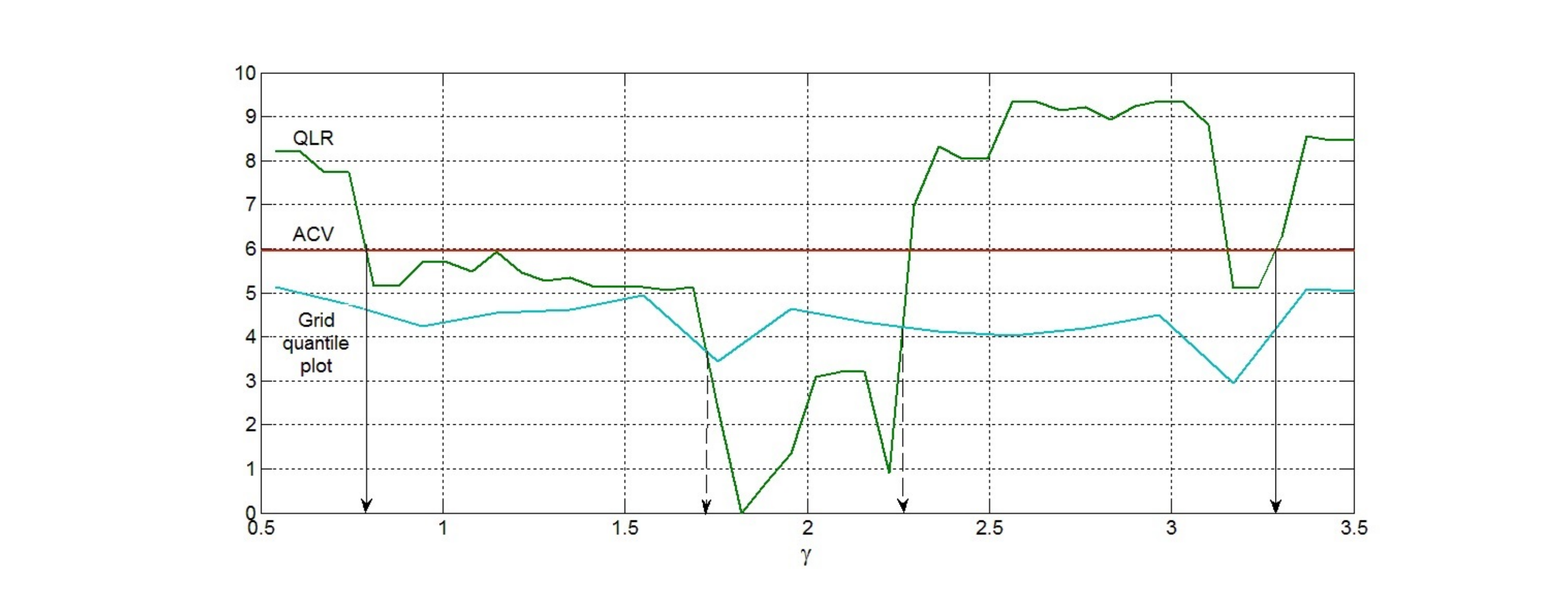}
\end{center}
\caption{$90\%$ asymptotic and grid bootstrap confidence intervals, $n=100$, 
$\protect\gamma_0=2$}
\end{figure}

Figure 1 illustrates how this confidence interval can be obtained in
practice. The $QLR_{n}\left( \gamma \right) $ line is the linear
interpolation of the rescaled $QLR_{n}\left( \gamma \right) $ statistic over
the grid of $\gamma $ at 50 points. The ACV line is the asymptotic critical
value of Hansen $\left( 2000\right) $. The true value of $\gamma_0$ was $2$.
We estimated bootstrap quantile function (described in the sequel) at 17
grid points and present the interpolated line as Grid quantile plot. The
vertical arrow at intersections between $QLR_{n}\left( \gamma \right) $ and
ACV yield the asymptotic confidence interval (ACI), while the vertical
broken arrows indicate grid BCI based on the bootstrap.

Now, we describe the bootstrap procedure for the grid bootstrap. We repeat
the following procedure for each values of $\gamma _{j}\in \left\{ \gamma
_{1},...,\gamma _{g}\right\} $.

\subsection{\textbf{Bootstrap Algorithm for each $\protect\gamma _{j}$}}

\begin{description}
\item[\emph{STEP 1}] Obtain LSE $\left( \widehat{\alpha }^{\prime },\widehat{%
\gamma }\right) ^{\prime }$ by minimizing $\left( \ref{s_theta}\right) $ and
compute the LSE residuals 
\begin{equation*}
\widehat{\varepsilon }_{t}=y_{t}-\widehat{\alpha }^{\prime }x_{t}\left( 
\widehat{\gamma }\right), \text{ \ }t=1,...,n\text{.}
\end{equation*}

\item[\emph{STEP 2}] Generate $\left\{ \eta _{t}\right\} _{t=1}^{n}$ as $%
i.i.d.$ zero mean random variables with unit variance and finite fourth
moments, and compute 
\begin{equation*}
y_{t}^{\ast }=\widehat{\alpha }^{\prime }x_{t}\left( \gamma _{j}\right) +%
\widehat{\varepsilon }_{t}\eta _{t}\text{,\ \ \ }t=1,...,n\text{.}
\end{equation*}

\item[\emph{STEP 3}] Obtain the least squares estimate using $\{y_{t}^{\ast
}\}_{t=1}^{n}$ and $\{x_{t}\}_{t=1}^{n},$ 
\begin{equation}
\widehat{\theta }^{\ast }= \underset{\theta}{\func{argmin}} \ {{\mathbb{S}}}%
_{n}^{\ast }\left( \theta \right) :=\frac{1}{n}\sum_{t=1}^{n}\left(
y_{t}^{\ast }-x_{t}\left( \gamma \right) ^{\prime }\alpha \right) ^{2}\text{.%
}  \label{theta_star}
\end{equation}

\item[\emph{STEP 4}] Compute the bootstrap analogues of $QLR_{n}$ and $%
\widehat{\xi }$ as 
\begin{equation*}
QLR_{n}^{\ast }=n\frac{\widehat{\mathbb{S}}_{n}^{\ast }\left( \gamma
_{j}\right) -\widehat{\mathbb{S}}_{n}^{\ast }\left( \widehat{\gamma }^{\ast
}\right) }{\widehat{\mathbb{S}}_{n}^{\ast }\left( \widehat{\gamma }^{\ast
}\right) },
\end{equation*}%
and 
\begin{equation}
\widehat{\xi }^{\ast }=\frac{\sum_{t=1}^{n}(\widehat{\delta }^{\ast \prime
}x_{t})^{2}\widehat{\varepsilon }_{t}^{\ast 2}K\left( \frac{q_{t}-\widehat{%
\gamma }^{\ast }}{a}\right) }{\mathbb{S}_{n}(\widehat{\theta }^{\ast
})\sum_{t=1}^{n}(\widehat{\delta }^{\ast \prime }x_{t})^{2}K\left( \frac{%
q_{t}-\widehat{\gamma }^{\ast }}{a}\right) }\text{,}  \label{xhihatBoot}
\end{equation}%
where $\widehat{\mathbb{S}}_{n}^{\ast }\left( \gamma \right) $ is defined
analogously as $\widehat{\mathbb{S}}_{n}\left( \gamma \right) $ in (\ref%
{ssngm}) by replacing $y_{t}$ with $y_{t}^{\ast }$.

\item[\emph{STEP 5}] Compute the bootstrap 100$s$-th quantile $F_{n}^{\ast
}\left( s|\gamma _{j}\right) $ from the empirical distribution of $\widehat{%
\xi }^{\ast -1}QLR_{n}^{\ast }$ by repeating STEPs 2-4.
\end{description}

Next, we derive the convergences of the bootstrap LSE $\widehat{\alpha }%
^{\ast }$ and $\widehat{\gamma }^{\ast }$ for both continuous and
discontinuous setups and show the consistency of the bootstrap statistic $%
\widehat{\xi }^{\ast }$. These results then yield the validity of the
bootstrap test inversion confidence set following the same arguments in the
proof of Theorem \ref{Th:Gamma_s}.

As usual, the superscript \textquotedblleft $^{\ast }$\textquotedblright 
{\Large \ }indicates the bootstrap quantities and convergences of bootstrap
statistics conditional on the original data. As in Shao and Tu (1995), the
notation \textquotedblleft $\overset{d^{\ast }}{\longrightarrow },$ \textit{%
in Probability\textquotedblright\ }signifies the the convergence in
Probability of the random distribution functions of the bootstrap statistics
in terms of the uniform metric and $A_{n}^{\ast }=o_{p^{\ast }}\left(
1\right) $ means that $A_{n}^{\ast }\overset{d^{\ast }}{\longrightarrow }0,$ 
\textit{in Probability.}

\begin{theorem}
\label{Th:AD_cubeBoot} Suppose that Assumptions \textbf{Z} and \textbf{Q}
hold true. \newline
$\left( \mathbf{a}\right) $ Under Assumption \textbf{C}, $\widehat{\alpha }%
^{\ast }$ and $\widehat{\gamma }^{\ast }$ are asymptotically independent and
(in probability) 
\begin{eqnarray*}
&&n^{1/2}(\widehat{\alpha }^{\ast }-\widehat{\alpha })\overset{d^{\ast }}{%
\longrightarrow }\mathcal{N}\left( 0,M^{-1}\Omega M^{-1}\right) \text{, \ }
\\
&&n^{1/3}(\widehat{\gamma }^{\ast }-\gamma _{0})\overset{d^{\ast }}{%
\longrightarrow }\arg \max_{g\in \mathbb{R}}\left( 2\delta _{30}\sqrt{\frac{%
\sigma ^{2}\left( \gamma _{0}\right) }{3}f\left( \gamma _{0}\right) }W\left(
g^{3}\right) +\frac{\delta _{30}^{2}}{3}f\left( \gamma _{0}\right)
\left\vert g\right\vert ^{3}\right) \text{. }
\end{eqnarray*}
$\left( \mathbf{b}\right) $ Under Assumption \textbf{J}, $\widehat{\alpha }%
^{\ast }$ and $\widehat{\gamma }^{\ast }$ are asymptotically independent and
(in probability) 
\begin{eqnarray*}
&&n^{1/2}(\widehat{\alpha }^{\ast }-\widehat{\alpha })\overset{d^{\ast }}{%
\longrightarrow }\mathcal{N}\left( 0,M^{-1}\Omega M^{-1}\right) \text{, } \\
&&n^{1-2\varphi }(\widehat{\gamma }^{\ast }-\gamma _{0})\overset{d^{\ast }}{%
\longrightarrow }\frac{2d^{\prime }Vd}{\left( d^{\prime }Dd\right)
^{2}f\left( \gamma _{0}\right) }\arg \max_{g\in \mathbb{R}}\left( 2W\left(
g\right) -\left\vert g\right\vert \right) \text{. }
\end{eqnarray*}
\end{theorem}

Our results can be compared with those already obtained in the literature
regarding the validity of bootstrap for non-standard estimators. First, our
consistency result seems to contradict Seijo and Sen's $\left( 2011\right) $
result on the inconsistency of a residual-based bootstrap and the
nonparametric bootstrap (with $i.i.d.$ data) for the case where $\varphi =0$%
, see also Yu $\left( 2014\right) $. The reason behind such contradictory
conclusions lies in the observation that our setup differs from theirs in an
important and vital way: they consider the case of a fixed size of the break
whereas we consider the situation that $\delta _{0}=d\cdot n^{-\varphi }$
decreases with the sample size. Thus, their limiting\ distribution depends
on the whole conditional distribution of $\varepsilon _{t}\eta _{t}d^{\prime
}x_{t}$ given $q_{t}=\gamma _{0}$ in a complicated manner, whereas ours
contains only an unknown scaling factor.

It is worth mentioning that the centering term for $\widehat{\gamma}^{\ast }$
is $\gamma _{0}$, which reflects the fact that our resampling scheme imposes
the hypothesized true value for the unknown threshold. This is important for
the validity of our bootstrap since we do not impose the continuity
restriction in our bootstrap resampling. By imposing the null value, our
resampling scheme builds on $\sqrt{n}$-consistent estimates.

Next, the consistency of $\widehat{\xi }^{\ast }$ is established in the
following proposition.

\begin{proposition}
\label{Prop:xihatstar}Suppose Assumptions \textbf{Z}, \textbf{Q} and \textbf{%
K} hold and either of Assumption \textbf{J} or Assumption \textbf{C} holds
true. Then, 
\begin{equation*}
\widehat{\xi }^{\ast }-\widehat{\xi }=o_{p^{\ast }}\left( 1\right) \text{.}
\end{equation*}
\end{proposition}

A direct consequence of Theorem \ref{Th:AD_cubeBoot} and Proposition \ref%
{Prop:xihatstar} is the following theorem.

\begin{theorem}
\label{Th:QLRBoot} Now, suppose either Assumption \textbf{J} or Assumption 
\textbf{C} hold true in addition to Assumptions \textbf{Z}, \textbf{Q} and 
\textbf{K}. Then, (in probability) 
\begin{equation*}
\widehat{\xi }^{\ast -1}QLR_{n}^{\ast }\overset{d^{\ast }}{\longrightarrow }%
\max_{g\in \mathbb{R}}\left( 2W\left( g\right) -\left\vert g\right\vert
\right) \text{.}
\end{equation*}
\end{theorem}

\section{\textbf{Monte Carlo Experiment}}

\label{sec:MC}

We generate data based on the following 3 specifications, with settings A
and B being jump models akin to that considered in Hansen (2000, Section
4.2) and setting C representing the kink case. 
\begin{align*}
& A:\,y_{t}=2+3x_{t}+\delta x_{t}1\left\{ q_{t}>\gamma _{0}\right\}
+\varepsilon _{t}, \\
& B:\,y_{t}=2+3q_{t}+\delta q_{t}1\left\{ q_{t}>\gamma _{0}\right\}
+\varepsilon _{t}, \\
& C:\,y_{t}=2+3q_{t}+\delta (q_{t}-\gamma_{0})1\left\{ q_{t}>\gamma_{0}\right\} +\varepsilon
_{t}.
\end{align*}%
The main difference in our data generating process from that of Hansen
(2000) is the conditional heteroscedasticity in $\varepsilon _{t}$: we set $%
\varepsilon _{t}=|q_{t}|e_{t}$ where $\left\{ e_{t}\right\} _{t\geq 1}$ and $%
\left\{ q_{t}\right\} _{t\geq 1}$ were generated as mutually independent and 
$i.i.d.$ normal random variables with unit variance. This leads to
conditional heteroscedasticity of the form $E(\varepsilon
_{t}^{2}|q_{t})=q_{t}^{2}$, in contrast to Hansen (2000) where $\varepsilon
_{t}$ was generated from $N(0,1)$. In setting A, we generated $x_{t}$ as $%
i.i.d.$ draws from $N(2,1)$, independent of $\left\{ e_{t}\right\} _{t\geq 1}
$ and $\left\{ q_{t}\right\} _{t\geq 1}$, while we set $Eq_{t}=2.$ We
generate $\left\{ e_{t}\right\} _{t\geq 1}$ and $\left\{ q_{t}\right\}
_{t\geq 1}$ the same for setting B. For both settings A and B, we try $%
\gamma _{0}=2$ and $2.674$, which correspond to the median and third
quartile of $q_{t}$, respectively. In setting C, we set $ \gamma_{0}=0 $ and try $Eq_{t}=0$ or $-0.674
$ so that the threshold corresponds to the median or the third quartile of $%
q_{t}$, respectively. 
For the grid $\Gamma _{n}$ used in estimation of $\gamma _{0}$, we discarded 
$10\%$ of extreme values of realized $q_{t}$ and used $n/2$ number of
equidistant points.

\begin{table}[tbph]
	\caption{Monte Carlo size of test $H_{0}:\protect\gamma =\protect\gamma _{0}$
		and coverage probability of confidence intervals of $\protect\gamma _{0}$,
		model A: $q_{t}\neq x_{t}$, $\protect\delta =n^{-\protect\varphi }\protect%
		\sqrt{10}/4$ }{\small \centering
		\setlength{\tabcolsep}{1pt} }
	\par
	{\small \ 
		\begin{tabular}{cc|r|rrr|c|rrr|rrr}
			\hline
			&  & \multicolumn{4}{|c|}{Size} & \multicolumn{7}{|c}{Coverage Probability}
			\\ \hline
			&  & $\gamma _{0}$ & \multicolumn{3}{|c|}{median of $q_{t}$(2)} & $\gamma
			_{0}$ & \multicolumn{3}{|c}{median of $q_{t}$(2)} & \multicolumn{3}{|c}{
				third quart. of $q_{t}$(2.674)} \\ 
			$\varphi$ &  & $s$\textbackslash {}$n$ & 100 & 250 & 500 & 
			\multicolumn{1}{|r|}{$\zeta $\textbackslash{}$n$} & 100 & 250 & 500 & 100 & 
			250 & 500 \\ \hline
			1/4 & Asym & 0.01 & 0.095 & 0.059 & 0.044 & \multicolumn{1}{|r|}{0.9} & 0.733
			& 0.770 & 0.774 & 0.811 & 0.834 & 0.844 \\ 
			&  & 0.05 & 0.195 & 0.153 & 0.130 & \multicolumn{1}{|r|}{0.95} & 0.818 & 
			0.832 & 0.857 & 0.870 & 0.895 & 0.914 \\ 
			&  & 0.1 & 0.290 & 0.242 & 0.200 & \multicolumn{1}{|r|}{0.99} & 0.916 & 0.938
			& 0.950 & 0.953 & 0.971 & 0.980 \\ 
			& B/rap & 0.01 & 0.003 & 0.015 & 0.009 & \multicolumn{1}{|r|}{0.9} & 0.756 & 
			0.810 & 0.840 & 0.783 & 0.826 & 0.852 \\ 
			&  & 0.05 & 0.052 & 0.055 & 0.037 & \multicolumn{1}{|r|}{0.95} & 0.833 & 
			0.880 & 0.910 & 0.859 & 0.892 & 0.915 \\ 
			&  & 0.1 & 0.106 & 0.095 & 0.083 & \multicolumn{1}{|r|}{0.99} & 0.928 & 0.959
			& 0.969 & 0.935 & 0.965 & 0.980 \\ \hline
			1/8 & Asym & 0.01 & 0.068 & 0.037 & 0.029 & \multicolumn{1}{|r|}{0.9} & 0.79
			& 0.837 & 0.897 & 0.817 & 0.835 & 0.872 \\ 
			&  & 0.05 & 0.164 & 0.092 & 0.077 & \multicolumn{1}{|r|}{0.95} & 0.856 & 
			0.898 & 0.923 & 0.873 & 0.91 & 0.914 \\ 
			&  & 0.1 & 0.214 & 0.15 & 0.129 & \multicolumn{1}{|r|}{0.99} & 0.933 & 0.961
			& 0.975 & 0.949 & 0.964 & 0.972 \\ 
			& B/rap & 0.01 & 0.006 & 0.009 & 0.008 & \multicolumn{1}{|r|}{0.9} & 0.791 & 
			0.846 & 0.881 & 0.792 & 0.827 & 0.871 \\ 
			&  & 0.05 & 0.046 & 0.052 & 0.049 & \multicolumn{1}{|r|}{0.95} & 0.858 & 
			0.907 & 0.93 & 0.859 & 0.9 & 0.917 \\ 
			&  & 0.1 & 0.099 & 0.095 & 0.105 & \multicolumn{1}{|r|}{0.99} & 0.936 & 0.968
			& 0.98 & 0.938 & 0.963 & 0.972 \\ \hline
		\end{tabular}
	}
	\\
	\parbox{1\textwidth}{\footnotesize \emph{Note: }  Size results for test of $H_{0}:\gamma
		=\gamma _{0}$ with nominal size $s$ based on Hansen $(2000)$'s asymptotic
		distribution (Asym), and our bootstrap (B/rap). Coverage probability results
		for $\gamma _{0}$ with asymptotic confidence interval based on Hansen $%
		(2000) $ and our grid bootstrap confidence interval, with nominal confidence
		level $\zeta $. $\delta =n^{-1/4}\sqrt{10}/4 = 0.25, 0.1988, 0.1672$, $%
		\delta = n^{-1/8}\sqrt{10}/4 =0.4446,0.3965,0.3636$ for $n=100, 250, 500$ }
	
\end{table}

We investigate finite-sample performance of testing and confidence regions
for $\gamma $ given in Sections \ref{sec:Unified} and \ref{sec:btrp}. We
first compare the Monte Carlo size of tests for the correct location of the
threshold, based on the asymptotic theory of Hansen $\left( 2000\right) $,
which covers diminishing jump models, and our bootstrap method. We then
investigate coverage probabilities of confidence intervals, constructed from
either the asymptotic theory of Hansen $\left( 2000\right) $, or
test-inversion based on our bootstrap. Our method has the virtue of
robustness across different settings, and the objective is to see how it
works across the jump settings of A and B and the kink setting of C. In A
and B, we try two sets of $\delta $ with $\varphi =1/4,1/8$: $\delta
=n^{-1/4}\sqrt{10}/4=0.25,0.1988,0.1672$, and $\delta =n^{-1/8}\sqrt{10}%
/4=0.4446,0.3965,0.3636$ for $n=100,250,500$ reflecting Assumption \textbf{J}%
. In setting C, $\delta $ is fixed at $\delta =2$ in line with Assumption 
\textbf{C}.\footnote{%
Note that $\delta =0.25,2$ were the smallest and the largest values of $%
\delta $ tried in Hansen (2000), respectively.} For the estimate $\widehat{%
\xi }$ of the scale factor for the $QLR_{n}$ statistic, Epanechnikov kernel
and minimum-MSE bandwidth choice, given in H\"{a}rdle and Linton $(1994)$,
were deployed. 

\begin{table}[tbp]
\caption{Monte Carlo size of test $H_{0}:\protect\gamma =\protect\gamma _{0}$
and coverage probability of confidence intervals of $\protect\gamma _{0}$,
model B: $q_{t}=x_{t}$, $\protect\delta =n^{-\protect\varphi}\protect\sqrt{10%
}/4 $ }{\small \centering
\setlength{\tabcolsep}{1pt} }
\par
{\small \ 
\begin{tabular}{cc|r|rrr|c|rrr|rrr}
\hline
&  & \multicolumn{4}{|c|}{Size} & \multicolumn{7}{|c}{Coverage Probability}
\\ \hline
&  & $\gamma _{0}$ & \multicolumn{3}{|c|}{median of $q_{t}$(2)} & $\gamma
_{0}$ & \multicolumn{3}{|c}{median of $q_{t}$(2)} & \multicolumn{3}{|c}{
third quart. of $q_{t}$(2.674)} \\ 
$\varphi$ &  & $s$\textbackslash {}$n$ & 100 & 250 & 500 & 
\multicolumn{1}{|r|}{$\zeta $\textbackslash{}$n$} & 100 & 250 & 500 & 100 & 
250 & 500 \\ \hline
1/4 & Asym & 0.01 & 0.185 & 0.145 & 0.155 & \multicolumn{1}{|r|}{0.9} & 0.608
& 0.612 & 0.658 & 0.740 & 0.730 & 0.725 \\ 
&  & 0.05 & 0.344 & 0.293 & 0.268 & \multicolumn{1}{|r|}{0.95} & 0.687 & 
0.707 & 0.742 & 0.813 & 0.817 & 0.827 \\ 
&  & 0.1 & 0.437 & 0.379 & 0.365 & \multicolumn{1}{|r|}{0.99} & 0.831 & 0.851
& 0.859 & 0.905 & 0.924 & 0.926 \\ 
& B/rap & 0.01 & 0.022 & 0.013 & 0.021 & \multicolumn{1}{|r|}{0.9} & 0.770 & 
0.836 & 0.866 & 0.868 & 0.882 & 0.878 \\ 
&  & 0.05 & 0.101 & 0.066 & 0.071 & \multicolumn{1}{|r|}{0.95} & 0.853 & 
0.894 & 0.924 & 0.932 & 0.943 & 0.943 \\ 
&  & 0.1 & 0.203 & 0.126 & 0.133 & \multicolumn{1}{|r|}{0.99} & 0.946 & 0.972
& 0.982 & 0.975 & 0.984 & 0.980 \\ \hline
1/8 & Asym & 0.01 & 0.155 & 0.098 & 0.079 & \multicolumn{1}{|r|}{0.9} & 0.661
& 0.72 & 0.786 & 0.771 & 0.779 & 0.791 \\ 
&  & 0.05 & 0.285 & 0.207 & 0.158 & \multicolumn{1}{|r|}{0.95} & 0.745 & 
0.802 & 0.852 & 0.852 & 0.844 & 0.855 \\ 
&  & 0.1 & 0.368 & 0.275 & 0.224 & \multicolumn{1}{|r|}{0.99} & 0.86 & 0.886
& 0.921 & 0.925 & 0.941 & 0.938 \\ 
& B/rap & 0.01 & 0.029 & 0.009 & 0.017 & \multicolumn{1}{|r|}{0.9} & 0.797 & 
0.871 & 0.904 & 0.886 & 0.891 & 0.888 \\ 
&  & 0.05 & 0.093 & 0.073 & 0.065 & \multicolumn{1}{|r|}{0.95} & 0.878 & 
0.917 & 0.945 & 0.936 & 0.946 & 0.943 \\ 
&  & 0.1 & 0.171 & 0.113 & 0.109 & \multicolumn{1}{|r|}{0.99} & 0.95 & 0.981
& 0.99 & 0.984 & 0.984 & 0.98 \\ \hline
\end{tabular}
}
\parbox{1\textwidth}{\footnotesize \emph{Note: }  Size results for test of $H_{0}:\gamma
=\gamma _{0}$ with nominal size $s$ based on Hansen $(2000)$'s asymptotic
distribution (Asym), and our bootstrap (B/rap). Coverage probability results
for $\gamma _{0}$ with asymptotic confidence interval based on Hansen $%
(2000) $ and our grid bootstrap confidence interval, with nominal confidence
level $\zeta $. $\delta =n^{-1/4}\sqrt{10}/4 = 0.25, 0.1988, 0.1672$, $%
\delta = n^{-1/8}\sqrt{10}/4 =0.4446,0.3965,0.3636$ for $n=100, 250, 500$.}
\end{table}

Columns 4-6 of Tables 1-3 present Monte Carlo size of test of $%
H_0:\gamma=\gamma_0$ when $\gamma_0$ is the median of $q_t$ for nominal
sizes $s=0.1, 0.05, 0.01$ for the three settings. We carried out 10,000
iterations, with one bootstrap per iteration, using the warp-speed method of
Giacomini, Dimitris and White (2013). Using the asymptotic critical values
delivers poor Monte Carlo sizes in settings A and B with substantial
over-sizing, which is more severe in setting B. In contrast, the bootstrap
test produces sizes that are close to the nominal ones, apart from $n=100$
in B, for both $\varphi$. For the asymptotic test, the size results are
somewhat better when $\varphi =1/8$ compared to $\varphi =1/4$ in settings A
and B, although the over-sizing remains severe even for $\varphi =1/8$ in
setting B as shown in Table 2. For the kink setting C, asymptotic test based
on Hansen's $(2000)$ results produces sizes that become very small with
increasing $n$, while the bootstrap test leads to good size results for $%
n=250, 500$.

Columns 8-10 of Tables 1-3 report the coverage probabilities of confidence
intervals for $\gamma_0$ in the three settings, when $\gamma_0$ is the
median of $q_t$, and columns 11-13 present the case when $\gamma_0$ is the
third quartile of $q_t$, for confidence levels $\zeta=0.9, 0.95, 0.99$.
Results are based on 1,000 iterations and in each iteration, we generated
bootstrap quantile plots by interpolating bootstrap quantiles obtained at 10
equidistant points of the realized support of $q_t$ from 399 bootstraps, and
found intersections with the sample $QLR_n$ plot formed by interpolating
between $n/2$ number of equidistant points after discarding $10\%$ of
extreme values of realized $q_t$.

In settings A and B reported in Tables 1 and 2, the coverage probability
results are better when $\gamma _{0}$ is the third quartile of $q_{t}$ for
both methods when $\varphi =1/4$. For $\varphi =1/8$, this is still the
case, with the exception of bootstrap coverage probabilities in setting A,
which are similar between the two values of $\gamma _{0}$. In setting A as
shown in Table 1, the asymptotic and bootstrap methods perform similarly,
reporting lower-than-nominal coverage probabilities which improve with
larger $n$. In setting B, the bootstrap method delivers substantially better
coverage probabilities than the asymptotic confidence intervals based on
Hansen $(2000)$, which remain substantially lower than the nominal level
even for $n=500$ for $\varphi =1/4$. Such under-coverage of asymptotic
confidence intervals for small $\delta =0.25$ was also reported in Hansen's
(2000) Table 2, for homoskedastic error case. The coverage probability
results are better when $\varphi =1/8$ compared to $\varphi =1/4$ for both
methods in setting B, especially so for asymptotic confidence intervals. In
Hansen's (2000) Table 2, coverage probability was also good for $\delta =0.5$.

In setting C reported in Table 3, the asymptotic coverage probabilities
becomes close to 1 for all values of $\zeta $ for $n=250, 500$, while
bootstrap coverage probabilities are satisfactory for $n=250, 500$. The
bootstrap coverage probability is better when $\gamma_0$ is the third
quartile of $q_t$ compared to when it is the median.\footnote{%
In Table 4 in Online Appendix, we report Monte Carlo size and coverage
probability results for $\gamma$ when $\varphi=0$ with $\delta$ fixed at $%
\sqrt{10}/4=0.7906$ and $0.25$ in setting A ($q_t\neq x_t$) with
homoscedastic error. Fixed jump setup is not covered by Hansen (2000) or our
bootstrap of Section \ref{sec:btrp}, but nonetheless we investigate how the
two methods perform in this setting for completeness.}

\begin{table}[tbp]
	\caption{Monte Carlo size of test $H_{0}:\protect\gamma =\protect\gamma _{0}$
		and coverage probability of confidence intervals of $\protect\gamma _{0}$,
		model C, kink, $\protect\delta=2$}{\small \centering
		\setlength{\tabcolsep}{1pt} }
	\par
	{\small \ 
		\begin{tabular}{cc|r|rrr|c|rrr|rrr}
			\hline
			&  & \multicolumn{4}{|c|}{Size} & \multicolumn{7}{|c}{Coverage Probability}
			\\ \hline
			&  & $\gamma _{0}$ & \multicolumn{3}{|c|}{median of $q_{t}$} & $\gamma _{0}$
			& \multicolumn{3}{|c}{median of $q_{t}$} & \multicolumn{3}{|c}{third quart.
				of $q_{t}$} \\ 
			&  & $s$\textbackslash {}$n$ & 100 & 250 & 500 & \multicolumn{1}{|r|}{$\zeta 
				$\textbackslash{}$n$} & 100 & 250 & 500 & 100 & 250 & 500 \\ \hline
			C & Asym & 0.01 & 0.123 & 0.028 & 0.005 & \multicolumn{1}{|r|}{0.9} & 0.802
			& 0.946 & 0.975 & 0.749 & 0.925 & 0.972 \\ 
			&  & 0.05 & 0.168 & 0.043 & 0.015 & \multicolumn{1}{|r|}{0.95} & 0.84 & 0.965
			& 0.983 & 0.784 & 0.945 & 0.98 \\ 
			&  & 0.1 & 0.200 & 0.056 & 0.024 & \multicolumn{1}{|r|}{0.99} & 0.892 & 0.982
			& 0.992 & 0.852 & 0.966 & 0.99 \\ 
			& B/rap & 0.01 & 0.027 & 0.014 & 0.012 & \multicolumn{1}{|r|}{0.9} & 0.768 & 
			0.854 & 0.805 & 0.828 & 0.894 & 0.877 \\ 
			&  & 0.05 & 0.091 & 0.054 & 0.052 & \multicolumn{1}{|r|}{0.95} & 0.817 & 
			0.918 & 0.889 & 0.88 & 0.949 & 0.943 \\ 
			&  & 0.1 & 0.153 & 0.108 & 0.104 & \multicolumn{1}{|r|}{0.99} & 0.905 & 0.979
			& 0.975 & 0.954 & 0.981 & 0.984 \\ \hline
		\end{tabular}
	}
	\parbox{1\textwidth}{\footnotesize \emph{Note: }  Size results for test of $H_{0}:\gamma
		=\gamma _{0}$ with nominal size $s$ based on Hansen $(2000)$'s asymptotic
		distribution(Asym), and our bootstrap(B/rap). Coverage probability results
		for $\gamma _{0}$ with asymptotic confidence interval based on Hansen $%
		(2000) $ and our grid bootstrap confidence interval, with nominal confidence
		level $\zeta $. When $\gamma_0$ is median, $q_t \sim N(0,1) $. When $%
		\gamma_0$ is third quartile, $q_t \sim N(-0.674,1)$ and $ \gamma_0 =0 $.}
\end{table}

\section{\textbf{EMPIRICAL APPLICATION: GROWTH AND DEBT}}

\label{sec:Application}

The so-called Reinhart-Rogoff hypothesis postulates that above some
threshold (90$\%$ being their estimate of this threshold), higher
debt-to-GDP ratio is associated with lower GDP growth rate. There have been
numerous studies that utilize the threshold regression models to assess this
hypothesis, including Hansen $(2017)$ who fitted a kink model to a time
series of US annual data, see Hansen $(2017)$ for references on earlier
studies which fitted jump models to various data sets. As there is little
guidance from economic theory on the choice between kink and jump models in
this setting, we advocate the use of our robust inference on the threshold
and slope parameters of the model.

Hansen $(2017)$ had fitted a kink model to US annual data on real GDP growth
rate in year $t$ ($y_{t}$) and debt-to-GDP ratio from the \textit{previous}
year ($q_{t}$) for the period spanning 1792-2009 ($n=218$), and estimated
the threshold to be $43.8\%$, while the slope parameters of $q_{t}$ were not
significant. Before fitting the jump model to this data, we first tested for
the presence of threshold effect using the testing procedure of Hansen
(1996) with 1,000 bootstrap replications, and obtained $p$-value of 0.047,
rejecting the null hypothesis of no threshold effect. This is in contrast to
the $p$-value of 0.15 obtained by Hansen's $(2017)$ test for presence of
threshold effect when imposing the kink model. Hansen $(2017)$ had remained
inconclusive on the presence of kink threshold effect, since the bootstrap
method used there did not account for the time series nature of data and the
high $p$-value could have been due to modest power of the test.

The fitted jump model is given by: 
\begin{equation*}
\widehat{y}_{t}=\left\{ 
\begin{array}{ll}
\underset{(0.87)}{4.82}-\underset{(0.16)}{0.052}y_{t-1}-\underset{(0.049)}{%
0.114}q_{t}, & \text{if }q_{t}\leq 17.2 \\ 
\underset{(0.74)}{2.78}+\underset{(0.082)}{0.49}y_{t-1}-\underset{(0.012)}{%
0.017}q_{t}, & \text{if }q_{t}>17.2%
\end{array}%
\right.
\end{equation*}%
The sizes of the two regimes were 99 (below 17.2$\%$) and 109 (above 17.2$\%$%
). We obtained grid bootstrap confidence intervals for $\gamma _{0}$ to be
(10.5, 39) for 95$\%$ confidence level and (10.8, 38.6) for 90$\%$, based on
399 bootstrap iterations. Bootstrap quantiles were obtained at 38 grid
points, which included $\widehat{\gamma }$, $\widetilde{\gamma }$ and
equidistant points on the realized support of $q_{t}$ after discarding 7.5$%
\% $ of the largest and smallest values of $q_{t}$ in the sample.{\footnote{%
There is currently no theoretical guide to the choice of the trimming
parameter. Our choice of trimming out 7.5$\%$ was guided by Sweden's
estimated $\tilde{\gamma}$ being the 12-th percentile of the $q_{t}$ in the
data. Sensitivity check on changing choices of the trimming value is
recommended.} }We find the points of intersection between the linearly
interpolated bootstrap quantile line and the linear interpolation of sample $%
QLR_{n}(\gamma )$ test statistics for $H_{0}:\gamma _{0}=\gamma _{j}$ at
grid points $\gamma _{j}$ consisting of 73 equidistant points and $\widehat{%
\gamma }$, $\widetilde{\gamma }$, as shown in Figure 2 for 90$\%$ confidence
level. 
\begin{figure}[tbp]
\begin{center}
\includegraphics[width=5in]{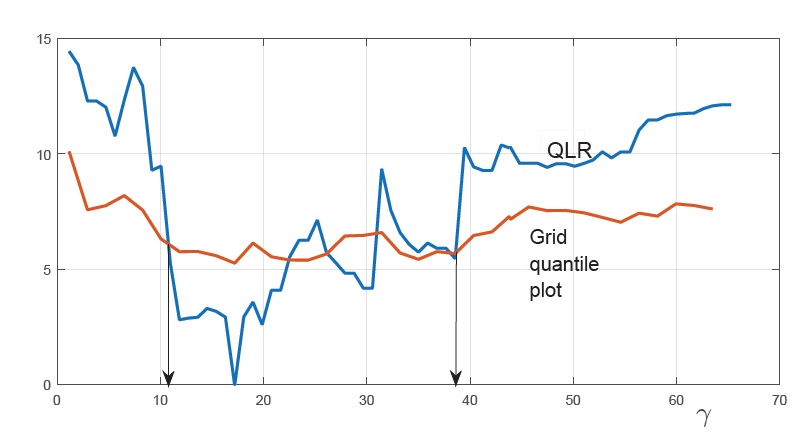}
\end{center}
\caption{$90\%$ grid bootstrap confidence interval for the US}
\end{figure}

As the estimated threshold under the jump model is noticeably small at 17.2$%
\%$, our estimated jump model which suggests insignificance of effect of $%
q_t $ on $y_t$ above the threshold does not necessarily contradict the
Reinhart-Rogoff hypothesis. To see if this could be an indication of
presence of further threshold points, we applied Hansen (1996)'s testing
procedure for presence of threshold effect on the lower and upper subsamples
with 1000 bootstraps and obtained $p$-values of 0.025 and 0.016,
respectively. Hence, we conclude that the US time series data should be
fitted to a threshold regression model with multiple threshold points.

To see if such conclusion holds across different countries, we proceeded by
first applying Hansen (1996)'s test for the presence of threshold effect on
Reinhart and Rogoff's (2010) data for countries with relatively long time
spans without missing observations. For Australia($n=107$) and the UK($n=178 
$), the $p$-values with 1000 bootstraps were 0.795 and 0.98 so we conclude
that there is no threshold effect for these countries in the relationship
between the GDP growth and the debt-to-GDP ratio.

For data from Sweden for the period 1881-2009 ($n=129$), the $p$-value for
Hansen (1996)'s test of presence of threshold effect with 1000 bootstraps
for the whole sample is 0.048, while for the lower and upper regimes,
divided by $\widehat{\gamma}$, they were 0.979 and 0.131, respectively. The
estimated jump model is: 
\begin{equation*}
\widehat{y}_{t}=\left\{ 
\begin{array}{rr}
\underset{(2.17)}{1.12}-\underset{(0.24)}{0.2}y_{t-1}+\underset{(0.11)}{0.13}%
q_{t},\quad \text{if } q_{t}\leq {21.3} &  \\ 
\underset{(0.58)}{1.86}+\underset{(0.11)}{0.48}y_{t-1}-\underset{(0.0082)}{%
0.004}q_{t},\quad \text{if } q_{t}>21.3 & 
\end{array}%
\right.
\end{equation*}%
with the lower regime having 61 observations and upper regime containing 68.
The coefficient of debt-to-GDP ratio is not statistically significant.

The grid bootstrap confidence intervals for $\gamma _{0}$ were (15.3, $%
\infty $) and (16.4, $\infty$) for 95$\%$ and 90$\%$ confidence levels.
Shown in Figures 3 are linear interpolation of 90$\%$ bootstrap quantiles at
27 grid points with 399 bootstraps and linear interpolation of QLR test
statistic at each of 54 grid points. 

\begin{figure}[tbp]
\begin{center}
\includegraphics[width=5in]{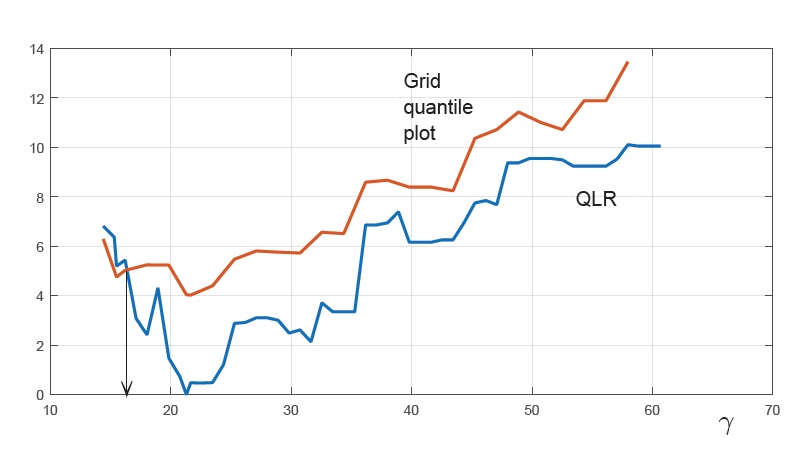}
\end{center}
\caption{$90\%$ grid bootstrap confidence interval for Sweden}
\end{figure}

We conclude that there is substantial heterogeneity across countries in the
relationship between the GDP growth and the debt-to-GDP ratio, not only in
the values of model parameters, but also in the kinds of models that are
suitable.

\section{\textbf{CONCLUSION}}

\label{sec:Conclusion}

This paper has developed unified inferential procedures for the threshold
regression model. The unconstrained least squares estimator of the
regression coefficient $\alpha $ turns out to enjoy the useful oracle
property, which enables the standard asymptotic normal inference as in the
linear regression model. On the other hand, we provide a judiciously
constructed statistic, with which one can make inference of the unknown
threshold without knowing the continuity of the threshold regression model.
Asymptotically valid bootstrap inference is also proposed and shown to
improve the finite sample performance of the asymptotic procedure.

An interesting future research area is extension to the nonparametric
setting. For instance, see Card et al. (2008) and Pan (2015), who use the
regression discontinuity methods \footnote{%
Pan (2015, p.378) and a referee emphasize that this setting is not identical
to the conventional regression discontinuity method (e.g. Angrist and Lavy
(1999); Hahn \textit{et al}. (2001)) due to the lack of knowledge on the
precise location of the discontinuity.} to test for the tipping phenomenon
in racial segregation and gender segregation, respectively, or Landais
(2014), who recommends testing for the location of the change-point as a
validity check for the regression discontinuity design, even when the
change-point is suggested by the institutional knowledge.

\appendix

\section{\textbf{PROOFS OF MAIN THEOREMS}}

Let us introduce some notation first. In what follows $C,C_{1},$... denote
generic positive finite constants, which may vary from line to line or
expression to expression. Recall that $x_{t}=\left( 1,x_{t2}^{\prime
},q_{t}\right) ^{\prime },$ $x_{t1}=\left( 1,x_{t2}^{\prime }\right)
^{\prime }$, and $\mathbf{1}_{t}\left( b\right) =\mathbf{1}\left\{
q_{t}>b\right\} $, and introduce $\mathbf{1}_{t}\left( a;b\right) =\mathbf{1}%
\left\{ a<q_{t}<b\right\} $. Finally, we abbreviate $\psi -\psi _{0}$ by $%
\overline{\psi }$ for any parameter $\psi $.

All the technical lemmas are given in the online supplement to this paper.

\subsection{\textbf{Proof of Proposition \protect\ref{Consistency}}}

$\left. {}\right. $

Without loss of generality we assume that $\widehat{\gamma }\geq \gamma _{0}$
and $\gamma _{0}=0$, so that $\delta _{10}=0\ $and $\delta _{20}=0$ under
Assumption C. By definition, we have that 
\begin{eqnarray*}
{{\mathbb{S}}}_{n}\left( \theta \right) -{{\mathbb{S}}}_{n}\left( \theta
_{0}\right) &=&\frac{1}{n}\sum_{t=1}^{n}\left\{ \left( y_{t}-\alpha ^{\prime
}x_{t}\left( \gamma \right) \right) ^{2}-\varepsilon _{t}^{2}\right\} \\
&=&\frac{1}{n}\sum_{t=1}^{n}\left\{ \left( \overline{\beta }^{\prime }x_{t}+ 
\overline{\delta }^{\prime }x_{t}\mathbf{1}_{t}\left( \gamma \right) +\delta
_{0}^{\prime }x_{t}\mathbf{1}_{t}\left( 0;\gamma \right) +\varepsilon
_{t}\right) ^{2}-\varepsilon _{t}^{2}\right\} \text{.}
\end{eqnarray*}

By standard algebra and denoting $\upsilon =\beta +\delta $, 
\begin{eqnarray*}
&&\overline{\beta }^{\prime }x_{t}+\overline{\delta }^{\prime }x_{t}\mathbf{%
\ 1 }_{t}\left( \gamma \right) +\delta _{0}^{\prime }x_{t}\mathbf{1}%
_{t}\left( 0;\gamma \right) \\
&=&\overline{\upsilon }^{\prime }x_{t}\mathbf{1}_{t}\left( \gamma \right)
+\left( \overline{\beta }+\delta _{0}\right) ^{\prime }x_{t}\mathbf{1}
_{t}\left( 0;\gamma \right) +\overline{\beta }^{\prime }x_{t}\mathbf{1}
_{t}\left( -\infty ;0\right) \text{,}
\end{eqnarray*}
which implies, because of the orthogonality of the terms on the right of the
last displayed expression, that 
\begin{equation*}
{{\mathbb{S}}}_{n}\left( \theta \right) -{{\mathbb{S}}}_{n}\left( \theta
_{0}\right) =\mathbb{A}_{n1}\left( \theta \right) +\mathbb{A}_{n2}\left(
\theta \right) +\mathbb{A}_{n3}\left( \theta \right) +\mathbb{B}_{n1}\left(
\theta \right) +\mathbb{B}_{n2}\left( \theta \right) +\mathbb{B}_{n3}\left(
\theta \right) \text{,}
\end{equation*}
where 
\begin{eqnarray*}
\mathbb{A}_{n1}\left( \theta \right) &=&\overline{\upsilon }^{\prime }\frac{
1 }{n}\sum_{t=1}^{n}x_{t}x_{t}^{\prime }\mathbf{1}_{t}\left( \gamma \right) 
\overline{\upsilon };~\ \ \ \mathbb{A}_{n2}\left( \theta \right) =\overline{
\beta }^{\prime }\frac{1}{n}\sum_{t=1}^{n}x_{t}x_{t}^{\prime }\mathbf{1}
_{t}\left( -\infty ;0\right) \overline{\beta } \\
\mathbb{A}_{n3}\left( \theta \right) &=&\left( \overline{\beta }+\delta
_{0}\right) ^{\prime }\frac{1}{n}\sum_{t=1}^{n}x_{t}x_{t}^{\prime }\mathbf{1}
_{t}\left( 0;\gamma \right) \left( \overline{\beta }+\delta _{0}\right) \\
\mathbb{B}_{n1}\left( \theta \right) &=&\overline{\upsilon }^{\prime }\frac{
2 }{n}\sum_{t=1}^{n}x_{t}\varepsilon _{t}\mathbf{1}_{t}\left( \gamma \right)
;~\ \ \ \ \mathbb{B}_{n2}\left( \theta \right) =\overline{\beta }^{\prime } 
\frac{2}{n}\sum_{t=1}^{n}x_{t}\varepsilon _{t}\mathbf{1}_{t}\left( -\infty
;0\right) \\
\mathbb{B}_{n3}\left( \theta \right) &=&\left( \overline{\beta }+\delta
_{0}\right) ^{\prime }\frac{2}{n}\sum_{t=1}^{n}x_{t}\varepsilon _{t}\mathbf{%
\ 1 }_{t}\left( 0;\gamma \right) \text{.}
\end{eqnarray*}
\textbf{Consistency. }\ It suffices to show that for any $\epsilon >0$, $%
\eta >0$, there is $n_{0}$ such that for all $n>n_{0}$, $\Pr \left\{
\left\Vert \widehat{\theta }-\theta _{0}\right\Vert >\eta \right\} <\epsilon 
$, which is implied by 
\begin{equation}
\Pr \left\{ \inf_{\left\Vert \overline{\theta }\right\Vert >\eta }\sum_{\ell
=1}^{3}E\left( \mathbb{A}_{n\ell }\left( \theta \right) \right) +\mathbb{D}
_{n\ell }\left( \theta \right) \leq 0\right\} <\epsilon \text{,}
\label{consi_2}
\end{equation}
where $\mathbb{D}_{n\ell }\left( \theta \right) =\mathbb{B}_{n\ell }\left(
\theta \right) +\left( \mathbb{A}_{n\ell }\left( \theta \right) -E\left( 
\mathbb{A}_{n\ell }\left( \theta \right) \right) \right) $ for $\ell =1,2,3$.

First $\left\Vert \overline{\theta }\right\Vert >\eta $ implies that either $%
\left( \mathbf{i}\right) $ $\left\Vert \overline{\gamma }\right\Vert >\eta
/3 $ and $\left\Vert \overline{\beta }\right\Vert \leq \eta /3$, or $\left( 
\mathbf{ii}\right) $ $\left\Vert \overline{\beta }\right\Vert >\eta /3$ or $%
\left\Vert \overline{\upsilon }\right\Vert >\eta /3$. When $\left( \mathbf{\
ii}\right) $ holds true, it is clear that 
\begin{equation}
\inf_{\left\Vert \overline{\upsilon }\right\Vert >\eta /3}E\left( \mathbb{A}%
_{n1}\left( \theta \right) \right) >C\eta ^{2}\ \ \ \ \ \text{or \ \ \ }%
\inf_{\left\Vert \overline{\beta }\right\Vert >\eta /3}E\left( \mathbb{A}%
_{n2}\left( \theta \right) \right) >C\eta ^{2}  \label{ineq_2}
\end{equation}%
whereas when $\left( \mathbf{i}\right) $ holds true, we have that 
\begin{equation}
\inf_{\left\Vert \overline{\gamma }\right\Vert >\eta /3,\left\Vert \overline{%
\beta }\right\Vert \leq \eta /3}E\left( \frac{1}{n}\sum_{t=1}^{n}\left(
x_{t}^{\prime }\left( \bar{\beta}+\delta _{0}\right) \right) ^{2}\mathbf{1}%
_{t}\left( 0;\gamma \right) \right) >C\eta ^{3}\text{,}  \label{ineq_3}
\end{equation}%
because Assumption Q implies that $E\left( x_{t}x_{t}^{\prime }\mathbf{1}%
_{t}\left( \gamma \right) \right) $, $E\left( x_{t}x_{t}^{\prime }\mathbf{1}%
_{t}\left( -\infty ;0\right) \right) $ and $E\left( x_{t}x_{t}^{\prime }%
\mathbf{1}_{t}\left( 0;\gamma \right) \right) $ are positive definite
matrices uniformly in $\gamma >\eta $ and $\left\vert \left\vert \bar{\beta}%
+\delta _{0}\right\vert \right\vert >\eta /3$ if $\left\Vert \overline{\beta 
}\right\Vert \leq \eta /3$ because we can always choose $\eta $ such that $%
\left\vert \delta _{0}\right\vert \geq 2\eta /3$. We have that 
\begin{equation}
C_{1}\leq \frac{E\mathbb{A}_{n3}\left( \theta \right) }{\left( \overline{%
\tau }_{1},\overline{\tau }_{2}^{\prime }\right) E\left(
x_{t1}x_{t1}^{\prime }\mathbf{1}_{t}\left( 0;\gamma \right) \right) \left( 
\overline{\tau }_{1},\overline{\tau }_{2}^{\prime }\right) ^{\prime }+%
\overline{\tau }_{3}^{2}E\left( q_{t}^{2}\mathbf{1}_{t}\left( 0;\gamma
\right) \right) }\leq C_{2}\text{,}  \label{pd}
\end{equation}%
where $\overline{\tau }=\left( \beta _{0}-\beta \right) +\delta _{0}$. The
motivation for the last displayed inequality comes from the fact that , say,
implies that $E\left\{ x_{t}x_{t}^{\prime }\mathbf{1}_{t}\left( \gamma
_{1};\gamma _{2}\right) \right\} $ is a strictly positive and finite
definite matrix which implies that for any vector $a^{\prime }=\left(
a_{1}^{\prime },a_{2}\right) $, 
\begin{equation*}
C^{-1}\leq \frac{a^{\prime }E\left\{ x_{t}x_{t}^{\prime }\mathbf{1}%
_{t}\left( \gamma _{1};\gamma _{2}\right) \right\} a}{a_{1}^{\prime
}E\left\{ x_{t1}x_{t1}^{\prime }\mathbf{1}_{t}\left( \gamma _{1};\gamma
_{2}\right) \right\} a_{1}+a_{2}^{2}E\left( q_{t}^{2}\mathbf{1}_{t}\left(
\gamma _{1};\gamma _{2}\right) \right) }\leq C.
\end{equation*}%
So, $\left( \ref{ineq_2}\right) $ and $\left( \ref{ineq_3}\right) $ imply
that 
\begin{equation}
\inf_{\left\Vert \overline{\theta }\right\Vert >\eta }\sum_{\ell
=1}^{3}E\left( \mathbb{A}_{n\ell }\left( \theta \right) \right) >C\eta ^{3}%
\text{.}  \label{ineq_1}
\end{equation}

On the other hand, Lemma \ref{lem:maxineq2} and the uniform law of large
numbers, respectively, imply that 
\begin{equation*}
\sup_{\left\Vert \overline{\theta }\right\Vert >\eta }\left\Vert \mathbb{B}
_{n\ell }\left( \theta \right) \right\Vert =O_{p}\left( n^{-1/2}\right) ~\ \
\ell =1,2,3;\text{ \ \ \ \ }\sup_{\gamma _{1},\gamma _{2}}\left\Vert \mathbb{%
\ \ F}_{n}\left( \gamma _{1};\gamma _{2}\right) \right\Vert =o_{p}\left(
1\right) \text{,}
\end{equation*}
where $\mathbb{F}_{n}\left( \gamma _{1};\gamma _{2}\right) =\frac{1}{n}
\sum_{t=1}^{n}\left( x_{t}x_{t}^{\prime }\mathbf{1}_{t}\left( \gamma
_{1};\gamma _{2}\right) -E\left( x_{t}x_{t}^{\prime }\mathbf{1}_{t}\left(
\gamma _{1};\gamma _{2}\right) \right) \right) $, and hence 
\begin{equation}
\sup_{\left\Vert \overline{\theta }\right\Vert >\eta /3}\left\Vert
\sum_{\ell =1}^{3}\mathbb{D}_{n\ell }\left( \theta \right) \right\Vert
=o_{p}\left( 1\right) \text{.}  \label{ineq_4}
\end{equation}

Thus $\widehat{\theta }-\theta _{0}=o_{p}\left( 1\right) $ because the left
side of $\left( \ref{consi_2}\right) $ is bounded by 
\begin{equation*}
\Pr \left\{ \inf_{\left\Vert \overline{\theta }\right\Vert >\eta }\sum_{\ell
=1}^{3}E\left( \mathbb{A}_{n\ell }\left( \theta \right) \right) \leq
\sup_{\left\Vert \overline{\theta }\right\Vert >\eta /3}\left\Vert
\sum_{\ell =1}^{3}\mathbb{D}_{n\ell }\left( \theta \right) \right\Vert
\right\} \rightarrow 0\text{,}
\end{equation*}
using $\left( \ref{ineq_1}\right) $ and $\left( \ref{ineq_4}\right) $.

\textbf{Convergence Rate. \ }We shall show next that for any $\epsilon >0$
there exist $C>0$, $\eta >0$, $n_{0}$ such that for $n>n_{0}$ we have that 
\begin{equation}
\Pr \left\{ \inf_{\frac{C}{n^{1/2}}<\left\Vert \overline{\upsilon }%
\right\Vert ,\left\Vert \overline{\beta }\right\Vert <\eta ;\frac{C}{n^{1/3}}%
<\left\Vert \overline{\gamma }\right\Vert <\eta }\sum_{\ell =1}^{3}E\left( 
\mathbb{A}_{n\ell }\left( \theta \right) \right) +\mathbb{D}_{n\ell }\left(
\theta \right) \leq 0\right\} <\epsilon \text{.}  \label{rate_1}
\end{equation}%
Since $\Pr \left\{ X_{n}+Y_{n}<0\right\} \leq \Pr \left\{ X_{n}<0\right\}
+\Pr \left\{ Y_{n}<0\right\} $ for any sequence $X_{n}$ and $Y_{n}$ and $%
\inf_{x}\left\{ f\left( x\right) +g\left( x\right) \right\} \geq
\inf_{x}f\left( x\right) +\inf_{x}g\left( x\right) $ for any \ functions $f$
and $g$, it suffices to show that for each $\ell =1,2,3$ 
\begin{equation}
\Pr \left\{ \inf_{\frac{C}{n^{1/2}}<\left\Vert \overline{\upsilon }%
\right\Vert ,\left\Vert \overline{\beta }\right\Vert <\eta ;\frac{C}{n^{1/3}}%
<\left\Vert \overline{\gamma }\right\Vert <\eta }E\left( \mathbb{A}_{n\ell
}\left( \theta \right) \right) /2+\left( \mathbb{A}_{n\ell }\left( \theta
\right) -E\left( \mathbb{A}_{n\ell }\left( \theta \right) \right) \right)
\leq 0\right\} <\epsilon  \label{rate_1a}
\end{equation}%
\begin{equation}
\Pr \left\{ \inf_{\frac{C}{n^{1/2}}<\left\Vert \overline{\upsilon }%
\right\Vert ,\left\Vert \overline{\beta }\right\Vert <\eta ;\frac{C}{n^{1/3}}%
<\left\Vert \overline{\gamma }\right\Vert <\eta }E\left( \mathbb{A}_{n\ell
}\left( \theta \right) \right) /2+\mathbb{B}_{n\ell }\left( \theta \right)
\leq 0\right\} <\epsilon \text{.}  \label{rate_1b}
\end{equation}%
To that end, we shall first examine 
\begin{equation*}
\Pr \left\{ \inf_{\Xi _{j}\left( \upsilon \right) ;\Xi _{j}\left( \beta
\right) ;\Xi _{k}\left( \gamma \right) }E\left( \mathbb{A}_{n\ell }\left(
\theta \right) \right) /2+\mathbb{B}_{n\ell }\left( \theta \right) \leq
0\right\} \text{, \ }\ell =1,2,3\text{,}
\end{equation*}%
where 
\begin{eqnarray}
\Xi _{j}\left( \psi \right) &=&\left\{ \psi :\frac{C}{n^{1/2}}%
2^{j-1}<\left\Vert \overline{\psi }\right\Vert <\frac{C}{n^{1/2}}%
2^{j}\right\} ;~\ \ \ \ j=1,...,\log _{2}\frac{\eta }{C}n^{1/2}  \notag \\
\Xi _{k}\left( \gamma \right) &=&\left\{ \gamma :\frac{C}{n^{1/3}}2^{k-1}<%
\overline{\gamma }<\frac{C}{n^{1/3}}2^{k}\right\} ;\text{ \ \ \ }%
k=1,...,\log _{2}\frac{\eta }{C}n^{1/3}\text{.}  \label{set_jk}
\end{eqnarray}%
Recall that we have assumed that $\gamma \geq 0$, as the case $\gamma \leq 0$
follows similarly.

First by standard arguments, 
\begin{eqnarray}
&&\Pr \left\{ \inf_{\Xi _{j}\left( \upsilon \right) ;\Xi _{k}\left( \gamma
\right) }E\left( \mathbb{A}_{n1}\left( \theta \right) \right) /2+\mathbb{B}
_{n1}\left( \theta \right) \leq 0\right\}  \notag \\
&\leq &\Pr \left\{ \inf_{\Xi _{j}\left( \upsilon \right) }\left\Vert 
\overline{\upsilon }\right\Vert \lambda _{\min }\left( Ex_{t}x_{t}^{\prime } 
\mathbf{1}_{t}\left( 0\right) \right) \leq \sup_{\Xi _{k}\left( \gamma
\right) }\left\Vert \frac{4}{n^{1/2}}\sum_{t=1}^{n}x_{t}\varepsilon _{t} 
\mathbf{1}_{t}\left( \gamma \right) \right\Vert \right\}  \notag \\
&\leq &\Pr \left\{ C2^{j-2}\leq \sup_{\left\{ \gamma :\left\Vert \overline{
\gamma }\right\Vert <\eta \right\} }\left\Vert \frac{1}{n^{1/2}}
\sum_{t=1}^{n}x_{t}\varepsilon _{t}\mathbf{1}_{t}\left( \gamma \right)
\right\Vert \right\}  \label{rate_2} \\
&\leq &C^{-1}2^{-j+2}\eta ^{1/2}  \notag
\end{eqnarray}
by Lemma \ref{lem:maxineq2} and the Markov's inequality. Observe that the
latter inequality is independent of $\Xi _{k}\left( \gamma \right) $. Since $%
\sum_{j=1}^{\infty }2^{-j}<\infty $, the probability in $\left( \ref{rate_1b}
\right) $ can be made arbitrary small for large $C$ or small $\eta $, thus
satisfying the condition $\left( \ref{rate_1b}\right) $. $\left( \ref%
{rate_1a}\right) $ follows similarly as is the case for $\ell =2$ and thus
it is omitted.

We next examine $\left( \ref{rate_1a}\right) $ and $\left( \ref{rate_1b}%
\right) $ for $\ell =3$. Observing $\left( \ref{pd}\right) $ and the
arguments that follow, defining 
\begin{equation*}
\widetilde{\mathbb{A}}_{n3}\left( \theta \right) =\overline{\tau }%
^{2}E\left( q_{t}^{2}\mathbf{1}_{t}\left( 0;\gamma \right) \right) ;\text{ \
\ \ \ \ \ }\widetilde{\mathbb{B}}_{n3}\left( \theta \right) =\overline{\tau }%
\frac{2}{n}\sum_{t=1}^{n}q_{t}\varepsilon _{t}\mathbf{1}_{t}\left( 0;\gamma
\right) \text{,}
\end{equation*}%
it suffices to show $\left( \ref{rate_1a}\right) $ and $\left( \ref{rate_1b}%
\right) $ for $\widetilde{\mathbb{A}}_{n3}\left( \theta \right) $ and \ $%
\widetilde{\mathbb{B}}_{n3}\left( \theta \right) $. To that end, because $%
\overline{\tau }>C_{1}$ as $\left\vert \delta _{30}\right\vert >C_{1}>0$, we
obtain, since $Eq_{t}^{2}\mathbf{1}_{t}\left( 0;\eta \right) \geq C_{1}\eta
^{3}$ 
\begin{eqnarray}
&&\Pr \left\{ \inf_{\Xi _{j}\left( \upsilon \right) ;\Xi _{k}\left( \gamma
\right) }E\left( \widetilde{\mathbb{A}}_{n3}\left( \theta \right) /2\right) +%
\widetilde{\mathbb{B}}_{n3}\left( \theta \right) \leq 0\right\}  \notag \\
&\leq &\Pr \left\{ \inf_{\Xi _{k}\left( \gamma \right) }\left\Vert \tau
_{0}\right\Vert E\left( q_{t}^{2}\mathbf{1}_{t}\left( 0;\gamma \right)
\right) \leq \sup_{\Xi _{k}\left( \gamma \right) }\left\Vert \frac{4}{n}%
\sum_{t=1}^{n}q_{t}\varepsilon _{t}\mathbf{1}_{t}\left( 0;\gamma \right)
\right\Vert \right\}  \notag \\
&\leq &\Pr \left\{ \frac{C}{n}2^{3\left( k-2\right) }\leq \sup_{\Xi
_{k}\left( \gamma \right) }\left\Vert \frac{1}{n}\sum_{t=1}^{n}q_{t}%
\varepsilon _{t}\mathbf{1}_{t}\left( 0;\gamma \right) \right\Vert \right\}
\label{rate_4} \\
&\leq &C^{-1}2^{-3k/2}\text{,}  \notag
\end{eqnarray}%
by Lemma \ref{lem:maxineq2} and Markov's inequality. Notice that this bound
is independent of $\Xi _{j}\left( \upsilon \right) $. But by summability of $%
2^{-3k/2}$, we conclude that $\left( \ref{rate_1b}\right) $ holds true for $%
\ell =3$ by choosing $C$ large enough.

We now conclude the proof after we note that the left side of $\left( \ref%
{rate_1}\right) $ is bounded by 
\begin{eqnarray*}
&&\Pr \left\{ \max_{j,k}\inf_{\Xi _{j}\left( \upsilon \right) ;\Xi
_{j}\left( \beta \right) ;\Xi _{k}\left( \gamma \right) }\sum_{\ell
=1}^{3}\left\{ E\mathbb{A}_{n\ell }\left( \theta \right) +\mathbb{B}_{n\ell
}\left( \theta \right) \right\} \leq 0\right\} \\
&\leq &C^{-1}\left( \sum_{j=1}^{\log _{2}\frac{\eta }{C}n^{1/2}}2^{-2j}+
\sum_{k=1}^{\log _{2}\frac{\eta }{C}n^{1/3}}2^{-3k/2}\right) <\epsilon
\end{eqnarray*}
using $\left( \ref{rate_2}\right) -\left( \ref{rate_4}\right) $.\hfill $%
\blacksquare $

\subsection{\textbf{Proof of Theorem \protect\ref{Th:AD_cube}}}

$\left. {}\right. $

Because the \textquotedblleft $\limfunc{argmin}$\textquotedblright\ is a
continuous mapping, see Kim and Pollard $\left( 1990\right) $, and the
convergence rates of $\widehat{\alpha }$ and $\widehat{\gamma }$ are
obtained in Proposition \ref{Consistency}, it suffices to examine the weak
limit of 
\begin{eqnarray*}
\mathbb{G}_{n}\left( h,g\right) &=&n\left( \mathbb{S}_{n}\left( \alpha _{0}+%
\frac{h}{n^{1/2}},\gamma _{0}+\frac{g}{n^{1/3}}\right) -\mathbb{S}_{n}\left(
\alpha _{0},\gamma _{0}\right) \right) \\
&=&\sum_{t=1}^{n}\left\{ \left( \varepsilon _{t}-\frac{h^{\prime }}{n^{1/2}}%
x_{t}\left( \frac{g}{n^{1/3}}\right) -\delta _{30}q_{t}\mathbf{1}_{t}\left(
0;\frac{g}{n^{1/3}}\right) \right) ^{2}-\varepsilon _{t}^{2}\right\}
\end{eqnarray*}%
over $\left\Vert h\right\Vert ,\left\vert g\right\vert \leq C$, where we
assume $\gamma _{0}=0$ as before for notational convenience and
reparametrize $h=\sqrt{n}\left( \alpha -\alpha _{0}\right) $ and $%
g=n^{1/3}\left( \gamma -\gamma _{0}\right) .$ First, due to the uniform law
of large numbers it follows that 
\begin{equation*}
\sup_{\left\vert g\right\vert \leq C}\left\vert \frac{1}{n}%
\sum_{t=1}^{n}\left\{ x_{t}\left( \frac{g}{n^{1/3}}\right) x_{t}^{\prime
}\left( \frac{g}{n^{1/3}}\right) -\mathbf{x_{t}x_{t}^{\prime }}\right\}
\right\vert =o_{p}\left( 1\right)
\end{equation*}%
whereas Lemma \ref{lem:maxineq2} and the expansion of $E\left\{ x_{t}\left( 
\frac{g}{n^{1/3}}\right) q_{t}\mathbf{1}_{t}\left( 0;\frac{g}{n^{1/3}}%
\right) \right\} $ as in (\ref{cubic exp}) imply that 
\begin{eqnarray*}
\sup_{\left\vert g\right\vert \leq C}\left\vert \frac{1}{\sqrt{n}}%
\sum_{t=1}^{n}\left\{ x_{t}\left( \frac{g}{n^{1/3}}\right) q_{t}\mathbf{1}%
_{t}\left( 0;\frac{g}{n^{1/3}}\right) \right\} \right\vert &=&O_{p}\left(
n^{-1/6}\right) \\
\sup_{\left\vert g\right\vert \leq C}\left\vert \frac{1}{\sqrt{n}}%
\sum_{t=1}^{n}\left( x_{t}\left( \frac{g}{n^{1/3}}\right) -\mathbf{x_{t}}%
\right) \varepsilon _{t}\right\vert &=&O_{p}\left( n^{-1/6}\right) \text{.}
\end{eqnarray*}%
Therefore%
\begin{equation}
\sup_{\left\Vert h\right\Vert ,\left\vert g\right\vert \leq C}\left\vert 
\mathbb{G}_{n}\left( h,g\right) -\widetilde{\mathbb{G}}_{n}\left( h,g\right)
\right\vert =o_{p}\left( 1\right) \text{,}  \label{diff}
\end{equation}%
where 
\begin{eqnarray*}
\widetilde{\mathbb{G}}_{n}\left( h,g\right) &=&\left\{ h^{\prime }\frac{1}{n}%
\sum_{t=1}^{n}\mathbf{x_{t}x_{t}^{\prime }}h-h^{\prime }\frac{2}{n^{1/2}}%
\sum_{t=1}^{n}\mathbf{x_{t}}\varepsilon _{t}\right\} \\
&&+\delta _{30}\left\{ \delta _{30}\sum_{t=1}^{n}q_{t}^{2}\mathbf{1}%
_{t}\left( 0;\frac{g}{n^{1/3}}\right) -2\sum_{t=1}^{n}q_{t}\varepsilon _{t}%
\mathbf{1}_{t}\left( 0;\frac{g}{n^{1/3}}\right) \right\} \\
&=&:\widetilde{\mathbb{G}}_{n}^{1}\left( h\right) +\widetilde{\mathbb{G}}%
_{n}^{2}\left( g\right) \text{.}
\end{eqnarray*}%
The consequence of $\left( \ref{diff}\right) $ is then that the minimizer of 
$\mathbb{G}_{n}\left( h,g\right) $ is asymptotically equivalent to that of $%
\widetilde{\mathbb{G}}_{n}\left( h,g\right) $. Thus, it suffices to show the
weak convergence of $\widetilde{\mathbb{G}}_{n}^{1}\left( h\right) $ and $%
\widetilde{\mathbb{G}}_{n}^{2}\left( g\right) $ and that 
\begin{equation*}
\widetilde{h}=:\arg \min_{h\in \mathbb{R}}\widetilde{\mathbb{G}}%
_{n}^{1}\left( h\right) ;\text{ \ \ \ \ }\widetilde{g}:=\underset{g\in 
\mathbb{R}}{\func{argmin}}\widetilde{\mathbb{G}}_{n}^{2}\left( g\right) 
\text{ }
\end{equation*}%
are $O_{p}\left( 1\right) $. The convergence of $\widetilde{\mathbb{G}}%
_{n}^{1}\left( h\right) $ and its minimization is straightforward since it
is a quadratic function of $h.$

Next, the first term of $\widetilde{\mathbb{G}}_{n}^{2}\left( g\right) $
converges to $3^{-1}\delta _{30}^{2}f\left( 0\right) \left\vert g\right\vert
^{3}$ uniformly in probability because Lemma \ref{lem:maxineq2}, i.e. (\ref%
{eq:maxineq2b}), implies the uniform law of large numbers and the Taylor
series expansion up to the third order yields 
\begin{equation}
nEq_{t}^{2}\mathbf{1}_{t}\left( 0;\frac{g}{n^{1/3}}\right) =n\int_{0}^{\frac{
g}{n^{1/3}}}q^{2}f\left( q\right) dq=n\frac{2f\left( \frac{\widetilde{g}}{
n^{1/3}}\right) }{3!}\left( \frac{g}{n^{1/3}}\right) ^{3}\rightarrow
3^{-1}f\left( 0\right) g^{3}\text{,}  \label{cubic exp}
\end{equation}
where $\widetilde{g}\in \left( 0,g\right) $. When $g<0$, it follows
similarly as in this case the derivative should be multiplied by $-1$, so
that the limit becomes $3^{-1}f\left( 0\right) \left\vert g\right\vert ^{3}$.

The second term in the definition of $\widetilde{\mathbb{G}}_{n}^{2}\left(
g\right) ,$ that is $-2\sum_{t=1}^{n}q_{t}\varepsilon _{t}\mathbf{1}%
_{t}\left( 0;\frac{g}{n^{1/3}}\right) $ converges weakly to $2\delta _{30}%
\sqrt{3^{-1}f\left( 0\right) \sigma _{\varepsilon }^{2}\left( 0\right) }%
W\left( g^{3}\right) $. To see this note that Lemma \ref{lem:maxineq2}, i.e.
(\ref{eq:maxineq2a}), yields the tightness of the process as explained in
Remark \ref{rem:tight}. For the finite dimensional convergence, we can
verify the conditions for martingale difference sequence CLT (e.g. Hall and
Heyde's (1980) Theorem 3.2). In particular, we need to show that for $u_{nt}=%
\sqrt{n}q_{t}\varepsilon _{t}\mathbf{1}_{t}\left( 0;\frac{g}{n^{1/3}}\right) 
$, 
\begin{eqnarray*}
&&\left( i\right) \ \ \ n^{-1/2}\max_{1\leq t\leq n}\left\vert
u_{nt}\right\vert \overset{p}{\longrightarrow }0 \\
&&\left( ii\right) \ \ \frac{1}{n}\sum_{t=1}^{n}u_{nt}^{2}\overset{p}{%
\longrightarrow }\frac{1}{3}E(\varepsilon _{t}^{2}|q_{t}=0)f(0)g^{3}
\end{eqnarray*}%
For $\left( i\right) $, note that $En^{-2}\max_{t}\left\vert
u_{nt}\right\vert ^{4}\leq n^{-1}E\left\vert u_{nt}\right\vert
^{4}=nEq_{t}^{4}\varepsilon _{t}^{4}\mathbf{1}_{t}\left( 0;\frac{g}{n^{1/3}}%
\right) \rightarrow 0$ as $n\rightarrow \infty $. For $\left( ii\right) $,
apply the same argument for the first term in $\widetilde{\mathbb{G}}%
_{n}^{2}\left( g\right) $ and an expansion similar to that in $\left( \ref%
{cubic exp}\right) $. We now characterize the covariance kernel. To that
end, we note that if $g_{1}$ and $g_{2}$ have different signs then the cross
product becomes zero and for $g_{2}>g_{1}>0$, similarly as with $\left( \ref%
{cubic exp}\right) $, we have that 
\begin{equation*}
nE\left( \varepsilon _{t}^{2}\left( q_{t}-\gamma _{0}\right) ^{2}\mathbf{1}%
\left\{ \frac{g_{1}}{n^{1/3}}<q_{t}<\frac{g_{2}}{n^{1/3}}\right\} \right) =%
\frac{f\left( \gamma _{0}\right) }{3}\sigma _{\varepsilon }^{2}\left( \gamma
_{0}\right) \left( g_{2}^{3}-g_{1}^{3}\right) +o\left( 1\right) \text{.}
\end{equation*}%
The cases for $g_{1}>g_{2}>0\ $or $g_{2}<g_{1}<0$ are similar and thus
omitted.

Finally, the covariance between $n^{-1/2}\sum_{t=1}^{n}\mathbf{x_{t}}%
\varepsilon _{t}$ and $\sum_{t=1}^{n}q_{t}\varepsilon _{t}\mathbf{1}%
_{t}\left( 0;g/n^{1/3}\right) $ vanishes for the same reasoning, yielding
the independence between $\widetilde{h}$ and $\widetilde{g}$ and thus the
asymptotic independence between $\widehat{\alpha }$ and the threshold
estimator $\widehat{\gamma }$.\hfill $\blacksquare $

\subsection{\textbf{Proof of Proposition \protect\ref{Th:QLR}}}

$\left. {}\right. $

Due to the asymptotic independence between $\widehat{\alpha }$ and $\widehat{
\gamma }$ in Theorem \ref{Th:AD_cube}, see (\ref{diff}) in its proof, we
have that 
\begin{equation*}
n\left( {{\mathbb{S}}}_{n}\left( \widehat{\alpha }\left( \gamma _{0}\right)
;\gamma _{0}\right) -{{\mathbb{S}}}_{n}\left( \widehat{\alpha };\widehat{
\gamma }\right) \right) =n\left( {{\mathbb{S}}}_{n}\left( \alpha _{0};\gamma
_{0}\right) -{{\mathbb{S}}}_{n}\left( \alpha _{0};\widehat{\gamma }\right)
\right) +o_{p}\left( 1\right) \text{,}
\end{equation*}%
which corresponds to $\min_{g}\widetilde{\mathbb{G}}_{n}^{2}\left( g\right) $
in the proof of Theorem \ref{Th:AD_cube} due to the reparameterization $%
g=n^{1/3}\left( \gamma -\gamma _{0}\right) $. It also shows that 
\begin{equation*}
\min_{g}\widetilde{\mathbb{G}}_{n}^{2}\left( g\right) \overset{d}{
\longrightarrow }f\left( \gamma _{0}\right) \min_{g\in \mathbb{R}}\left(
2\delta _{30}\sqrt{3^{-1}f\left( \gamma _{0}\right) \sigma _{\varepsilon
}^{2}\left( \gamma _{0}\right) }W\left( g^{3}\right) +3^{-1}\delta
_{30}^{2}f\left( \gamma _{0}\right) \left\vert g\right\vert ^{3}\right) .
\end{equation*}%
Finally, the desired result follows from applying the change of variables $%
g^{3}=3\phi \sigma _{\varepsilon }^{2}\left( \gamma _{0}\right) /\delta
_{30}^{2}f\left( \gamma _{0}\right) $ because of the distributional
equivalence $W\left( a^{2}g\right) =^{d}aW\left( g\right) $ (and $W\left(
s\right) =^{d}-W\left( s\right) $) and the fact that $\min_{x}g\left(
x\right) =-\max_{x}-g\left( x\right) $ for any function $g$. \hfill $%
\blacksquare $

\subsection{\textbf{Proof of Theorem \protect\ref{Th:Gamma_s}}}

$\left. {}\right. $

It is known that the distribution function of $\max_{g\in \mathbb{R}}\left(
2W\left( g\right) -\left\vert g\right\vert \right) $ is $F$, as in Hansen
(2000). Thus, under Assumption C, Propositions \ref{Th:QLR} and \ref%
{Prop:xihat} yield the conclusion, while under Assumption J, Theorem 2 of
Hansen (2000) verified the conclusion. \hfill $\blacksquare $

\subsection{\textbf{Proof of Theorem \protect\ref{Th:AD_cubeBoot}}}

$\left. {}\right. $

Recalling our definition of $\widehat{\alpha }^{\ast }$ and $\widehat{\gamma 
}^{\ast }$ in $\left( \ref{theta_star}\right) $, we begin by showing their
consistency and rate of convergence, which is given in Proposition \ref%
{Prop:ConsistencyBoot}.

We now discuss the asymptotic distribution of the bootstrap estimators. We
begin with part $\left( \mathbf{a}\right) $. We assume $\gamma _{0}=0$ to
simplify notation. Because the \textquotedblleft $\arg \max $%
\textquotedblright\ is continuous as mentioned in Theorem 2, it suffices to
examine the weak limit of 
\begin{eqnarray*}
\mathbb{G}_{n}^{\ast }\left( h,g\right) &=&n\left( \mathbb{S}_{n}^{\ast
}\left( \widetilde{\alpha }+\frac{h}{n^{1/2}},\frac{g}{n^{1/3}}\right) - 
\mathbb{S}_{n}^{\ast }\left( \widetilde{\alpha },0\right) \right) \\
&=&\sum_{t=1}^{n}\left\{ \left( \frac{h^{\prime }}{n^{1/2}}x_{t}\left( \frac{
g}{n^{1/3}}\right) +\widetilde{\delta }^{\prime }q_{t}\mathbf{1}_{t}\left(
0; \frac{g}{n^{1/3}}\right) +\varepsilon _{t}^{\ast }\right)
^{2}-\varepsilon _{t}^{\ast 2}\right\} \text{,}
\end{eqnarray*}%
where $\left\Vert h\right\Vert ,\left\vert g\right\vert \leq C$.

First, recall that $\widetilde{\delta }_{1}=O_{p}\left( n^{-1/2}\right) $
and $\widetilde{\delta }_{2}=O_{p}\left( n^{-1/2}\right) $ under Assumption
C and note that Lemma \ref{lem:maxineq2} and Lemma \ref{lem:max_boot} imply
that, uniformly in $\left\Vert h\right\Vert ,\left\vert g\right\vert <C$, 
\begin{eqnarray*}
\frac{1}{n}\sum_{t=1}^{n}\left\{ x_{t}\left( \frac{g}{n^{1/3}}\right)
x_{t}^{\prime }\left( \frac{g}{n^{1/3}}\right) -\mathbf{x}_{t}\mathbf{x}
_{t}^{\prime }\right\} &=&O_{p}\left( n^{-1/3}\right) \\
\frac{1}{n^{1/2}}\sum_{t=1}^{n}\left\{ x_{t}\left( \frac{g}{n^{1/3}}\right)
q_{t}\mathbf{1}_{t}\left( 0;\frac{g}{n^{1/3}}\right) \right\} &=&O_{p}\left(
n^{-1/6}\right) \\
E^{\ast }\left\Vert \frac{1}{n^{1/2}}\sum_{t=1}^{n}\left( x_{t}\left( \frac{
g }{n^{1/3}}\right) -\mathbf{x}_{t}\right) \varepsilon _{t}^{\ast
}\right\Vert ^{2} &=&O_{p}\left( n^{-1/3}\right) \text{.}
\end{eqnarray*}%
Thus, the latter implies that 
\begin{equation}
E^{\ast }\sup_{h,g\in \mathbb{R}}\left\vert \mathbb{G}_{n}^{\ast }\left(
h,g\right) -\widetilde{\mathbb{G}}_{n}^{\ast }\left( h,g\right) \right\vert
=O_{p}\left( n^{-1/6}\right) \text{,}  \label{diffboot}
\end{equation}%
where 
\begin{eqnarray*}
\widetilde{\mathbb{G}}_{n}^{\ast }\left( h,g\right) &=&\left\{ h^{\prime } 
\frac{1}{n}\sum_{t=1}^{n}\mathbf{x}_{t}\mathbf{x}_{t}^{\prime }h+h^{\prime } 
\frac{1}{n^{1/2}}\sum_{t=1}^{n}\mathbf{x}_{t}\varepsilon _{t}^{\ast }\right\}
\\
&&+\widetilde{\delta }_{3}\left\{ \widetilde{\delta }_{3}
\sum_{t=1}^{n}q_{t}^{2}\mathbf{1}_{t}\left( 0;\frac{g}{n^{1/3}}\right)
+\sum_{t=1}^{n}q_{t}\varepsilon _{t}^{\ast }\mathbf{1}_{t}\left( 0;\frac{g}{
n^{1/3}}\right) \right\} \\
&=&:\widetilde{\mathbb{G}}_{1n}^{\ast }\left( h\right) +\widetilde{\mathbb{G}
}_{2n}^{\ast }\left( g\right) \text{.}
\end{eqnarray*}%
The consequence of $\left( \ref{diffboot}\right) $ is then that the
minimizer of $\mathbb{G}_{n}^{\ast }\left( h,g\right) $ is asymptotically
equivalent to that of $\widetilde{\mathbb{G}}_{n}^{\ast }\left( h,g\right) $%
. Thus, it suffices to show the weak convergence of $\widetilde{\mathbb{G}}%
_{1n}^{\ast }\left( h\right) $ and $\widetilde{\mathbb{G}}_{2n}^{\ast
}\left( g\right) $ and that 
\begin{equation*}
\widetilde{h}=:\arg \max_{h\in \mathbb{R}}\widetilde{\mathbb{G}}_{1n}^{\ast
}\left( h\right) ;\text{ \ \ \ \ }\widetilde{g}=:\arg \max_{g\in \mathbb{R}}%
\widetilde{\mathbb{G}}_{2n}^{\ast }\left( g\right) \text{ }
\end{equation*}%
are $O_{p^{\ast }}\left( 1\right) $. The convergence of $\widetilde{\mathbb{%
\ \ G }}_{1n}^{\ast }\left( h\right) $ and its minimization follows by
standard arguments as it is a quadratic function of $h$ so that it suffices
to examine $\widetilde{\mathbb{G}}_{2n}^{\ast }\left( g\right) $ and it
minimum.

Turning to the second term in the definition of $\widetilde{\mathbb{G}}%
_{2n}^{\ast }\left( g\right) ,$ we show that it converges to $2\delta _{30}%
\sqrt{3^{-1}f\left( 0\right) \sigma _{\varepsilon }^{2}\left( 0\right) }%
W\left( g^{3}\right) $ weakly (in probability). To this end, note that Lemma %
\ref{lem:max_boot}'s, and the Remark 4 that follows, yields the tightness of
the process as explained in Remark \ref{rem:tight}. For the finite
dimensional convergence, it follows by standard arguments as 
\begin{equation*}
E^{\ast }\left( \sum_{t=1}^{n}q_{t}\varepsilon _{t}^{\ast }\mathbf{1}%
_{t}\left( 0;\frac{g}{n^{1/3}}\right) \right) ^{2}=\sum_{t=1}^{n}q_{t}^{2}%
\widehat{\varepsilon }_{t}^{2}\mathbf{1}_{t}\left( 0;\frac{g}{n^{1/3}}\right)
\end{equation*}%
which converges in probability to $3^{-1}f\left( 0\right) \sigma
_{\varepsilon }^{2}\left( 0\right) g^{3}$ and the Lindeberg's condition
follows easily.

Part $\left( \mathbf{b}\right) $ is also proved similarly and thus omitted
for the sake of space. \hfill $\blacksquare $

\subsection{\textbf{Proof of Theorem \protect\ref{Th:QLRBoot}}}

$\left. {}\right. $

This is a direct consequence of Theorem \ref{Th:AD_cubeBoot} and Proposition %
\ref{Prop:xihatstar} and the same arguments as the proof of Theorem \ref%
{Th:QLR}. \hfill $\blacksquare $


\newpage


\renewcommand\thesection{B-\arabic{section}} \setcounter{section}{0}

\renewcommand\thepage{A-\arabic{page}} \setcounter{page}{1}

\noindent {\Large \textbf{Online Supplement to \textquotedblleft Robust
Inference in Threshold Regression Models\textquotedblright }}

\begin{center}
by Javier Hidalgo, Jungyoon Lee, and Myung Hwan Seo \vspace{0.2in}

\parbox{5in}{
This supplement contains  more numerical results for Section \ref{sec:MC} and the remaining proofs of main theorems and supporting lemmas.}
\end{center}


\section{Table 4 for Monte Carlo study in Section \protect\ref{sec:MC}}

\begin{table}[htbp]
\caption{Monte Carlo size of test $H_{0}:\protect\gamma =\protect\gamma _{0}$
and coverage probability of confidence intervals of $\protect\gamma _{0}$,
model A: $q_{t}\neq x_{t}$, homoscedastic error, $\protect\varphi=0$ }%
{\small \centering
\setlength{\tabcolsep}{1pt} }
\par
{\small \ 
\begin{tabular}{cc|r|rrr|c|rrr|rrr}
\hline
&  & \multicolumn{4}{|c|}{Size} & \multicolumn{7}{|c}{Coverage Probability}
\\ \hline
&  & $\gamma _{0}$ & \multicolumn{3}{|c|}{median of $q_{t}$(2)} & $\gamma
_{0}$ & \multicolumn{3}{|c}{median of $q_{t}$(2)} & \multicolumn{3}{|c}{
third quart. of $q_{t}$(2.674)} \\ 
$\delta$ &  & $s$\textbackslash {}$n$ & 100 & 250 & 500 & 
\multicolumn{1}{|r|}{$\zeta $\textbackslash{}$n$} & 100 & 250 & 500 & 100 & 
250 & 500 \\ \hline
$\sqrt{10}/4$ & Asym & 0.01 & 0.0033 & 0.0032 & 0.002 & \multicolumn{1}{|r|}{
0.9} & 0.969 & 0.976 & 0.971 & 0.969 & 0.979 & 0.975 \\ 
(=0.7906) &  & 0.05 & 0.0133 & 0.0109 & 0.0093 & \multicolumn{1}{|r|}{0.95}
& 0.987 & 0.988 & 0.987 & 0.98 & 0.991 & 0.986 \\ 
&  & 0.1 & 0.0266 & 0.0219 & 0.0203 & \multicolumn{1}{|r|}{0.99} & 0.999 & 
0.998 & 0.998 & 0.998 & 0.999 & 0.997 \\ 
& B/rap & 0.01 & 0.0104 & 0.0173 & 0.0114 & \multicolumn{1}{|r|}{0.9} & 0.837
& 0.859 & 0.836 & 0.839 & 0.848 & 0.843 \\ 
&  & 0.05 & 0.0691 & 0.0713 & 0.0674 & \multicolumn{1}{|r|}{0.95} & 0.87 & 
0.901 & 0.868 & 0.87 & 0.883 & 0.875 \\ 
&  & 0.1 & 0.1353 & 0.1358 & 0.1276 & \multicolumn{1}{|r|}{0.99} & 0.935 & 
0.936 & 0.925 & 0.926 & 0.933 & 0.928 \\ \hline
0.25 & Asym & 0.01 & 0.016 & 0.0074 & 0.0075 & \multicolumn{1}{|r|}{0.9} & 
0.88 & 0.909 & 0.93 & 0.879 & 0.925 & 0.931 \\ 
&  & 0.05 & 0.0599 & 0.0402 & 0.0322 & \multicolumn{1}{|r|}{0.95} & 0.938 & 
0.95 & 0.972 & 0.927 & 0.958 & 0.961 \\ 
&  & 0.1 & 0.1102 & 0.076 & 0.0648 & \multicolumn{1}{|r|}{0.99} & 0.985 & 
0.992 & 0.993 & 0.982 & 0.994 & 0.984 \\ 
& B/rap & 0.01 & 0.0146 & 0.0075 & 0.0121 & \multicolumn{1}{|r|}{0.9} & 0.873
& 0.876 & 0.894 & 0.851 & 0.896 & 0.897 \\ 
&  & 0.05 & 0.0585 & 0.0518 & 0.0563 & \multicolumn{1}{|r|}{0.95} & 0.934 & 
0.93 & 0.939 & 0.916 & 0.949 & 0.943 \\ 
&  & 0.1 & 0.1123 & 0.1024 & 0.1117 & \multicolumn{1}{|r|}{0.99} & 0.984 & 
0.986 & 0.992 & 0.975 & 0.987 & 0.981 \\ \hline
\end{tabular}
}{\small \ } \raggedright{\small \ Size results for test of $H_{0}:\gamma
=\gamma _{0}$ with nominal size $s$ based on Hansen $(2000)$'s asymptotic
distribution(Asym), and bootstrap(B/rap). Coverage probability results for $%
\gamma _{0}$ with asymptotic confidence interval based on Hansen $(2000)$
and grid bootstrap confidence interval, with nominal confidence level $\zeta 
$. }
\end{table}

In Table 4, we report Monte Carlo size and coverage probability results for $%
\gamma$ when $\varphi=0$ with $\delta$ fixed at $\sqrt{10}/4=0.7906$ and $%
0.25$ in setting A ($q_t\neq x_t$) with homoscedastic error. In Table 2 of
Hansen (2000), Monte Carlo coverage probability of his asymptotic confidence
interval is reported in a similar setup. He found that coverage rates
increase with larger $\delta$ and larger $n$, significantly above the
nominal rate. Similar results are reported for Hansen's asymptotic method in
our Table 4: for $\delta=0.7906$, under-sizing of test $H_0: \gamma=\gamma_0$
and over-coverage of confidence intervals for $\gamma$ are severe for all $n$%
. For $\delta=0.25$, the under-sizing and over-coverage become an issue for
larger $n=250,500$. On the other hand, our bootstrap method for the case $%
\delta=0.7906$ led to some over-sizing and severe under-coverage for all $n$%
. For $\delta=0.25$, results were more satisfactory, with the Monte Carlo
size being close to the nominal size for all $n$, and the coverage
probability approaching the nominal level with larger $n$.

\section{Proofs of Propositions 3 and 4 and Proposition \protect\ref%
{Prop:ConsistencyBoot}}

\subsection{\textbf{Proof of Proposition \protect\ref{Prop:xihat}}}

$\left. {}\right. $

Recalling our notation in $\left( \ref{x_not}\right) $ and that $\delta
_{1}+\delta _{3}\gamma _{0}=0$ and $\delta _{2}=0$ under Assumption C, we
then have that 
\begin{equation}
\widehat{\delta }^{\prime }x_{t}=\left( \widehat{\delta }_{1}-\delta
_{1}\right) +\widehat{\delta }_{2}^{\prime }x_{2t}+\left( \widehat{\delta }%
_{3}-\delta _{3}\right) q_{t}+\delta _{3}\left( q_{t}-\gamma _{0}\right) 
\text{.}  \label{deltax}
\end{equation}%
Because we can rename $q_{t}-\gamma _{0}$ as $q_{t}$, we shall assume
without loss of generality that $\gamma _{0}=0$ so that $\delta _{1}=0$.

Consider the case where $\widehat{\gamma }>0$. The proof when $\widehat{
\gamma }<0$ is analogous and thus it is omitted. By construction, we have
that%
\begin{equation*}
\widehat{\varepsilon }_{t}=\varepsilon _{t}+\left( \widehat{\beta }-\beta
\right) ^{\prime }x_{t}+\left( \widehat{\delta }-\delta \right) ^{\prime
}x_{t}\boldsymbol{1}_{t}\left( \widehat{\gamma }\right) +\delta _{3}q_{t}%
\boldsymbol{1}_{t}\left( 0;\widehat{\gamma }\right) \text{.}
\end{equation*}%
Because $\left( \delta _{1},\delta _{2}^{\prime }\right) =0$ and $\widehat{
\beta }-\beta =O_{p}\left( n^{-1/2}\right) $, $\widehat{\delta }-\delta
=O_{p}\left( n^{-1/2}\right) $ and $\widehat{\gamma }=O_{p}\left(
n^{-1/3}\right) $, we obtain that%
\begin{eqnarray}
\widehat{\varepsilon }_{t}^{2} &=&\varepsilon _{t}^{2}+O_{p}\left(
n^{-1}\right) +\left( \delta _{3}q_{t}\right) ^{2}\boldsymbol{1}_{t}\left( 0;%
\widehat{\gamma }\right) +2\delta _{3}\varepsilon _{t}q_{t}\boldsymbol{1}%
_{t}\left( 0;\widehat{\gamma }\right)  \notag \\
&&+O_{p}\left( n^{-1/2}\right) \varepsilon _{t}x_{t}\left( 1+\boldsymbol{1}%
_{t}\left( \widehat{\gamma }\right) \right) +2\delta _{3}\left\Vert
x_{t}\right\Vert q_{t}\boldsymbol{1}_{t}\left( 0;\widehat{\gamma }\right)
O_{p}\left( n^{-1/2}\right)  \notag \\
&=&\varepsilon _{t}^{2}+O_{p}\left( n^{-1/2}\right) \left\Vert
x_{t}\right\Vert \varepsilon _{t}+2\delta _{3}\varepsilon _{t}q_{t}%
\boldsymbol{1}_{t}\left( 0;\widehat{\gamma }\right) +\left\Vert
x_{t}\right\Vert O_{p}\left( n^{-2/3}\right) \text{.}  \label{eps}
\end{eqnarray}%
Now $\left( \ref{deltax}\right) $ implies that $\left( \widehat{\delta }%
^{\prime }x_{t}\right) ^{2}=\delta _{3}^{2}q_{t}^{2}+O_{p}\left(
n^{-1/2}\right) \delta _{3}\left\Vert x_{t}\right\Vert q_{t}+O_{p}\left(
n^{-1}\right) $. So, by Lemma \ref{lem:4proposition} and \ref%
{lem:4proposition_1} and by the standard arguments using $na^{3}\rightarrow
\infty $, we conclude that the behaviour of numerator of $\left( \ref{xhihat}%
\right) $ is that of 
\begin{equation*}
\frac{1}{na^{3}}\sum_{t=1}^{n}\delta _{3}^{2}q_{t}^{2}\varepsilon
_{t}^{2}K\left( \frac{q_{t}-\widehat{\gamma }}{a}\right) =\kappa _{2}\delta
_{3}^{2}a^{2}\sigma ^{2}\left( 0\right) f\left( 0\right) \left(
1+o_{p}\left( 1\right) \right)
\end{equation*}%
when $\kappa _{2}\neq 0$, that is we do not assume higher-order kernels.
Observe that $g_{0}\left( q\right) \ $in Lemma \ref{lem:4proposition}
corresponds to $\sigma ^{2}\left( q\right) $. More specifically, the
contribution due to other terms in $\left( \ref{eps}\right) $ are indeed
negligible by Lemma \ref{lem:4proposition_1}.

Similarly, the leading term in the denominator in $\left( \ref{xhihat}%
\right) $ is%
\begin{equation*}
\frac{1}{na^{3}}\sum_{t=1}^{n}\left( \widehat{\delta }^{\prime }x_{t}\right)
^{2}K\left( \frac{q_{t}-\widehat{\gamma }}{a}\right) =\kappa _{2}\delta
_{3}^{2}a^{2}f\left( 0\right) \left( 1+o_{p}\left( 1\right) \right) \text{.}
\end{equation*}%
So, the convergence in $\left( \ref{xhihat}\right) $ follows from the last
two displayed expressions. Finally, it is standard to show that $\mathbb{S}%
_{n}(\widehat{\theta })-\sigma ^{2}=o_{p}\left( 1\right) $. This completes
the proof of the proposition.\hfill $\blacksquare $

\subsection{\textbf{Proof of Proposition \protect\ref{Prop:xihatstar}}}

$\left. {}\right. $

As before we assume $\gamma_0 = 0 $. We show this proposition under
Assumption C and the case with Assumption J is similar and thus omitted. Let 
$\widehat{\gamma }^{\ast }>0$. The case when $\widehat{\gamma }^{\ast }<0 $
is analogous and thus omitted. We shall examine the behaviour of the
numerator of $\left( \ref{xhihatBoot}\right) $, that of its denominator
being similarly handled. By construction, 
\begin{equation*}
\widehat{\varepsilon }_{t}^{\ast }=\varepsilon _{t}^{\ast }+\left( \widehat{
\beta }^{\ast }-\widetilde{\beta }\right) ^{\prime }x_{t}+\left( \widehat{
\delta }^{\ast }-\widetilde{\delta }\right) ^{\prime }x_{t}\boldsymbol{1}%
_{t}\left( \widehat{\gamma }^{\ast }\right) +\left( \widetilde{\delta }_{1}+%
\widetilde{\delta }_{3}q_{t}\right) \boldsymbol{1}_{t}\left( 0;\widehat{
\gamma }^{\ast }\right) \text{.}
\end{equation*}%
Recall that when the constraint given in $\left( \ref{eq:conti}\right) $
holds true $\widetilde{\delta }_{2}$ and $\widetilde{\delta }_{1}$ are both $%
O_{p}\left( n^{-1/2}\right) $. On the other hand Proposition \ref%
{Prop:ConsistencyBoot} yields that $\widehat{\beta }^{\ast }-\widetilde{%
\beta }=O_{p^{\ast }}\left( n^{-1/2}\right) $, $\widehat{\delta }^{\ast }-%
\widetilde{\delta }=O_{p^{\ast }}\left( n^{-1/2}\right) $ and $\widehat{
\gamma }^{\ast }=O_{p^{\ast }}\left( n^{-1/3}\right) $. Then, $\left( 
\widehat{\delta }^{\ast \prime }x_{t}\right) ^{2}=\widetilde{\delta }%
^{\prime 2}x_{t}^{2}+O_{p^{\ast }}\left( n^{-1/2}\right) \widetilde{\delta }%
^{\prime }x_{t}q_{t}+O_{p^{\ast }}\left( n^{-1}\right) $. And, proceeding as
we did in the proof of Proposition \ref{Prop:xihat}, we easily deduce that%
\begin{equation}
\widehat{\varepsilon }_{t}^{\ast 2}=\varepsilon _{t}^{\ast 2}+O_{p^{\ast
}}\left( n^{-1/2}\right) x_{t}\varepsilon _{t}^{\ast }+2\widetilde{\delta }
_{3}\varepsilon _{t}^{\ast }q_{t}\boldsymbol{1}_{t}\left( 0;\widehat{\gamma }
^{\ast }\right) +x_{t}O_{p^{\ast }}\left( n^{-2/3}\right) \text{.}
\label{epsBoot}
\end{equation}
By obvious arguments and those in $\left( \ref{eps_1Boot}\right) $, it
suffices to examine the behaviour of 
\begin{equation*}
\frac{1}{na}\sum_{t=1}^{n}\left( \widetilde{\delta }^{\prime }x_{t}\right)
^{2}\varepsilon _{t}^{\ast 2}K\left( \frac{q_{t}-\widehat{\gamma }^{\ast }}{
a }\right) \text{.}
\end{equation*}
Now, because $\widetilde{\delta }_{2}$ and $\widetilde{\delta }_{1}$ are
both $O_{p}\left( n^{-1/2}\right) $ when $\left( \ref{eq:conti}\right) $
holds true the behaviour of the last displayed expression is governed by 
\begin{equation*}
\frac{1}{na}\sum_{t=1}^{n}\widetilde{\delta }_{3}^{2}q_{t}^{2}\varepsilon
_{t}^{\ast 2}K\left( \frac{q_{t}-\widehat{\gamma }^{\ast }}{a}\right)
\end{equation*}
which is $\kappa _{2}\delta _{30}^{2}a^{2}E^{\ast }\left[ \varepsilon
_{t}^{\ast 2}\mid q_{t}=\gamma _{0}\right] f\left( 0\right) \left(
1+o_{p^{\ast }}\left( 1\right) \right) $ by Lemma \ref{lem:4propositionBoot}
when $\kappa _{2}\neq 0$, that is we do not assume higher-order kernels.
Notice that, by standard results, the contribution due to other terms in $%
\left( \ref{epsBoot}\right) $ are indeed negligible by Lemma \ref%
{lem:4proposition_1Boot}.

Likewise the denominator in $\left( \ref{xhihatBoot}\right) $, is 
\begin{equation*}
\frac{1}{na}\sum_{t=1}^{n}\left( \widetilde{\delta }^{\prime }x_{t}\right)
^{2}K\left( \frac{q_{t}-\widehat{\gamma }^{\ast }}{a}\right) =\kappa
_{2}\delta _{30}^{2}a^{2}f\left( 0\right) \left( 1+o_{p^{\ast }}\left(
1\right) \right) \text{.}
\end{equation*}
So, the convergence in $\left( \ref{xhihatBoot}\right) $ follows from the
last two displayed expressions. Finally, it is standard that $\mathbb{S}
_{n}( \widehat{\theta }^{\ast })-\sigma ^{2}=o_{p^{\ast }}\left( 1\right) $.
This completes the proof of the proposition.\hfill $\blacksquare $

\subsection{Convergence Rate of Bootstrap Estimator}

\begin{proposition}
\label{Prop:ConsistencyBoot}Suppose that Assumptions Z and Q hold. Then, 
\newline
$\left( \mathbf{a}\right) $ Under Assumption C, 
\begin{equation*}
\widehat{\alpha }^{\ast }-\widehat{\alpha }=O_{p^{\ast }}\left(
n^{-1/2}\right) \ \ \ \text{and \ \ }\widehat{\gamma }^{\ast }-\gamma
_{0}=O_{p^{\ast }}\left( n^{-1/3}\right) \text{.}
\end{equation*}
$\left( \mathbf{b}\right) $ Under Assumption J, 
\begin{equation*}
\widehat{\alpha }^{\ast }-\widehat{\alpha }=O_{p^{\ast }}\left(
n^{-1/2}\right) \ \ \ \text{and \ \ }\widehat{\gamma }^{\ast }-\gamma
_{0}=O_{p^{\ast }}\left( n^{2\varphi -1}\right) \text{.}
\end{equation*}
\end{proposition}

\noindent \textbf{Proof of Proposition\ \ref{Prop:ConsistencyBoot}} \ \
Assuming without loss of generality that $\gamma \geq \widehat{\gamma }%
=\gamma _{0}$ and abbreviating $\widehat{\psi }-\psi $ by $\overline{\psi }$
for any parameter $\psi $, proceeding as in Proposition \ref{Consistency},
we obtain that 
\begin{eqnarray*}
{{\mathbb{S}}}_{n}^{\ast }\left( \theta \right) -{{\mathbb{S}}}_{n}^{\ast
}\left( \widehat{\theta }\right) &=&\frac{1}{n}\sum_{t=1}^{n}\left\{ \left( 
\overline{\beta }^{\prime }x_{t}+\overline{\delta }^{\prime }x_{t}\mathbf{1}
_{t}\left( \gamma \right) +\widehat{\delta }^{\prime }x_{t}\mathbf{1}
_{t}\left( \widehat{\gamma };\gamma \right) +\varepsilon _{t}^{\ast }\right)
^{2}-\varepsilon _{t}^{\ast 2}\right\} \\
&=&\widehat{\mathbb{A}}_{n1}\left( \theta \right) +\widehat{\mathbb{A}}
_{n2}\left( \theta \right) +\widehat{\mathbb{A}}_{n3}\left( \theta \right) + 
\mathbb{B}_{n1}^{\ast }\left( \theta \right) +\mathbb{B}_{n2}^{\ast }\left(
\theta \right) +\mathbb{B}_{n3}^{\ast }\left( \theta \right) \text{,}
\end{eqnarray*}%
where 
\begin{eqnarray*}
\widehat{\mathbb{A}}_{n1}\left( \theta \right) &=&\overline{\upsilon }
^{\prime }M_{n}^{x}\left( \gamma \right) \overline{\upsilon };~\text{ \ } 
\widehat{\mathbb{A}}_{n2}\left( \theta \right) =\overline{\beta }^{\prime
}M_{n}^{x}\left( -\infty ;\widehat{\gamma }\right) \overline{\beta } \\
\widehat{\mathbb{A}}_{n3}\left( \theta \right) &=&\left( \overline{\beta }+ 
\widehat{\delta }\right) ^{\prime }M_{n}^{x}\left( \widehat{\gamma };\gamma
\right) \left( \overline{\beta }+\widehat{\delta }\right) \\
\mathbb{B}_{n1}^{\ast }\left( \theta \right) &=&\overline{\upsilon }^{\prime
}\frac{2}{n}\sum_{t=1}^{n}x_{t}\varepsilon _{t}^{\ast }\mathbf{1}_{t}\left(
\gamma \right) ;~\ \ \mathbb{B}_{n2}^{\ast }\left( \theta \right) =\overline{
\beta }^{\prime }\frac{2}{n}\sum_{t=1}^{n}x_{t}\varepsilon _{t}^{\ast } 
\mathbf{1}_{t}\left( -\infty ;\widehat{\gamma }\right) \\
\mathbb{B}_{n3}^{\ast }\left( \theta \right) &=&\left( \overline{\beta }+ 
\widehat{\delta }\right) ^{\prime }\frac{2}{n}\sum_{t=1}^{n}x_{t}\varepsilon
_{t}^{\ast }\mathbf{1}_{t}\left( \widehat{\gamma };\gamma \right) \text{,}
\end{eqnarray*}%
where, in what follows, for a generic sequence $\left\{ z_{t}\right\} _{t\in 
\mathbb{Z}}$ we employ the notation $M_{n}^{z}\left( \gamma \right) =\frac{1%
}{n}\sum_{t=1}^{n}z_{t}z_{t}^{\prime }\mathbf{1}_{t}\left( \gamma \right) $
and $M_{n}^{z}\left( \gamma _{1};\gamma _{2}\right) =\frac{1}{n}%
\sum_{t=1}^{n}z_{t}z_{t}^{\prime }\mathbf{1}_{t}\left( \gamma _{1};\gamma
_{2}\right) $. It is also worth recalling that for $n$ large enough $%
0<\sup_{\gamma \in \Gamma }\left\Vert M_{n}^{x}\left( \gamma \right)
\right\Vert =H_{n}$ and $0<\sup_{\gamma _{1}<\gamma _{2}}\left\Vert
M_{n}^{x}\left( \gamma _{1};\gamma _{2}\right) \right\Vert =H_{n}$, where in
what follows $H_{n}$ denotes a sequence of strictly positive $O_{p}\left(
1\right) $ random variables. Finally as we have in the proof of Proposition %
\ref{Consistency}, because $E\left( x_{t}x_{t}^{\prime }\mathbf{1}_{t}\left(
\gamma \right) \right) $ and $E\left( x_{t}x_{t}^{\prime }\mathbf{1}%
_{t}\left( 0;\gamma \right) \right) $ are strictly finite positive definite
matrices, $M_{n}^{x}\left( -\infty ;\gamma \right) -E\left(
x_{t}x_{t}^{\prime }\mathbf{1}_{t}\left( -\infty ;\gamma \right) \right)
=O_{p}\left( n^{-1/2}\right) $ and $M_{n}^{x}\left( \gamma \right) -E\left(
x_{t}x_{t}^{\prime }\mathbf{1}_{t}\left( \gamma \right) \right) =O_{p}\left(
n^{-1/2}\right) $ uniformly in $\gamma \in \Gamma $, we have that 
\begin{eqnarray}
C_{1}H_{n} \leq \frac{\widehat{\mathbb{A}}_{n2}\left( \theta \right) }{
\left( \overline{\beta }_{1},\overline{\beta }_{2}^{\prime }\right)
M_{n}^{x_{1}}\left( -\infty ;0\right) \left( \overline{\beta }_{1},\overline{
\beta }_{2}^{\prime }\right) ^{\prime }+\overline{\beta }_{3}^{2}M_{n}^{q} 
\mathbf{1}_{t}\left( -\infty ;0\right) }&\leq& C_{2}H_{n}  \notag \\
C_{1}H_{n} \leq \frac{\widehat{\mathbb{A}}_{n3}\left( \theta \right) }{
\left( \overline{\tau }_{1},\overline{\tau }_{2}^{\prime }\right)
M_{n}^{x_{1}}\left( 0;\gamma \right) \left( \overline{\tau }_{1},\overline{
\tau }_{2}^{\prime }\right) ^{\prime }+\overline{\tau }_{3}^{2}M_{n}^{q}
\left( 0;\gamma \right) }&\leq& C_{2}H_{n}\text{,}  \label{definboot}
\end{eqnarray}%
where $\overline{\tau }=\left( \widehat{\beta }-\beta \right) +\widehat{
\delta }$. The motivation is that we employ in the proof of Proposition \ref%
{Consistency}, after observing that Proposition \ref{Consistency} implies
that $\widehat{\gamma }-\gamma _{0}=O_{p}\left( n^{-1/3}\right) $ and Lemma %
\ref{lem:maxineq2} that uniformly in $\gamma _{1}<\gamma _{2}\in \Gamma $, 
\begin{equation*}
M_{n}^{x}\left( \gamma _{1};\gamma _{2}\right) -Ex_{t}x_{t}^{\prime }\mathbf{%
\ \ \ 1}_{t}\left( \gamma _{1};\gamma _{2}\right) =O_{p}\left(
n^{-1/2}\right)
\end{equation*}%
together with the fact that $M_{n}^{x}\left( -\infty ;\widehat{\gamma }%
\right) =M_{n}^{x}\left( -\infty ;\gamma _{0}\right) +M_{n}^{x}\left( \gamma
_{0};\widehat{\gamma }\right) $.

\textbf{Consistency}. We begin with part $\left( \mathbf{a}\right) $.
Arguing as in the proof of Proposition \ref{Consistency}, it suffices to
show that 
\begin{equation}
\Pr \left. ^{\ast }\right. \left\{ \inf_{\left\Vert \overline{\theta }
\right\Vert >\eta }\sum_{\ell =1}^{3}\widehat{\mathbb{A}}_{n\ell }\left(
\theta \right) +\mathbb{B}_{n\ell }^{\ast }\left( \theta \right) \leq
0\right\} \leq \epsilon H_{n}\text{.}  \label{consiboot_2}
\end{equation}

First, when $\left\Vert \overline{\theta }\right\Vert >\eta $, it implies
that either $\left( \mathbf{i}\right) $ $\left\Vert \overline{\gamma }%
\right\Vert >\eta /2$ or $\left( \mathbf{ii}\right) $ $\left\Vert \overline{
\beta }\right\Vert ,\left\Vert \overline{\upsilon }\right\Vert >\eta /2$.
When $\left( \mathbf{ii}\right) $ holds true, it is clear that 
\begin{equation}
\inf_{\left\Vert \overline{\upsilon }\right\Vert >\eta /2}\widehat{\mathbb{A}%
}_{n\ell }\left( \theta \right) >\eta ^{2}H_{n}~\ \ \ \ \ \ell =1,2
\label{ineqboot_2}
\end{equation}%
whereas when $\left( \mathbf{i}\right) $ holds true, we obtain that%
\begin{equation}
\inf_{\left\Vert \gamma \right\Vert >\eta /2}M_{n}^{x}\left( \widehat{\gamma 
};\gamma \right) >\eta H_{n}\text{,}  \label{ineqboot_3}
\end{equation}%
because $E\left( x_{t}x_{t}^{\prime }\mathbf{1}_{t}\left( \gamma \right)
\right) $ and $E\left( x_{t}x_{t}^{\prime }\mathbf{1}_{t}\left( 0;\gamma
\right) \right) $ are strictly positive definite matrices, since say $%
E\left( x_{t}x_{t}^{\prime }\mathbf{1}_{t}\left( 0;\gamma \right) \right)
-E\left( x_{t}x_{t}^{\prime }\mathbf{1}_{t}\left( 0;\eta /4\right) \right) $
is a positive definite matrix when $\left\Vert \overline{\gamma }\right\Vert
>\eta /2$, $M_{n}^{x}\left( \widehat{\gamma };\gamma \right) =E\left(
x_{t}x_{t}^{\prime }\mathbf{1}_{t}\left( 0;\gamma \right) \right) \left(
1+o_{p}\left( 1\right) \right) \ $and $\widehat{\mathbb{A}}_{n\ell }\left(
\theta \right) -E\left( \mathbb{A}_{n\ell }\left( \theta \right) \right)
=o_{p}\left( 1\right) $. Recall that $E\left( a^{\prime }x_{t}\mathbf{1}%
_{t}\left( 0;\eta \right) \right) >\eta \min_{q\in \left( 0,\eta \right)
}f\left( q\right) E\left( a^{\prime }x_{t}\right) $. So, $\left( \ref%
{ineqboot_2}\right) $ and $\left( \ref{ineqboot_3}\right) $ implies that 
\begin{equation}
\inf_{\left\Vert \overline{\theta }\right\Vert >\eta }\sum_{\ell =1}^{3}%
\widehat{\mathbb{A}}_{n\ell }\left( \theta \right) >\eta ^{2}H_{n}\text{. }
\label{ineqboot_7}
\end{equation}%
On the other hand, Lemma \ref{lem:max_boot} implies that 
\begin{equation}
E^{\ast }\left( \sup_{\gamma }\left\Vert \frac{1}{n^{1/2}}%
\sum_{t=1}^{n}x_{t}\varepsilon _{t}^{\ast }\mathbf{1}_{t}\left( \gamma
\right) \right\Vert \right) ^{2}+E^{\ast }\left( \sup_{\gamma }\left\Vert 
\frac{1}{n^{1/2}}\sum_{t=1}^{n}x_{t}\varepsilon _{t}^{\ast }\mathbf{1}%
_{t}\left( -\infty ;\gamma \right) \right\Vert \right) ^{2}=H_{n}\text{,}
\label{ineqboot_1}
\end{equation}%
so that 
\begin{equation}
E^{\ast }\sup_{\left\Vert \overline{\theta }\right\Vert >\eta /2}\left\Vert 
\mathbb{B}_{n\ell }^{\ast }\left( \theta \right) \right\Vert
=n^{-1/2}H_{n}~\ \ \ \ \ \ell =1,2,3\text{.}  \label{ineqboot_4}
\end{equation}

Thus $\left( \ref{ineqboot_7}\right) $ and $\left( \ref{ineqboot_4}\right) $
yields that $\widehat{\theta }^{\ast }-\widehat{\theta }=o_{p^{\ast }}\left(
1\right) $ because the left side of $\left( \ref{consiboot_2}\right) $ is
bounded by 
\begin{equation*}
\Pr \left. ^{\ast }\right. \left\{ \inf_{\left\Vert \overline{\theta }
\right\Vert >\eta }\sum_{\ell =1}^{3}\widehat{\mathbb{A}}_{n\ell }\left(
\theta \right) \leq \sup_{\left\Vert \overline{\theta }\right\Vert >\eta
}\left\Vert \sum_{\ell =1}^{3}\mathbb{B}_{n\ell }^{\ast }\left( \theta
\right) \right\Vert \right\}
\end{equation*}%
and then Markov's inequality. This concludes the consistency proof.

\textbf{Convergence rate}. To that end, we shall show that for some $C>0$
large enough and $\epsilon >0$, 
\begin{equation}
\Pr \left. ^{\ast }\right. \left\{ \inf_{\frac{C}{n^{1/2}}<\left\Vert 
\overline{\upsilon }\right\Vert ;\left\Vert \overline{\beta }\right\Vert
<\eta ;\frac{C}{n^{1/3}}<\left\Vert \gamma \right\Vert <\eta }\sum_{\ell
=1}^{3}\widehat{\mathbb{A}}_{n\ell }\left( \theta \right) +\mathbb{B}_{n\ell
}^{\ast }\left( \theta \right) \leq 0\right\} <\epsilon H_{n}\text{.}
\label{rateboot_1}
\end{equation}%
To that end, we shall first examine 
\begin{equation*}
\Pr \left. ^{\ast }\right. \left\{ \inf_{\Xi _{j}\left( \upsilon \right)
;\Xi _{j}\left( \beta \right) ;\Xi _{k}\left( \gamma \right) }\sum_{\ell
=1}^{3}\widehat{\mathbb{A}}_{n\ell }\left( \theta \right) +\mathbb{B}_{n\ell
}^{\ast }\left( \theta \right) \leq 0\right\}
\end{equation*}%
where for some $j=1,...,\log _{2}\frac{\eta }{C}n^{1/2}$ and $k=1,...,\log
_{2}\frac{\eta }{C}n^{1/3}$, and $\Xi _{j}\left( \upsilon \right) $ and $\Xi
_{k}\left( \gamma \right) $ are defined similarly to $\left( \ref{set_jk}%
\right) $. Recall that we have assumed that $\gamma \geq 0$ since when $%
\gamma \leq 0$ the proof follows similarly.

Now Lemma \ref{lem:max_boot} implies that 
\begin{eqnarray}
&&\Pr \left. ^{\ast }\right. \left\{ \inf_{\Xi _{j}\left( \upsilon \right)
;\Xi _{j}\left( \beta \right) ;\Xi _{k}\left( \gamma \right) }\widehat{ 
\mathbb{A}}_{n1}\left( \theta \right) +\mathbb{B}_{n1}^{\ast }\left( \theta
\right) \leq 0\right\}  \notag \\
&\leq &\Pr \left. ^{\ast }\right. \left\{ \inf_{\Xi _{j}\left( \upsilon
\right) ;\Xi _{j}\left( \beta \right) }\left\Vert \overline{\upsilon }%
\right\Vert \left\Vert M_{n}^{x}\left( \gamma \right) \right\Vert \leq
\sup_{\Xi _{k}\left( \gamma \right) }\left\Vert \frac{2}{n^{1/2}}%
\sum_{t=1}^{n}x_{t}\varepsilon _{t}^{\ast }\mathbf{1}_{t}\left( \gamma
\right) \right\Vert \right\}  \notag \\
&\leq &\Pr \left. ^{\ast }\right. \left\{ \left\Vert M_{n}^{x}\left( \gamma
\right) \right\Vert C2^{j-1}\leq \sup_{\left\{ \gamma :\left\Vert \gamma
\right\Vert <\eta \right\} }\left\Vert \frac{1}{n^{1/2}}\sum_{t=1}^{n}x_{t}%
\varepsilon _{t}^{\ast }\mathbf{1}_{t}\left( \gamma \right) \right\Vert
\right\}  \label{rateboot_2} \\
&\leq &C^{-1}2^{-2j}H_{n}\text{.}  \notag
\end{eqnarray}%
Observe that the bound in $\left( \ref{rateboot_2}\right) $ is independent
of $k$, i.e. the set $\Xi _{k}\left( \gamma \right) $. Defining%
\begin{eqnarray*}
\widetilde{\mathbb{A}}_{n2}\left( \theta \right) &=&\left( \overline{\beta }
_{1},\overline{\beta }_{2}^{\prime }\right) M_{n}^{x}\left( -\infty
;0\right) \left( \overline{\beta }_{1},\overline{\beta }_{2}^{\prime
}\right) ^{\prime } \\
\widetilde{\mathbb{B}}_{n2}^{\ast }\left( \theta \right) &=&\left( \overline{
\beta }_{1},\overline{\beta }_{2}^{\prime }\right) \frac{2}{n}
\sum_{t=1}^{n}x_{t1}\varepsilon _{t}^{\ast }\mathbf{1}_{t}\left( -\infty
;\gamma \right) \text{,}
\end{eqnarray*}%
$\left( \ref{definboot}\right) $ yields that 
\begin{eqnarray}
&&\Pr \left. ^{\ast }\right. \left\{ \inf_{\Xi _{j}\left( \upsilon \right)
;\Xi _{k}\left( \gamma \right) }\widetilde{\mathbb{A}}_{n2}^{\ast }\left(
\theta \right) +\widetilde{\mathbb{B}}_{n2}^{\ast }\left( \theta \right)
\leq 0\right\}  \notag \\
&\leq &\Pr \left. ^{\ast }\right. \left\{ \inf_{\Xi _{j}\left( \upsilon
\right) }\left\Vert \left( \overline{\beta }_{1},\overline{\beta }%
_{2}^{\prime }\right) \right\Vert M_{n}^{x_{1}}\left( -\infty ;0\right) \leq
\sup_{\Xi _{k}\left( \gamma \right) }\left\Vert \frac{2}{n}%
\sum_{t=1}^{n}x_{t1}\varepsilon _{t}^{\ast }\mathbf{1}_{t}\left( -\infty
;\gamma \right) \right\Vert \right\}  \notag \\
&\leq &\Pr \left. ^{\ast }\right. \left\{ \left\Vert M_{n}^{x_{1}}\left(
-\infty ;0\right) \right\Vert C2^{j-1}\leq \sup_{\left\{ \gamma :\left\Vert
\gamma \right\Vert <\eta \right\} }\left\Vert \frac{1}{n^{1/2}}%
\sum_{t=1}^{n}x_{t1}\varepsilon _{t}^{\ast }\mathbf{1}_{t}\left( -\infty
;\gamma \right) \right\Vert \right\}  \label{rateboot_3} \\
&\leq &C^{-1}2^{-2j}H_{n}\text{,}  \notag
\end{eqnarray}%
by Lemma \ref{lem:max_boot}, which once again the bound is independent of $k$%
.

Next, define%
\begin{equation*}
\widetilde{\mathbb{A}}_{n3}\left( \theta \right) =\widehat{\tau }%
^{2}q_{t}^{2}\mathbf{1}_{t}\left( 0;\gamma \right) ;\text{ \ \ }\widetilde{ 
\mathbb{B}}_{n3}^{\ast }\left( \theta \right) =\widehat{\tau }\frac{2}{n}%
\sum_{t=1}^{n}q_{t}\varepsilon _{t}^{\ast }\mathbf{1}_{t}\left( 0;\gamma
\right) \text{,}
\end{equation*}%
then, because $\widehat{\tau }=H_{n}+C_{1}$, 
\begin{eqnarray}
&&\Pr \left. ^{\ast }\right. \left\{ \inf_{\Xi _{j}\left( \upsilon \right)
;\Xi _{k}\left( \gamma \right) }\widetilde{\mathbb{A}}_{n3}\left( \theta
\right) +\widehat{\tau }\widetilde{\mathbb{B}}_{n3}^{\ast }\left( \theta
\right) \leq 0\right\}  \notag \\
&\leq &\Pr \left. ^{\ast }\right. \left\{ \inf_{\Xi _{j}\left( \upsilon
\right) }\left\Vert \widehat{\tau }\right\Vert \frac{1}{n}%
\sum_{t=1}^{n}q_{t}^{2}\mathbf{1}_{t}\left( 0;\gamma \right) \leq \sup_{\Xi
_{k}\left( \gamma \right) }\left\Vert \widetilde{\mathbb{B}}_{n3}^{\ast
}\left( \theta \right) /\tau _{0}\right\Vert \right\}  \notag \\
&\leq &\Pr \left. ^{\ast }\right. \left\{ \frac{C}{n}2^{3\left( k-1\right)
}\leq \sup_{\Xi _{k}\left( \gamma \right) }\left\Vert \widetilde{\mathbb{B}}%
_{n3}^{\ast }\left( \theta \right) /\widehat{\tau }\right\Vert \right\}
\label{rateboot_4} \\
&\leq &C^{-1}2^{-3k/2}H_{n}\text{,}  \notag
\end{eqnarray}%
by Lemma \ref{lem:max_boot} and Markov's inequality. Observe that the latter
displayed bound is independent of $j$, i.e. the set $\Xi _{j}\left( \upsilon
\right) $.

So, the left side of $\left( \ref{rateboot_1}\right) $ is bounded by 
\begin{eqnarray*}
&&\Pr \left. ^{\ast }\right. \left\{ \max_{j,k}\inf_{\Xi _{j}\left( \upsilon
\right) ;\Xi _{k}\left( \gamma \right) }\sum_{\ell =1}^{3}\widehat{\mathbb{A}
}_{n\ell }\left( \theta \right) +\mathbb{B}_{n\ell }^{\ast }\left( \theta
\right) \leq 0\right\} \\
&\leq &C^{-1}\left( \sum_{j=1}^{\log _{2}\frac{\eta }{C}n^{1/2}}2^{-2j}+
\sum_{k=1}^{\log _{2}\frac{\eta }{C}n^{1/3}}2^{-3k/2}\right) <\epsilon H_{n} 
\text{.}
\end{eqnarray*}%
using $\left( \ref{rateboot_2}\right) -\left( \ref{rateboot_4}\right) $.
This concludes the proof of part $\left( \mathbf{a}\right) $.

The proof of part $\left( \mathbf{b}\right) $ is similarly handled after
obvious changes, so it is omitted.\hfill $\blacksquare $

\section{\textbf{AUXILIARY LEMMAS}}

We begin with a set of maximal inequalities, which play a central role in
deriving convergence rates and tightness of various empirical processes. For 
$j=1$ or $2,$ let 
\begin{eqnarray*}
J_{n}\left( \gamma ,\gamma ^{\prime }\right) &=&\frac{1}{n^{1/2}}%
\sum_{t=1}^{n}\varepsilon _{t}x_{t}\mathbf{1}_{t}\left( \gamma ;\gamma
^{\prime }\right) \\
J_{1n}\left( \gamma ,\gamma ^{\prime }\right) &=&\frac{1}{n^{1/2}}%
\sum_{t=1}^{n}\varepsilon _{t}\left\vert q_{t}-\gamma \right\vert ^{j}%
\mathbf{1}_{t}\left( \gamma ;\gamma ^{\prime }\right) \\
J_{2n}\left( \gamma \right) &=&\frac{1}{n^{1/2}}\sum_{t=1}^{n}\left\{
\left\vert q_{t}-\gamma _{0}\right\vert ^{j}\mathbf{1}_{t}\left( \gamma
_{0};\gamma \right) -E\left\vert q_{t}-\gamma _{0}\right\vert ^{j}\mathbf{1}%
_{t}\left( \gamma _{0};\gamma \right) \right\}
\end{eqnarray*}%
and for some sequence $\left\{ z_{t}\right\} _{t=1}^{n}$, 
\begin{equation*}
J_{3n}\left( \gamma \right) =\frac{1}{n^{1/2}}\sum_{t=1}^{n}\left( z_{t}%
\mathbf{1}_{t}\left( \gamma _{0};\gamma \right) -Ez_{t}\mathbf{1}_{t}\left(
\gamma _{0};\gamma \right) \right) \text{.}
\end{equation*}

\begin{lemma}
\label{lem:maxineq2}Suppose Assumptions Z and Q hold for the sequence $%
\left\{ x_{t},\varepsilon _{t}\right\} _{t=1}^{n}$. In addition, for $%
J_{3n}\left( \gamma \right) ,$ assume that $\left\{ z_{t},q_{t}\right\}
_{t=1}^{n}$ be a sequence of strictly stationary, ergodic, and $\rho $
-mixing with $\sum_{m=1}^{\infty }\rho _{m}^{1/2}<\infty $, $E\left\vert
z_{t}\right\vert ^{4}<\infty $ and, for all $\gamma \in \Gamma $,$\ ${{\ 
\textrm{$E$}}}$\left( \left\vert z_{t}\right\vert ^{4}|q_{t}=\gamma \right)
<C<\infty $. Then, there exists $n_{0}<\infty $ such that for all $\gamma
^{\prime }\ $in a neighbourhood of $\gamma _{0}\ $and for all $n>n_{0}$ and $%
\epsilon \geq n_{0}^{-1}$, 
\begin{eqnarray}
\left( \mathbf{a}\right) \text{~\ ~}E\sup_{\gamma ^{\prime }<\gamma <\gamma
^{\prime }+\epsilon }\left\vert J_{n}\left( \gamma ^{\prime },\gamma \right)
\right\vert &\leq &C\epsilon ^{1/2}  \notag \\
\left( \mathbf{b}\right) \text{~~}E\sup_{\gamma ^{\prime }<\gamma <\gamma
^{\prime }+\epsilon }\left\vert J_{1n}\left( \gamma ^{\prime },\gamma
\right) \right\vert &\leq &C\epsilon ^{1/2}\left( \epsilon +\left\vert
\gamma _{0}-\gamma ^{\prime }\right\vert \right) ^{j}  \label{eq:maxineq2a}
\\
\left( \mathbf{c}\right) \ \ \text{~~}E\sup_{\gamma _{0}<\gamma <\gamma
_{0}+\epsilon }\left\vert J_{2n}\left( \gamma \right) \right\vert &\leq
&C\epsilon ^{j+1/2}  \label{eq:maxineq2b} \\
\left( \mathbf{d}\right) \ \ \text{~~}E\sup_{\gamma _{0}<\gamma <\gamma
_{0}+\epsilon }\left\vert J_{3n}\left( \gamma \right) \right\vert &\leq
&C\epsilon ^{1/2}\text{,}  \label{eq:maxineq2c}
\end{eqnarray}%
where $j=1$ or $2$.
\end{lemma}

\begin{proof}
Part $\left( \mathbf{a}\right) $ proceeds as in Hansen's $\left( 2000\right) 
$ Lemma A.3, so it is omitted.

Next part $\left( \mathbf{b}\right) $. This is almost identical to that of
Hansen's $\left( 2000\right) $ Lemma A.3 once observing that if $|\gamma
_{1}-\gamma ^{\prime }|\leq \epsilon $ and $|\gamma _{2}-\gamma ^{\prime
}|\leq \epsilon $ and $h_{t}(\gamma _{1},\gamma _{2})=|\varepsilon
_{t}(q_{t}-\gamma _{0})^{j}|\mathbf{1}_{t}(\gamma _{1},\gamma _{2})$, then
the bound in his Lemma A.1 (12) should be updated to 
\begin{equation*}
Eh_{i}^{r}\left( \gamma _{1},\gamma _{2}\right) \leq C\int_{\gamma
_{1}}^{\gamma _{2}}\left\vert q-\gamma _{0}\right\vert ^{jr}dq\leq C|\gamma
_{1}-\gamma _{2}|\epsilon _{1}^{jr}\text{,}
\end{equation*}
where $C<\infty $ and $\epsilon _{1}=\left( \epsilon +\left\vert \gamma
_{0}-\gamma ^{\prime }\right\vert \right) $, since $E\left( \left\vert
\varepsilon _{t}^{r}\right\vert |q_{t}\right) $ and the density $f\left(
q\right) $ of $q_{t}$ are bounded around $q_{t}=\gamma _{0}$. Hansen's bound
in (13) should be changed to $|\gamma _{1}-\gamma _{2}|\epsilon _{1}^{jr}$
for the same reason. Then, these new bounds imply that the bounds (15) and
(16) in his Lemma A.3 and the bounds (18) and (20) in the proof of his Lemma
A.2 should change to $\left\vert \gamma _{1}-\gamma _{2}\right\vert
^{2}\epsilon _{1}^{4j}$ and $n^{-1}\left\vert \gamma _{1}-\gamma
_{2}\right\vert \epsilon _{1}^{4j}+\left\vert \gamma _{1}-\gamma
_{2}\right\vert ^{2}\epsilon _{1}^{4j}$, respectively, to yield the desired
bound in $\left( \ref{eq:maxineq2a}\right) $.

Part $\left( \mathbf{c}\right) $. For notational simplicity we assume that $%
\gamma _{0}=0$. Let $\gamma _{k}=k/n$, for $k=1,...,m$, where $m=\left[
\epsilon n\right] +1$. By triangle inequality, 
\begin{equation}
\sup_{\gamma _{0}<\gamma <\gamma _{0}+\epsilon }\left\vert J_{2n}\left(
\gamma \right) \right\vert \leq \max_{k=1,...,m-1}\left\vert J_{2n}\left(
\gamma _{k}\right) \right\vert +\max_{k=1,...,m}\sup_{\gamma _{k-1}\leq
\gamma \leq \gamma _{k}}\left\vert J_{2n}\left( \gamma \right) -J_{2n}\left(
\gamma _{k-1}\right) \right\vert \text{.}  \label{eq:max_2b_1}
\end{equation}
Now because $f\left( \cdot \right) $ is continuous differentiable at $\gamma
_{0}$, standard algebra yields that 
\begin{equation}
E\left\vert q_{t}\right\vert ^{j}~\mathbf{1}_{t}\left( \gamma _{k-1};\gamma
_{k}\right) \leq C\gamma _{k}^{j}/n\text{.}  \label{eq:max_1}
\end{equation}
Next, using $\left( \ref{eq:max_1}\right) $ 
\begin{align*}
& \sup_{\gamma _{k-1}\leq \gamma \leq \gamma _{k}}\left\vert \frac{1}{
n^{1/2} }\sum_{t=1}^{n}\left\vert q_{t}\right\vert ^{j}~\mathbf{1}_{t}\left(
\gamma _{k-1};\gamma \right) \right\vert \\
& \leq \left( J_{2n}\left( \gamma _{k}\right) -J_{2n}\left( \gamma
_{k-1}\right) \right) +n^{1/2}E\left\vert q_{t}\right\vert ^{j}\text{~} 
\mathbf{1}_{t}\left( \gamma _{k-1};\gamma _{k}\right) \\
& =\left( J_{2n}\left( \gamma _{k}\right) -J_{2n}\left( \gamma _{k-1}\right)
\right) +C\gamma _{k}^{j}/n^{1/2}\text{.}
\end{align*}

Thus, using the inequality $\left( \sup_{j=1,...,\ell }\left\vert
c_{j}\right\vert \right) ^{4}\leq \sum_{j=1}^{\ell }\left\vert
c_{j}\right\vert ^{4}$, we conclude that second term on the right of $\left( %
\ref{eq:max_2b_1}\right) $ has absolute moment bounded by 
\begin{equation}
\left( \sum_{k=1}^{m}E\left\vert J_{2n}\left( \gamma _{k}\right)
-J_{2n}\left( \gamma _{k-1}\right) \right\vert ^{4}\right) ^{1/4}+C\gamma
_{m}^{j}/n^{1/2}\text{.}  \label{eq:max_2}
\end{equation}%
However, from Lemma 3.6 of Peligrad $\left( 1982\right) $, for any $k>i$, 
\begin{equation}
E\left\vert J_{2n}\left( \gamma _{k}\right) -J_{2n}\left( \gamma _{i}\right)
\right\vert ^{4}\leq C\left( n^{-1}E\left\vert q_{t}\right\vert ^{4j}~%
\mathbf{1}_{t}\left( \gamma _{i};\gamma _{k}\right) +\left( E\left\vert
q_{t}\right\vert ^{2j}~\mathbf{1}_{t}\left( \gamma _{i};\gamma _{k}\right)
\right) ^{2}\right) \text{.}  \notag
\end{equation}%
So, using again $\left( \ref{eq:max_1}\right) $ and that $m=[\varepsilon
n]+1 $ and $n^{-1}<\varepsilon $, we conclude that the first moment of the
second term on the right of $\left( \ref{eq:max_2b_1}\right) $ is $C\epsilon
^{j+1/2}$.

Next the first moment of the first term on the right of $\left( \ref%
{eq:max_2b_1}\right) $ is also bounded by $C\epsilon ^{j+1/2}$ by
Billingsley's $\left( 1968\right) $ Theorem 12.2 using the last displayed
inequality.

Finally part $\left( \mathbf{d}\right) $\textbf{. }This is similar to that
of $\left( \ref{eq:maxineq2b}\right) $. It is sufficient to note that, with $%
J_{3n}\left( \gamma \right) $, the bounds in $\left( \ref{eq:max_1}\right) $
and $\left( \ref{eq:max_2}\right) $ change to $C/n^{1/2}$ and $C\epsilon
^{2} $, respectively. This yields the results as $n^{-1}<\epsilon $.
\end{proof}

\begin{remark}
\label{rem:tight}One of the consequences of the previous lemma $\left( 
\mathbf{a}\right) $ and $\left( \mathbf{b}\right) $, which allows the
maximal inequality to hold for any $\gamma ^{\prime }$ in a neighbourhood of 
$\gamma _{0}$, is that 
\begin{equation*}
nE\sup_{g_{1}<g<g_{1}+\epsilon }\left\vert J_{n}\left( \gamma
_{0}+g/r_{n}\right) -J_{n}\left( \gamma _{0}+g_{1}/r_{n}\right) \right\vert
\leq C\left( \epsilon +g_{1}\right) \epsilon ^{1/2}\text{,}
\end{equation*}
which can be made small by choosing small $\epsilon $ and $r_{n}\rightarrow
\infty $. This is used to verify the stochastic equicontinuity of the
rescaled and reparameterized empirical processes in the proof of Theorem \ref%
{Th:AD_cube}.
\end{remark}

The following two lemmas are used in the proof of Proposition \ref%
{Prop:xihat}. Before we state our next lemma, we need to introduce some
notation. In what follows 
\begin{eqnarray}
g_{r}\left( q\right) &=&E\left( x_{t2}^{r}\varepsilon _{t}^{2}\mid
q_{t}=q\right) ;\text{ \ \ }g_{r}^{\ast }\left( q\right) =E\left(
x_{t2}^{r}\mid q_{t}=q\right)  \notag \\
h_{r,k}\left( q\right) &=&\sum_{j=0}^{4-k}a^{j}\kappa _{j+k}\frac{\partial
^{j}}{\partial q^{j}}\left( f\left( q\right) g_{r}\left( q\right) \right) 
\text{, }k\leq 4  \label{not} \\
h_{r,k}^{\ast }\left( q\right) &=&\sum_{j=0}^{4-k}a^{j}\kappa _{j+k}\frac{
\partial ^{j}}{\partial q^{j}}\left( f\left( q\right) g_{r}^{\ast }\left(
q\right) \right) \text{, }k\leq 4\text{.}  \notag
\end{eqnarray}
Note that we have implicitly assumed that $g_{r}\left( q\right) $ and $%
f\left( q\right) $ have four continuous derivatives. Also, without loss of
generality, we assume $\gamma _{0}=0$ and $x_{t2}$ is a scalar to ease
notation.

\begin{lemma}
\label{lem:4proposition}Under $\mathbf{K1,K2}$ and $\mathbf{K4}$\textbf{, }
we have that for integers $0\leq \ell ,r\leq 4$, 
\begin{eqnarray}
&&\frac{1}{na^{1+\ell }}\sum_{t=1}^{n}\varepsilon
_{t}^{2}x_{t2}^{r}q_{t}^{\ell }K\left( \frac{q_{t}-\widehat{\gamma }}{a}
\right) -h_{r,\ell }\left( 0\right) =o_{p}\left( 1\right)  \notag \\
&&\frac{1}{na^{1+\ell }}\sum_{t=1}^{n}x_{t2}^{r}q_{t}^{\ell }K\left( \frac{
q_{t}-\widehat{\gamma }}{a}\right) -h_{r,\ell }^{\ast }\left( 0\right)
=o_{p}\left( 1\right) \text{.}  \label{kernel_22}
\end{eqnarray}
\end{lemma}

\begin{proof}
First, observe that we are using the normalization $\left( na^{1+\ell
}\right) ^{-1}$ instead of the standard $\left( na\right) ^{-1}$. This is
due to the factor $q_{t}^{\ell }$. We shall consider only the first equality
in $\left( \ref{kernel_22}\right) $, the second one being similarly handled.
Now abbreviating $K_{t}\left( \gamma \right) =K\left( \frac{q_{t}-\gamma }{a}
\right) $, we have that standard kernel arguments imply 
\begin{equation*}
\frac{1}{na^{1+\ell }}\sum_{t=1}^{n}\varepsilon
_{t}^{2}x_{t2}^{r}q_{t}^{\ell }K_{t}\left( 0\right) -h_{r,\ell }\left(
0\right) =O_{p}\left( \left( na\right) ^{-1/2}\right) +o\left( a^{4-\ell
}\right) \text{.}
\end{equation*}

So, to complete the proof of the lemma, it suffices to show that 
\begin{equation}
\frac{1}{na^{1+\ell }}\sum_{t=1}^{n}\varepsilon
_{t}^{2}x_{t2}^{r}q_{t}^{\ell }\left\{ K_{t}\left( \widehat{\gamma }\right)
-K_{t}\left( 0\right) \right\} =o_{p}\left( 1\right) \text{.}
\label{kernel_3}
\end{equation}

Proposition \ref{Consistency} implies that there exists $C$ such that $\Pr
\left\{ \left\vert \widehat{\gamma }\right\vert >Cn^{-1/3}\right\} \leq \eta 
$, for any $\eta >0$. So, we only need to show that $\left( \ref{kernel_3}
\right) $ holds true when $\left\vert \widehat{\gamma }\right\vert \leq
Cn^{-1/3}$. In that case, we have that the left side of $\left( \ref%
{kernel_3}\right) $ is bounded by 
\begin{eqnarray}
&&\sup_{\left\vert \gamma \right\vert \leq Cn^{-1/3}}\left\vert \frac{1}{
na^{1+\ell }}\sum_{t=1}^{n}\varepsilon _{t}^{2}x_{t2}^{r}q_{t}^{\ell
}\left\{ K_{t}\left( \gamma \right) -K_{t}\left( 0\right) \right\}
\right\vert  \notag \\
&\leq &\sup_{\left\vert \gamma \right\vert \leq Cn^{-1/3}}\left\vert \frac{1 
}{na^{1+\ell }}\sum_{t=1}^{n}\varepsilon _{t}^{2}x_{t2}^{r}q_{t}^{\ell
}\left\{ K_{t}\left( \gamma \right) -K_{t}\left( 0\right) \right\} \mathbf{1}
\left( \left\vert q_{t}\right\vert <a^{1/2}\right) \right\vert
\label{kernel_31} \\
&&+\sup_{\left\vert \gamma \right\vert \leq Cn^{-1/3}}\left\vert \frac{1}{
na^{1+\ell }}\sum_{t=1}^{n}\varepsilon _{t}^{2}x_{t2}^{r}q_{t}^{\ell
}\left\{ K_{t}\left( \gamma \right) -K_{t}\left( 0\right) \right\} \mathbf{1}
\left( \left\vert q_{t}\right\vert \geq a^{1/2}\right) \right\vert \text{.} 
\notag
\end{eqnarray}

The expectation of second term on the right of $\left( \ref{kernel_31}
\right) $ is bounded by 
\begin{eqnarray*}
&&\frac{C_{1}}{na}\sum_{t=1}^{n}E\left( \varepsilon _{t}^{2}\left\vert
x_{t2}\right\vert ^{r}\left\vert \frac{q_{t}}{a}\right\vert ^{\ell }K\left( 
\frac{q_{t}}{a}\right) \mathbf{1}\left( \left\vert q_{t}\right\vert \geq
a^{1/2}\right) \right) \\
&\leq &\frac{C_{1}}{a}\int_{q}\left\vert \frac{q}{a}\right\vert ^{\ell
}g_{r}\left( q\right) f\left( q\right) K\left( \frac{q}{a}\right) \mathbf{1}
\left( \left\vert q_{t}\right\vert \geq a^{1/2}\right) dq \\
&=&C_{1}\int_{\left\vert q\right\vert \geq a^{-1/2}}\left\vert q\right\vert
^{\ell }g_{r}\left( aq\right) f\left( aq\right) K\left( q\right) dq \\
&=&o\left( a^{2-\ell /4}\right) \text{,}
\end{eqnarray*}
because by $\mathbf{K1}$, $\kappa _{\ell }<C_{1}$, for $\ell \leq 4$.

For some $0<\psi <1$, the first term on the right of $\left( \ref{kernel_31}
\right) $ is bounded by 
\begin{eqnarray}
&&\frac{C}{n^{1/3}}\sup_{\left\vert \gamma \right\vert \leq
Cn^{-1/3}}\left\vert \frac{1}{na^{2}}\sum_{t=1}^{n}\varepsilon
_{t}^{2}\left\vert x_{t2}\right\vert ^{r}\left\vert \frac{q_{t}}{a}
\right\vert ^{\ell }K^{\prime }\left( \frac{q_{t}-\psi \gamma }{a}\right) 
\mathbf{1}\left( \left\vert q_{t}\right\vert <a^{1/2}\right) \right\vert 
\notag \\
&\leq &\frac{C}{n^{1/3}}\left\vert \frac{1}{na^{2}}\sum_{t=1}^{n}\varepsilon
_{t}^{2}\left\vert x_{t2}\right\vert ^{r}\left\vert \frac{q_{t}}{a}
\right\vert ^{\ell }K^{\prime }\left( \frac{q_{t}}{a}\right) \mathbf{1}
\left( a^{3/2}<\left\vert q_{t}\right\vert <a^{1/2}\right) \right\vert
\label{ineq} \\
&&+\frac{C}{n^{1/3}}\sup_{\left\vert \gamma \right\vert \leq
Cn^{-1/3}}\left\vert \frac{1}{na^{2}}\sum_{t=1}^{n}\varepsilon
_{t}^{2}\left\vert x_{t2}\right\vert ^{r}\left\vert \frac{q_{t}}{a}
\right\vert ^{\ell }K^{\prime }\left( \frac{q_{t}-\phi \gamma }{a}\right) 
\mathbf{1}\left( \left\vert q_{t}\right\vert <a^{3/2}\right) \right\vert 
\notag
\end{eqnarray}
because $\mathbf{K4}$ implies that $\gamma =o\left( a\right) $ when $%
\left\vert \gamma \right\vert \leq Cn^{-1/3}$, and hence if $%
a^{3/2}<\left\vert q_{t}\right\vert <a^{1/2}$ we have $\left\vert K^{\prime
}\left( \frac{q_{t}-\phi \gamma }{a}\right) /K^{\prime }\left( \frac{q_{t}}{%
a }\right) \right\vert \leq C_{1}$ by $\mathbf{K2}$. But, it is well known
that the first moment of the first term on the right of $\left( \ref{ineq}
\right) $ is bounded, whereas that of the second term on the right is also
bounded because $E\left\vert \frac{q_{t}}{a}\right\vert ^{\ell }\mathbf{1}
\left( \left\vert q_{t}\right\vert <a^{3/2}\right) <a^{\left( \ell +3\right)
/2}$ and 
\begin{equation}
\left\vert K^{\prime }\left( \frac{q_{t}-\phi \gamma }{a}\right)
-K_{t}^{\prime }\left( 0\right) \right\vert \mathbf{1}\left( \left\vert
q_{t}\right\vert <a^{3/2}\right) \leq Ca^{1/2}\text{.}  \label{k_1}
\end{equation}

So, the expectation of the first term on the right of $\left( \ref{kernel_31}
\right) $ is $O\left( n^{-1/3}\right) $. This concludes the proof of the
lemma.
\end{proof}

\begin{lemma}
\label{lem:4proposition_1}Under $\mathbf{K1-K4}$\textbf{, }we have that for
integers $0\leq r,\ell \leq 4$, 
\begin{equation}
\frac{1}{na}\sum_{t=1}^{n}x_{t2}^{r}q_{t}^{\ell }K_{t}\left( \widehat{\gamma 
}\right) \varepsilon _{t}=o_{p}\left( a^{\ell }n^{1/2}\right) \text{.}
\label{prop9_1}
\end{equation}
\end{lemma}

\begin{proof}
To simplify the notation, we assume that $r=0$. The left side of $\left( \ref%
{prop9_1}\right) $ is 
\begin{equation*}
\frac{1}{na}\sum_{t=1}^{n}q_{t}^{\ell }\left\{ K_{t}\left( \widehat{\gamma }
\right) -K_{t}\left( 0\right) \right\} \varepsilon _{t}+\frac{1}{na}
\sum_{t=1}^{n}q_{t}^{\ell }K_{t}\left( 0\right) \varepsilon _{t}\text{.}
\end{equation*}
The second term is easily shown to be $O_{p}\left( n^{-1/2}a^{\ell
-1/2}\right) $. Next the first term of the last displayed expression is 
\begin{eqnarray}
&&\frac{1}{na}\sum_{t=1}^{n}q_{t}^{\ell }\left\{ K_{t}\left( \widehat{\gamma 
}\right) -K_{t}\left( 0\right) \right\} \varepsilon _{t}\mathbf{1}\left(
\left\vert q_{t}\right\vert <a^{\zeta }\right)  \label{eps_1} \\
&&+\frac{1}{na}\sum_{t=1}^{n}q_{t}^{\ell }\left\{ K_{t}\left( \widehat{
\gamma }\right) -K_{t}\left( 0\right) \right\} \varepsilon _{t}\mathbf{1}
\left( \left\vert q_{t}\right\vert \geq a^{\zeta }\right) \text{,}  \notag
\end{eqnarray}
where $\zeta =1-2/\ell $, if $\ell >2$, and $\zeta <1$ if $\ell \leq 2$. The
second term of $\left( \ref{eps_1}\right) $ is 
\begin{equation*}
a^{\ell }\frac{1}{na}\sum_{t=1}^{n}\left( \frac{q_{t}}{a}\right) ^{\ell
}\left\{ K_{t}\left( \widehat{\gamma }\right) -K_{t}\left( 0\right) \right\}
\varepsilon _{t}\mathbf{1}\left( \left\vert q_{t}\right\vert \geq a^{\zeta
}\right) \text{,}
\end{equation*}
whose first absolute moment is bounded by 
\begin{equation*}
a^{\ell -1}\int_{\left\vert q\right\vert \geq a^{\zeta }}\left( \frac{q}{a}
\right) ^{\ell }K\left( \frac{q}{a}\right) f_{q}\left( q\right) dq\leq
C_{1}a^{\ell }\int_{\left\vert q\right\vert \geq a^{\zeta -1}}q^{\ell
}K\left( q\right) f_{q}\left( aq\right) dq=o\left( a^{\ell }\right)
\end{equation*}
because by $\mathbf{K1}$, $\kappa _{4}<\infty $. So to complete the proof we
need to examine the first term of $\left( \ref{eps_1}\right) $, which using
the characteristic function of the kernel function is 
\begin{equation*}
\int \phi \left( av\right) \left( e^{iv\widehat{\gamma }}-1\right) \left\{ 
\frac{1}{n}\sum_{t=1}^{n}q_{t}^{\ell }\varepsilon _{t}e^{ivq_{t}}\mathbf{1}
\left( \left\vert q_{t}\right\vert <a^{\zeta }\right) \right\} dv\text{.}
\end{equation*}
But its clear that the last displayed expression is bounded by 
\begin{eqnarray*}
&&\widehat{\gamma }\int v\left\vert \phi \left( av\right) \right\vert
\left\vert \frac{1}{n}\sum_{t=1}^{n}q_{t}^{\ell }\varepsilon _{t}e^{ivq_{t}} 
\mathbf{1}\left( \left\vert q_{t}\right\vert <a^{\zeta }\right) \right\vert
dv=O_{p}\left( a^{\ell \zeta }n^{-1/2}\widehat{\gamma }\right) \int
v\left\vert \phi \left( av\right) \right\vert dv \\
&=&O_{p}\left( a^{\ell }\left( na^{3}\right) ^{-4/3}n^{1/2}\right)
\end{eqnarray*}
using that $\zeta =1-2/\ell $, if $\ell \geq 2$ and $\zeta <1$ when $0\leq
\ell <2$, $\widehat{\gamma }=O_{p}\left( n^{-1/3}\right) $ and $\mathbf{K4}$
. This concludes the proof of the lemma.
\end{proof}

We now extend the maximal inequalities in Lemma \ref{lem:maxineq2} to its
bootstrap analogues. Define $J_{n}^{\ast }\left( \gamma ,\gamma ^{\prime
}\right) $ and $J_{1n}^{\ast }\left( \gamma ,\gamma ^{\prime }\right) $ by
replacing $\varepsilon _{t}$ in $J_{n}\ $and $J_{1n}$ with $\widehat{e}%
_{t}\eta _{t}$, that is 
\begin{eqnarray*}
J_{n}^{\ast }\left( \gamma ,\gamma ^{\prime }\right) &=&\frac{1}{n^{1/2}}
\sum_{t=1}^{n}x_{t}\mathbf{1}_{t}\left( \gamma ,\gamma ^{\prime }\right) 
\widehat{e}_{t}\eta _{t} \\
J_{1n}^{\ast }\left( \gamma ,\gamma ^{\prime }\right) &=&\frac{1}{n^{1/2}}
\sum_{t=1}^{n}\left\vert q_{t}-\gamma \right\vert ^{j}\mathbf{1}_{t}\left(
\gamma ;\gamma ^{\prime }\right) \widehat{e}_{t}\eta _{t}\text{,}
\end{eqnarray*}%
and recall that $H_{n}$ denotes a sequence of positive $O_{p}\left( 1\right) 
$ random variables.

\begin{lemma}
\label{lem:max_boot}Under Assumption Z, we have that for all $\epsilon
,\varsigma >0$, there exists $\zeta >0$ such that 
\begin{eqnarray}
\Pr \left. ^{\ast }\right. \left\{ \sup_{\gamma ^{\prime }<\gamma <\gamma
^{\prime }+\epsilon }\left\vert J_{n}^{\ast }\left( \gamma ^{\prime },\gamma
\right) \right\vert >\epsilon \right\} &\leq &\zeta \varsigma H_{n}\text{,}
\label{propboot_1} \\
\Pr \left. ^{\ast }\right. \left\{ \sup_{\gamma ^{\prime }<\gamma <\gamma
^{\prime }+\epsilon }\left\vert J_{1n}^{\ast }\left( \gamma ^{\prime
},\gamma \right) \right\vert >C\epsilon ^{1/2}\left( \epsilon +\left\vert
\gamma _{0}-\gamma ^{\prime }\right\vert \right) ^{j}\right\} &\leq &\zeta
\varsigma H_{n}\text{.}  \label{propboot_2}
\end{eqnarray}
\end{lemma}

\begin{proof}
We shall assume for notational simplicity that $\gamma _{0}<\widehat{\gamma }
$, and that $\gamma _{j}=\gamma _{1}+\frac{\zeta }{m}j$ and $n\zeta
/2<m<n\zeta $, as $n$ can be chosen such that $n\zeta >1$. By definition, 
\begin{eqnarray*}
J_{n}^{\ast }\left( \gamma _{k},\gamma _{j}\right) &=&\frac{1}{n^{1/2}}
\sum_{t=1}^{n}x_{t}\varepsilon _{t}\boldsymbol{1}_{t}\left( \gamma
_{j};\gamma _{k}\right) \eta _{t} \\
&&+\frac{1}{n^{1/2}}\sum_{t=1}^{n}x_{t}x_{t}^{\prime }\boldsymbol{1}
_{t}\left( \gamma _{j};\gamma _{k}\right) \eta _{t}\left( \widehat{\beta }
-\beta \right) \\
&&+\frac{1}{n^{1/2}}\sum_{t=1}^{n}x_{t}x_{t}^{\prime }\boldsymbol{1}
_{t}\left( \gamma _{0}\right) \boldsymbol{1}_{t}\left( \gamma _{j};\gamma
_{k}\right) \eta _{t}\left( \widehat{\delta }-\delta \right) \\
&&+\frac{1}{n^{1/2}}\sum_{t=1}^{n}x_{t}x_{t}^{\prime }\boldsymbol{1}
_{t}\left( \gamma _{0};\widehat{\gamma }\right) \boldsymbol{1}_{t}\left(
\gamma _{j};\gamma _{k}\right) \eta _{t}\widehat{\delta }\text{.}
\end{eqnarray*}

Now by standard inequalities and that $\eta _{t}\sim iid\left( 0,1\right) $
with a finite fourth moments, the fourth (bootstrap) moment of the right
side of last displayed equation is bounded by 
\begin{eqnarray}
&&\left\vert \frac{1}{n}\sum_{t=1}^{n}\left\Vert x_{t}\right\Vert
^{2}\varepsilon _{t}^{2}\boldsymbol{1}_{t}\left( \gamma _{j};\gamma
_{k}\right) \right\vert ^{2}+\left\Vert \widehat{\beta }-\beta \right\Vert
^{4}\left\vert \frac{1}{n}\sum_{t=1}^{n}\left\Vert x_{t}\right\Vert ^{4} 
\boldsymbol{1}_{t}\left( \gamma _{j};\gamma _{k}\right) \right\vert ^{2} 
\notag \\
&&+\left\Vert \widehat{\delta }-\delta \right\Vert ^{4}\left\vert \frac{1}{n}
\sum_{t=1}^{n}\left\Vert x_{t}\right\Vert ^{4}\boldsymbol{1}_{t}\left(
\gamma _{j};\gamma _{k}\right) \boldsymbol{1}_{t}\left( \gamma _{0}\right)
\right\vert ^{2}  \label{boot_12} \\
&&+\left\Vert \widehat{\delta }\right\Vert ^{4}\left\vert \frac{1}{n}
\sum_{t=1}^{n}\left\Vert x_{t}\right\Vert ^{4}\boldsymbol{1}_{t}\left(
\gamma _{j};\gamma _{k}\right) \boldsymbol{1}_{t}\left( \gamma _{0};\widehat{
\gamma }\right) \right\vert ^{2}\text{.}  \notag
\end{eqnarray}
Because for fixed $\zeta >0$, there exists $n_{0}$ such that for $n>n_{0}$, $%
Cn^{-1}<\zeta $, the expectation of the first term of $\left( \ref{boot_12}
\right) $ is bounded by 
\begin{equation*}
C\left[ \left( k-j\right) \zeta _{m}+\left( \frac{\left( k-j\right) \zeta
_{m}}{n}\right) ^{1/2}\right] ^{2}\leq C\left( k-j\right) ^{2}\zeta _{m}^{2} 
\text{,}
\end{equation*}
arguing similarly as in Hansen's $\left( 2000\right) $ Lemma A.3 and $\zeta
_{m}=\zeta /m$.

Next, recalling that $\widehat{\gamma }=\gamma _{0}+D/n^{1/3}$, because $%
\boldsymbol{1}\left( \gamma _{j}<q_{t}<\gamma _{k}\right) \boldsymbol{1}
\left( \gamma _{0}<q_{t}<\widehat{\gamma }\right) \leq \boldsymbol{1}\left(
\gamma _{j}<q_{t}<\gamma _{k}\right) $, the expectation of the fourth term
of $\left( \ref{boot_12}\right) $ is bounded by 
\begin{eqnarray*}
&&\left\vert E\left\{ \left\Vert x_{t}\right\Vert ^{4}\boldsymbol{1}
_{t}\left( \gamma _{j};\gamma _{k}\right) \right\} \right\vert
^{2}+\left\vert \frac{1}{n}\sum_{t=1}^{n}\left\{ \left\Vert x_{t}\right\Vert
^{4}\boldsymbol{1}_{t}\left( \gamma _{j};\gamma _{k}\right) -E\left\{
\left\Vert x_{t}\right\Vert ^{4}\boldsymbol{1}_{t}\left( \gamma _{j};\gamma
_{k}\right) \right\} \right\} \right\vert ^{2} \\
&\leq &C\left( k-j\right) ^{2}\zeta _{m}^{2}\text{.}
\end{eqnarray*}

Finally, the second and third terms of $\left( \ref{boot_12}\right) $ are 
\begin{equation*}
H_{n}\frac{1}{n^{3}}\sum_{t=1}^{n}E\left( \left\Vert x_{t}\right\Vert ^{8} 
\boldsymbol{1}_{t}\left( \gamma _{j};\gamma _{k}\right) \right) =H_{n}\left(
k-j\right) ^{2}\zeta _{m}^{2}\text{.}
\end{equation*}
From here we now conclude that $\left( \ref{propboot_1}\right) $ holds true,
so is the lemma proceeding as in Hansen's $\left( 2000\right) $ Lemma A.3
and in particular his expressions $\left( 20\right) -\left( 22\right) $
because if a sequence of random variables has finite first moments, it
implies that it is $O_{p}\left( 1\right) $. The proof of $\left( \ref%
{propboot_2}\right) $ proceeds similarly and thus omitted.
\end{proof}

\begin{remark}
One of the consequences of the previous lemma is that 
\begin{equation*}
nE^{\ast }\sup_{g_{1}<g<g_{1}+\epsilon }\left\vert J_{n}^{\ast }\left(
\gamma _{0}+g/r_{n}\right) -J_{n}^{\ast }\left( \gamma
_{0}+g_{1}/r_{n}\right) \right\vert =\left( \epsilon +g_{1}\right) \epsilon
^{1/2}H_{n}\text{,}
\end{equation*}
which can be made small by choosing small $\epsilon $ and $r_{n}\rightarrow
\infty $.
\end{remark}

\begin{lemma}
\label{lem:4propositionBoot}Under $\mathbf{K1,K2}$ and $\mathbf{K4}$\textbf{%
\ \ , }we have that for integers $0\leq \ell ,r\leq 4$, 
\begin{eqnarray}
&&\frac{1}{na^{1+\ell }}\sum_{t=1}^{n}\varepsilon _{t}^{\ast
2}x_{t2}^{r}q_{t}^{\ell }K\left( \frac{q_{t}-\widehat{\gamma }^{\ast }}{a}
\right) -h_{r,\ell }\left( 0\right) =o_{p^{\ast }}\left( 1\right)  \notag \\
&&\frac{1}{na^{1+\ell }}\sum_{t=1}^{n}x_{t2}^{r}q_{t}^{\ell }K\left( \frac{
q_{t}-\widehat{\gamma }^{\ast }}{a}\right) -h_{r,\ell }^{\ast }\left(
0\right) =o_{p^{\ast }}\left( 1\right) \text{.}  \label{kernel_22Boot}
\end{eqnarray}
\end{lemma}

\begin{proof}
We shall consider only the first equality in $\left( \ref{kernel_22Boot}
\right) $, the second one being similarly handled. Now standard kernel
arguments imply 
\begin{equation*}
\frac{1}{na^{1+\ell }}\sum_{t=1}^{n}\varepsilon _{t}^{\ast
2}x_{t2}^{r}q_{t}^{\ell }K_{t}\left( 0\right) -h_{r,\ell }\left( 0\right)
=O_{p^{\ast }}\left( \left( na\right) ^{-1/2}\right) +o_{p}\left( a^{4-\ell
}\right) \text{.}
\end{equation*}
So, to complete the proof of the lemma, it suffices to show that 
\begin{equation}
\frac{1}{na^{1+\ell }}\sum_{t=1}^{n}\varepsilon _{t}^{\ast
2}x_{t2}^{r}q_{t}^{\ell }\left\{ K_{t}\left( \widehat{\gamma }^{\ast
}\right) -K_{t}\left( 0\right) \right\} =o_{p^{\ast }}\left( 1\right) \text{.%
}  \label{kernel_3Boot}
\end{equation}%
Proposition \ref{Prop:ConsistencyBoot} implies that there exists $C>0$ such that $%
\Pr^{\ast }\left\{ \left\vert \widehat{\gamma }^{\ast }\right\vert
>Cn^{-1/3}\right\} \leq H_{n}$. So, we only need to show that $\left( \ref%
{kernel_3}\right) $ holds true when $\left\vert \widehat{\gamma }^{\ast
}\right\vert \leq Cn^{-1/3}$, so that we have that the left side of $\left( %
\ref{kernel_3Boot}\right) $ is bounded by 
\begin{eqnarray}
&&\sup_{\left\vert \gamma \right\vert \leq Cn^{-1/3}}\left\vert \frac{1}{
na^{1+\ell }}\sum_{t=1}^{n}\varepsilon _{t}^{\ast 2}x_{t2}^{r}q_{t}^{\ell
}\left\{ K_{t}\left( \gamma \right) -K_{t}\left( 0\right) \right\}
\right\vert  \notag \\
&\leq &\sup_{\left\vert \gamma \right\vert \leq Cn^{-1/3}}\left\vert \frac{1%
}{na^{1+\ell }}\sum_{t=1}^{n}\varepsilon _{t}^{\ast 2}x_{t2}^{r}q_{t}^{\ell
}\left\{ K_{t}\left( \gamma \right) -K_{t}\left( 0\right) \right\} \mathbf{1}%
\left( \left\vert q_{t}\right\vert <a^{1/2}\right) \right\vert
\label{kernel_31Boot} \\
&&+\sup_{\left\vert \gamma \right\vert \leq Cn^{-1/3}}\left\vert \frac{1}{
na^{1+\ell }}\sum_{t=1}^{n}\varepsilon _{t}^{\ast 2}x_{t2}^{r}q_{t}^{\ell
}\left\{ K_{t}\left( \gamma \right) -K_{t}\left( 0\right) \right\} \mathbf{1}%
\left( \left\vert q_{t}\right\vert \geq a^{1/2}\right) \right\vert \text{.} 
\notag
\end{eqnarray}

The expectation of second term on the right of $\left( \ref{kernel_31Boot}%
\right) $ is bounded by 
\begin{eqnarray*}
&&\frac{C_{1}}{na}\sum_{t=1}^{n}E^{\ast }\left( \varepsilon _{t}^{\ast
2}\left\vert x_{t2}\right\vert ^{r}\left\vert \frac{q_{t}}{a}\right\vert
^{\ell }K\left( \frac{q_{t}}{a}\right) \mathbf{1}\left( \left\vert
q_{t}\right\vert \geq a^{1/2}\right) \right) \\
&=&\frac{C_{1}}{na}\sum_{t=1}^{n}\left\vert x_{t2}\right\vert ^{r}\left\vert 
\frac{q_{t}}{a}\right\vert ^{\ell }K\left( \frac{q_{t}}{a}\right) \mathbf{1}
\left( \left\vert q_{t}\right\vert \geq a^{1/2}\right) \frac{1}{n}
\sum_{s=1}^{n}\widehat{\varepsilon }_{t}^{2} \\
&=&\frac{C_{1}}{na}\sum_{t=1}^{n}\left\vert x_{t2}\right\vert ^{r}\left\vert 
\frac{q_{t}}{a}\right\vert ^{\ell }K\left( \frac{q_{t}}{a}\right) \mathbf{1}
\left( \left\vert q_{t}\right\vert \geq a^{1/2}\right) H_{n}\text{,}
\end{eqnarray*}%
where $C_{1}$ denotes a generic positive finite constant. Now, 
\begin{equation*}
E\frac{1}{na}\sum_{t=1}^{n}\left\vert x_{t2}\right\vert ^{r}\left\vert \frac{
q_{t}}{a}\right\vert ^{\ell }K\left( \frac{q_{t}}{a}\right) \mathbf{1}\left(
\left\vert q_{t}\right\vert \geq a^{1/2}\right) =o\left( a^{2-\ell /4}\right)
\end{equation*}%
proceeding as we did in Lemma \ref{lem:4proposition}. So, we conclude that
right of $\left( \ref{kernel_31Boot}\right) $ is $o\left( a^{2-\ell
/4}\right) H_{n}$.

For some $0<\psi <1$, the first term on the right of $\left( \ref%
{kernel_31Boot}\right) $ is bounded by 
\begin{eqnarray}
&&\frac{C_{1}}{n^{1/3}}\sup_{\left\vert \gamma \right\vert \leq
Cn^{-1/3}}\left\vert \frac{1}{na^{2}}\sum_{t=1}^{n}\varepsilon _{t}^{\ast
2}\left\vert x_{t2}\right\vert ^{r}\left\vert \frac{q_{t}}{a}\right\vert
^{\ell }K^{\prime }\left( \frac{q_{t}-\psi \gamma }{a}\right) \mathbf{1}%
\left( \left\vert q_{t}\right\vert <a^{1/2}\right) \right\vert  \notag \\
&\leq &\frac{C_{1}}{n^{1/3}}\left\vert \frac{1}{na^{2}}\sum_{t=1}^{n}%
\varepsilon _{t}^{\ast 2}\left\vert x_{t2}\right\vert ^{r}\left\vert \frac{
q_{t}}{a}\right\vert ^{\ell }K^{\prime }\left( \frac{q_{t}}{a}\right) 
\mathbf{1}\left( a^{3/2}<\left\vert q_{t}\right\vert <a^{1/2}\right)
\right\vert  \label{ineqBoot} \\
&&+\frac{C_{1}}{n^{1/3}}\sup_{\left\vert \gamma \right\vert \leq
Cn^{-1/3}}\left\vert \frac{1}{na^{2}}\sum_{t=1}^{n}\varepsilon _{t}^{\ast
2}\left\vert x_{t2}\right\vert ^{r}\left\vert \frac{q_{t}}{a}\right\vert
^{\ell }K^{\prime }\left( \frac{q_{t}-\phi \gamma }{a}\right) \mathbf{1}%
\left( \left\vert q_{t}\right\vert <a^{3/2}\right) \right\vert  \notag
\end{eqnarray}%
because $\mathbf{K4}$ implies that $\gamma =o\left( a\right) $ when $%
\left\vert \gamma \right\vert \leq Cn^{-1/3}$, and hence $\left\vert
K^{\prime }\left( \frac{q_{t}-\phi \gamma }{a}\right) /K^{\prime }\left( 
\frac{q_{t}}{a}\right) \right\vert \leq C_{1}$ by $\mathbf{K2}$ if $%
a^{3/2}<\left\vert q_{t}\right\vert <a^{1/2}$. But, it is well known that
the first moment of the first term on the right of $\left( \ref{ineqBoot}%
\right) $ is bounded, whereas that of the second term on the right is also
bounded because $E\left\vert \frac{q_{t}}{a}\right\vert ^{\ell }\mathbf{1}%
\left( \left\vert q_{t}\right\vert <a^{3/2}\right) <a^{\left( \ell +3\right)
/2}$ and $\left( \ref{k_1}\right) $. So, the expectation of the first term
on the right of $\left( \ref{kernel_31Boot}\right) $ is $O_{p}\left(
n^{-1/3}\right) $. This concludes the proof of the lemma.
\end{proof}

\begin{lemma}
\label{lem:4proposition_1Boot}Under $\mathbf{K1-K4}$\textbf{, }we have that
for integers $0\leq r,\ell \leq 4$, 
\begin{equation}
\frac{1}{na}\sum_{t=1}^{n}x_{t2}^{r}q_{t}^{\ell }K_{t}\left( \widehat{\gamma 
}^{\ast }\right) \varepsilon _{t}^{\ast }=o_{p^{\ast }}\left( a^{\ell
}n^{1/2}\right) \text{.}  \label{prop9_1Boot}
\end{equation}
\end{lemma}

\begin{proof}
To simplify the notation, we assume that $r=0$. The left side of $\left( \ref%
{prop9_1Boot}\right) $ is 
\begin{equation*}
\frac{1}{na}\sum_{t=1}^{n}q_{t}^{\ell }\left\{ K_{t}\left( \widehat{\gamma }
^{\ast }\right) -K_{t}\left( 0\right) \right\} \varepsilon _{t}^{\ast }+ 
\frac{1}{na}\sum_{t=1}^{n}q_{t}^{\ell }K_{t}\left( 0\right) \varepsilon
_{t}^{\ast }\text{.}
\end{equation*}
The second term is easily shown to be $O_{p^{\ast }}\left( n^{-1/2}a^{\ell
-1/2}\right) $, whereas the first term is 
\begin{eqnarray}
&&\frac{1}{na}\sum_{t=1}^{n}q_{t}^{\ell }\left\{ K_{t}\left( \widehat{\gamma 
}^{\ast }\right) -K_{t}\left( 0\right) \right\} \varepsilon _{t}^{\ast } 
\mathbf{1}\left( \left\vert q_{t}\right\vert <a^{\zeta }\right)
\label{eps_1Boot} \\
&&+\frac{1}{na}\sum_{t=1}^{n}q_{t}^{\ell }\left\{ K_{t}\left( \widehat{
\gamma }^{\ast }\right) -K_{t}\left( 0\right) \right\} \varepsilon
_{t}^{\ast }\mathbf{1}\left( \left\vert q_{t}\right\vert \geq a^{\zeta
}\right) \text{,}  \notag
\end{eqnarray}
where $\zeta =1-2/\ell $ if $\ell >2$ and $\zeta <1$ if $\ell \leq 2$. The
second term of $\left( \ref{eps_1Boot}\right) $ is 
\begin{equation*}
a^{\ell }\frac{1}{na}\sum_{t=1}^{n}\left( \frac{q_{t}}{a}\right) ^{\ell
}\left\{ K_{t}\left( \widehat{\gamma }^{\ast }\right) -K_{t}\left( 0\right)
\right\} \varepsilon _{t}^{\ast }\mathbf{1}\left( \left\vert
q_{t}\right\vert \geq a^{\zeta }\right) \text{,}
\end{equation*}
whose first absolute bootstrap moment is 
\begin{eqnarray*}
&&a^{\ell }\frac{1}{na}\sum_{t=1}^{n}\left\vert \frac{q_{t}}{a}\right\vert
^{\ell }\left\vert K_{t}\left( \widehat{\gamma }^{\ast }\right) -K_{t}\left(
0\right) \right\vert \mathbf{1}\left( \left\vert q_{t}\right\vert \geq
a^{\zeta }\right) \frac{1}{n}\sum_{s=1}^{n}\left\vert \widehat{\varepsilon }
_{s}\right\vert \\
&&a^{\ell }\frac{1}{na}\sum_{t=1}^{n}\left\vert \frac{q_{t}}{a}\right\vert
^{\ell }\left\vert K_{t}\left( \widehat{\gamma }^{\ast }\right) -K_{t}\left(
0\right) \right\vert \mathbf{1}\left( \left\vert q_{t}\right\vert \geq
a^{\zeta }\right) H_{n}\text{.}
\end{eqnarray*}
Now, proceed as in Lemma \ref{lem:4propositionBoot} to conclude that second
term of $\left( \ref{eps_1Boot}\right) $ is $O_{p^{\ast }}\left( a^{\ell
}\right) $. So, to complete the proof we need to examine the first term of $%
\left( \ref{eps_1Boot}\right) $ which, as we did with the first term of $%
\left( \ref{eps_1}\right) $, is 
\begin{equation*}
\int \phi \left( av\right) \left( e^{iv\widehat{\gamma }^{\ast }}-1\right)
\left\{ \frac{1}{n}\sum_{t=1}^{n}q_{t}^{\ell }\varepsilon _{t}^{\ast
}e^{ivq_{t}}\mathbf{1}\left( \left\vert q_{t}\right\vert <a^{\zeta }\right)
\right\} dv\text{.}
\end{equation*}
But it is clear that the last displayed expression is bounded by 
\begin{eqnarray*}
&&\widehat{\gamma }^{\ast }\int v\left\vert \phi \left( av\right)
\right\vert \left\vert \frac{1}{n}\sum_{t=1}^{n}q_{t}^{\ell }\varepsilon
_{t}^{\ast }e^{ivq_{t}}\mathbf{1}\left( \left\vert q_{t}\right\vert
<a^{\zeta }\right) \right\vert dv=O_{p^{\ast }}\left( a^{\ell \zeta
}n^{-1/2} \widehat{\gamma }^{\ast }\right) \int v\left\vert \phi \left(
av\right) \right\vert dv \\
&=&O_{p^{\ast }}\left( a^{\ell }\left( na^{3}\right) ^{-4/3}n^{1/2}\right)
\end{eqnarray*}
using $\mathbf{K4}$ and that $\zeta =1-2/\ell $ if $\ell \geq 2$ and $\zeta
<1$ when $0\leq \ell <2$, $\widehat{\gamma }^{\ast }=O_{p^{\ast }}\left(
n^{-1/3}\right) $ and that by standard arguments, it yields 
\begin{equation*}
E^{\ast }\left\vert \frac{1}{n}\sum_{t=1}^{n}q_{t}^{\ell }\varepsilon
_{t}^{\ast }e^{ivq_{t}}\mathbf{1}\left( \left\vert q_{t}\right\vert
<a^{\zeta }\right) \right\vert ^{2}=O_{p}\left( a^{2\ell \zeta
}n^{-1}\right) \text{.}
\end{equation*}
This concludes the proof of the lemma.
\end{proof}

\end{document}